\documentclass[12pt]{amsart}
\usepackage{cases}
\usepackage{amsthm}
\usepackage{amsmath}
\usepackage{amscd}
\usepackage{graphicx}
\usepackage[mathscr]{eucal}
\usepackage[colorlinks,linkcolor=blue,citecolor=blue, pdfstartview=FitH]{hyperref}

\input xy
\xyoption{all} \numberwithin{equation}{section}
\numberwithin{figure}{section}
\begin{document}

\title[On the asymptotic expansion of various quantum invariants III]
{On the asymptotic expansion of various quantum invariants III:
the Reshetikhin-Turaev invariants of closed hyperbolic 3-manifolds obtained by doing integral surgery along the twist knot}
\author[Qingtao Chen and Shengmao Zhu]{Qingtao Chen and
Shengmao Zhu}

\address{Department of Pure Mathematics \\
Xi'an Jiaotong-Liverpool University \\
Suzhou Jiangsu \\
China}
\email{Qingtao.Chen@xjtlu.edu.cn,chenqtao@hotmail.com}

\address{Department of Mathematics \\
Zhejiang Normal University  \\
Jinhua Zhejiang,  321004, China }
\email{szhu@zju.edu.cn}

\begin{abstract} 
This is the third article in a series devoted to the study of the asymptotic expansions of various quantum invariants related to the twist knots. In this paper,  by using the saddle point method developed by Ohtsuki and Yokota, we obtain an asymptotic expansion formula for the Reshetikhin-Turaev invariants of closed hyperbolic 3-manifolds obtained by doing integral $q$-surgery along the twist knots $\mathcal{K}_p$ at the root of unity $e^{\frac{4\pi\sqrt{-1}}{r}}$ ($r$ is odd). 
\end{abstract}

\maketitle

\theoremstyle{plain} \newtheorem{thm}{Theorem}[section] \newtheorem{theorem}[%
thm]{Theorem} \newtheorem{lemma}[thm]{Lemma} \newtheorem{corollary}[thm]{%
Corollary} \newtheorem{proposition}[thm]{Proposition} \newtheorem{conjecture}%
[thm]{Conjecture} \theoremstyle{definition}
\newtheorem{remark}[thm]{Remark}
\newtheorem{remarks}[thm]{Remarks} \newtheorem{definition}[thm]{Definition}
\newtheorem{example}[thm]{Example}





\tableofcontents
\newpage

\section{Introduction}
Volume Conjectures relate the exponential growth rate of various quantum invariants of hyperbolic 3-
manifolds at certain root of unity to the hyperbolic volume of the manifolds.  In a series started from \cite{CZ23-1,CZ23-2}, we plan to prove various volume conjectures related to the twist knots.   

In \cite{CZ23-1} and \cite{CZ23-2}, we have proved the Detcherry-Kalfagianni-Yang's  volume conjecture  \cite{DKY18} and  Kashaev-Murakami-Murakami's volume conjecture \cite{Kash97,MuMu01}
for the colored Jones polynomial of twist knot $\mathcal{K}_p$ with $p\geq 6$.

Furthermore,  the volume conjecture   \cite{CY18} relates the Reshetikhin-Turaev and the Turaev-Viro invariants of a hyperbolic 3-manifold at certain roots of unity to the hyperbolic volume of this hyperbolic 3-manifold.  In \cite{Oht18}, Ohtsuki proved the volume conjecture \cite{CY18} for the hyperbolic 3-manifolds obtained by doing integral surgery along the figure-8 knot. Later, in \cite{WongYang20-1}, Wong and Yang proved volume conjecture \cite{CY18} for the hyperbolic 3-manifolds obtained by doing rational surgery along the figure-8 knot.  The goal of this article is to prove the volume conjecture  for  Reshetikhin-Turaev  invariants of the hyperbolic 3-manifolds $M_{p,q}$ obtained by doing integral $q$-surgery along the twist knot $\mathcal{K}_p$.

Let us state the main results of this paper precisely. First, we recall that the volume conjecture for Reshetikhin-Turaev invariant \cite{CY18}: 
\begin{conjecture} \label{conjecture-CY}
    Let $M$ be an oriented closed hyperbolic 3-manifold, for an odd integer $r$, let $RT_r(M)$ be its $r$-th Reshetikhin-Turaev invariant, then at the root of unity $t=e^{\frac{4\pi\sqrt{-1}}{r}}$, 
\begin{align}
\lim_{r\rightarrow \infty} \frac{4\pi}{r}RT_{r}(M)=Vol(M)+\sqrt{-1}CS(M) \mod \pi^2\mathbb{Z}. 
\end{align}
\end{conjecture}

We define the subset $S\subset\mathbb{Z}_+\times \mathbb{Z}_+$ as follows:
\begin{align}
S=\{(p,q)\in \mathbb{Z}_+\times \mathbb{Z}_+ |(p,q) \ \text{satisfies the condition} \ (\ref{formula-pqconditions})\}.    
\end{align}

\begin{equation}  \label{formula-pqconditions}
q\geq \left\{ \begin{aligned}
         & 27  &  \ ( \text{if} \ p=6) \\
         & 19 & \ ( \text{if} \ p=7)  \\ 
         & 17 & \ ( \text{if} \ p=8)  \\
         & 15 & \ ( \text{if} \ p=9, 10)  \\
         & 14 & \ ( \text{if} \ p=11, 12, 13)  \\
         & 13 & \ ( \text{if} \ 14 \leq p\leq 32)  \\ 
         & 12 & \ ( \text{if} \ p\geq 33).  \\ 
                          \end{aligned} \right.
                          \end{equation}

   Let $\hat{V}(p,q,\theta_1,\theta_2,\theta_3)$ be the potential function of Reshetikhin-Turaev invariants of $M_{p,q}$ given by formula (\ref{formula-PotentialV}). 
By Proposition \ref{prop-critical}, there exists a unique critical point $(\theta_1^0,\theta_2^0,\theta_3^0)$ of $\hat{V}(p,q;\theta_1,\theta_2,\theta_3)$. Let $z_i^0=e^{2\pi\sqrt{-1}\theta_i^0}$ for $i=1,2,3$, we put
\begin{align}
  \zeta(p,q)&=
\hat{V}(p,q;\theta_1^0,\theta_2^0,\theta_3^0)\\\nonumber
&=\pi\sqrt{-1}\left(\frac{3q}{2}+(\frac{q}{2}-1)(\theta_1^0)^2-\theta_1^0+(2p+1)(\theta_3^0)^2-(2p+3)\theta_3^0-2\theta_2^0\right)\\\nonumber
&+\frac{1}{2\pi\sqrt{-1}}\left(\frac{\pi^2}{6}+\text{Li}_2(z_2^0z_3^0)+\text{Li}_2\left(\frac{z_2^0}{z_3^0}\right)-\text{Li}_2(z_2^0z_1^0)-\text{Li}_2(z_2^0)-\text{Li}_2\left(\frac{z_2^0}{z_1^0}\right)\right).
\end{align}
and
\begin{align} \label{formula-omegapq}
    \omega(p,q)=\frac{z_2^0(z_3^0-(z_3^0)^{-1})(\sqrt{z_1^0}-\frac{1}{\sqrt{z_1^0}})}{\sqrt{(1-z_2^0z_3^0)(1-z_2^0(z_{3}^0)^{-1})}\sqrt{H(p,q;z_1^0,z_2^0,z_3^0)}},
\end{align}
where
\begin{align}
  &H(p,q;z_1^0,z_2^0,z_3^0)=4\frac{z_2^0}{z_1^0-z_2^0}\cdot \frac{z_1^0z_2^0}{1-z_1^0z_2^0}\left(\frac{z_2^0z_3^0}{1-z_2^0z_3^0}+\frac{z_2^0}{z_3^0-z_2^0}\right)\\\nonumber
  &-4\left(\frac{z_2^0}{z_1^0-z_2^0}+\frac{z_1^0z_2^0}{1-z_1^0z_2^0}\right)\frac{z_2^0z_3^0}{1-z_2^0z_3^0}\cdot\frac{z_2^0}{z_3^0-z_2^0}+\left(\frac{z_2^0}{z_1^0-z_2^0}+\frac{z_1^0z_2^0}{1-z_1^0z_2^0}\right)\frac{z_2^0}{1-z_2^0}\left(\frac{z_2^0z_3^0}{1-z_2^0z_3^0}+\frac{z_2^0}{z_3^0-z_2^0}\right)\\\nonumber
  &+(8p+4)\frac{z_2^0}{z_1^0-z_2^0}\cdot \frac{z_1^0z_2^0}{1-z_1^0z_2^0}+(2p+1)\left(\frac{z_2^0}{z_1^0-z_2^0}+\frac{z_1^0z_2^0}{1-z_1^0z_2^0}\right)\frac{z_2^0}{1-z_2^0}\\\nonumber
  &-(2p+\frac{q}{2})\left(\frac{z_2^0}{z_1^0-z_2^0}+\frac{z_1^0z_2^0}{1-z_1^0z_2^0}\right)\left(\frac{z_2^0z_3^0}{1-z_2^0z_3^0}+\frac{z_2^0}{z_3^0-z_2^0}\right)-(\frac{q}{2}-1)\frac{z_2^0}{1-z_2^0}\left(\frac{z_2^0z_3^0}{1-z_2^0z_3^0}+\frac{z_2^0}{z_3^0-z_2^0}\right)\\\nonumber
  &+4(\frac{q}{2}-1)\frac{z_2^0z_3^0}{1-z_2^0z_3^0}\cdot\frac{z_2^0}{z_3^0-z_2^0}-(2p+1)(\frac{q}{2}-1)\left(\frac{z_2^0}{z_1^0-z_2^0}+\frac{z_1^0z_2^0}{1-z_1^0z_2^0}+\frac{z_2^0}{1-z_{2}^0}-\frac{z_2^0z_3^0}{1-z_2^0z_3^0}-\frac{z_2^0}{z_3^0-z_2^0}\right).
\end{align}

Then, we have
\begin{theorem}  \label{theorem-main}
  For $(p,q)\in S$, the asymptotic expansion of Reshetikhin-Turaev $RT_{r}(M_{p,q})$ is given by the following form
    \begin{align}
        RT_{r}(M_{p,q})&=(-1)^{p+1}\sqrt{-1}e^{\sigma((\mathcal{K}_{p})_q)\left(\frac{3}{r}+\frac{r+1}{4}\right)\pi\sqrt{-1}}\omega(p,q)e^{(N+\frac{1}{2})\zeta(p,q)}\\\nonumber
&\cdot\left(1+\sum_{i=1}^d\kappa_i(p,q)\left(\frac{4\pi\sqrt{-1}}{r}\right)^i+O\left(\frac{1}{r^{d+1}}\right)\right),
    \end{align}
    for $d\geq 1$, where $\omega(p,q)$ is given by formula (\ref{formula-omegapq}) and $\kappa_i(p,q)$ are constants determined by the 3-manifold $M_{p,q}$.
\end{theorem}
By Theorem  \ref{theorem-critical=volume}, we know that
\begin{align}
     2\pi \zeta(p,q)= Vol(M_{p,q})+\sqrt{-1}CS(M_{p,q}) \mod \pi^2\sqrt{-1}\mathbb{Z}.
\end{align}
It implies that
\begin{corollary}
    For $(p,q)\in S$, we have
    \begin{align}
        \lim_{r\rightarrow \infty}\frac{4\pi}{r}\log RT_r(M_{p,q};e^{\frac{4\pi \sqrt{-1}}{r}})=Vol(M_{p,q})+\sqrt{-1}CS(M_{p,q}) \mod \pi^2\sqrt{-1}\mathbb{Z}.
    \end{align}
\end{corollary}
Hence we proved Conjecture \ref{conjecture-CY} for $M_{p,q}$.

The rest of this article is organized as follows. In Section \ref{Section-2RT}, we fix the notations and
review the related materials that will be used in this paper. In Section \ref{Section-Potentialfunction}, we compute
the potential function for the Reshetikhin-Turaev invariant for the 3-manifold obtained by doing integral surgery along the twist knot $\mathcal{K}_p$. 
In Section \ref{Section-Poissonsummation}, we express the Reshetikhin-Turaev invariant as a summation of Fourier coefficients with the help of Poisson summation formula. 
In Section \ref{Section-asympticexpansion}, 
we first show that infinite terms of the Fourier coefficients can be neglected. Then we estimate the remained term of Fourier coefficients by using the saddle point method, we obtain that only four main Fourier coefficients will contribute to the final form of the  asymptotic expansion.  Hence we finish the proof of the main Theorem  \ref{theorem-main}. 
The final Appendix  Section \ref{Section-Appendices} is devoted to the proof of several results which will be used in previous sections.

\textbf{Acknowledgements.} 

The first author would like to thank Nicolai Reshetikhin, Kefeng Liu and Weiping Zhang for bringing him to this area and a lot of discussions during his career, thank Francis Bonahon,   Giovanni Felder and Shing-Tung Yau for their continuous encouragement, support and discussions, and thank Jun Murakami and Tomotada Ohtsuki for their helpful discussions and support. He also want to thank Jørgen Ellegaard Andersen, Sergei Gukov, Thang Le, Gregor Masbaum,  Rinat Kashaev, Vladimir Turaev, Tian Yang and Hiraku Nakajima for their support, discussions and interests, and thank Yunlong Yao who built him solid analysis foundation twenty years ago. The second author would like to thank Kefeng Liu and Hao Xu for bringing him to this area when he was a graduate student at CMS of Zhejiang University, and for their constant encouragement and helpful discussions since then. Both of us thanks Ruifeng Qiu for his interests and supports.

\section{Preliminaries} \label{Section-2RT}

\subsection{Defintion of the Reshetikhin-Turaev invariants}
We use the skein theory approach  of Reshetikhin-Turaev invariants \cite{BHMV92,Lic93} 
following the concise illustration given in \cite{WongYang23}.  We focus on the $SO(3)$ TQFT theory and consider the root of unity $e^{\frac{2\pi\sqrt{-1}}{r}}$, where $r$ is an odd number, we write $r=2N+1$ with $N\geq 1$. 

\begin{definition}
    Let $F$ be an oriented surface, given $A=e^{\frac{\pi\sqrt{-1}}{r}}$.  The Kauffman bracket skein module of $F$, denoted by $K(F)$, is a $\mathbb{C}$-module generated by the isotopic classes of link diagrams in $F$ modulo the submodule generated by the following two  relations:
    
(1) Kauffman bracket skein relation:
\begin{figure}[!htb] 
\begin{align}
\langle\raisebox{-15pt}{
\includegraphics[width=50 pt]{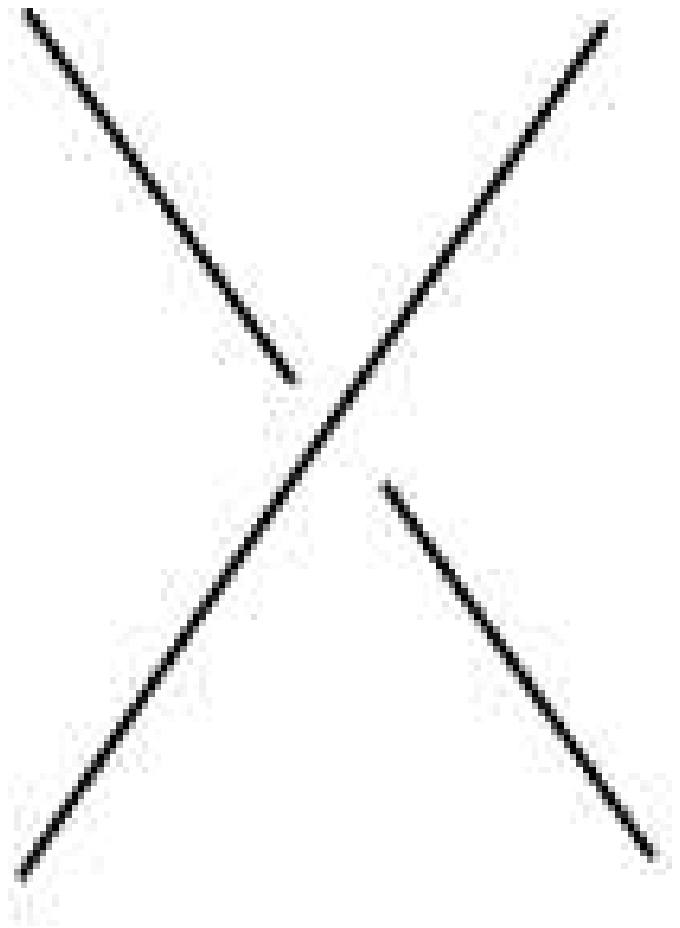}}\rangle=A\langle\raisebox{-15pt}{ \includegraphics[width=50 pt]{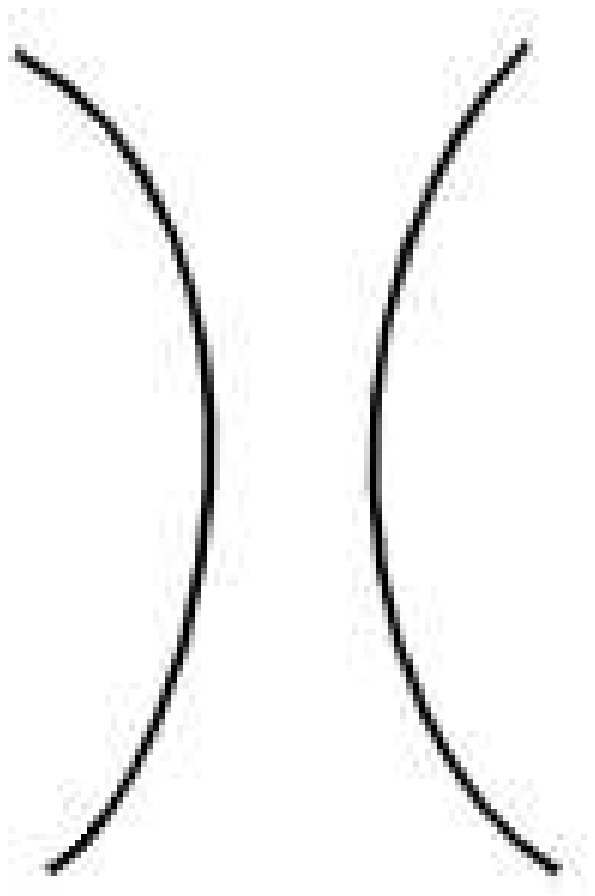}}\rangle+A^{-1}\langle\raisebox{-15pt}{ \includegraphics[width=50 pt]{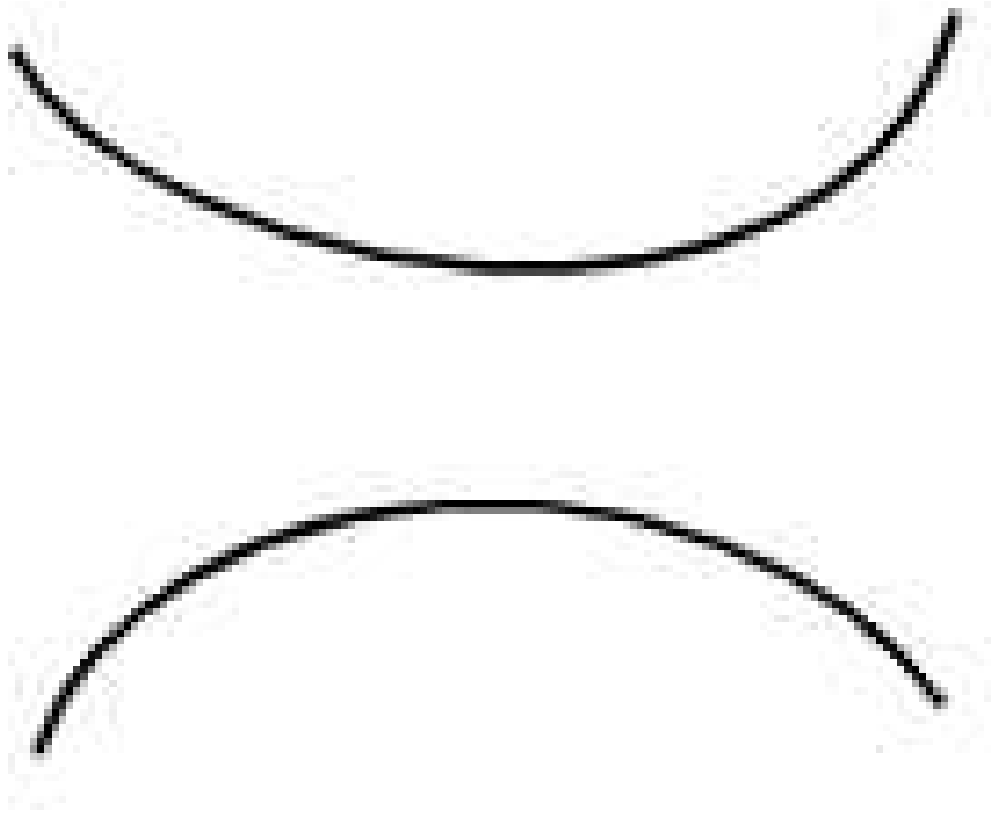}}\rangle  
\end{align}
\end{figure}

(2) Framing relation:
\begin{figure}[!htb] 
$
\langle D\cup \raisebox{-15pt}{
\includegraphics[width=50 pt]{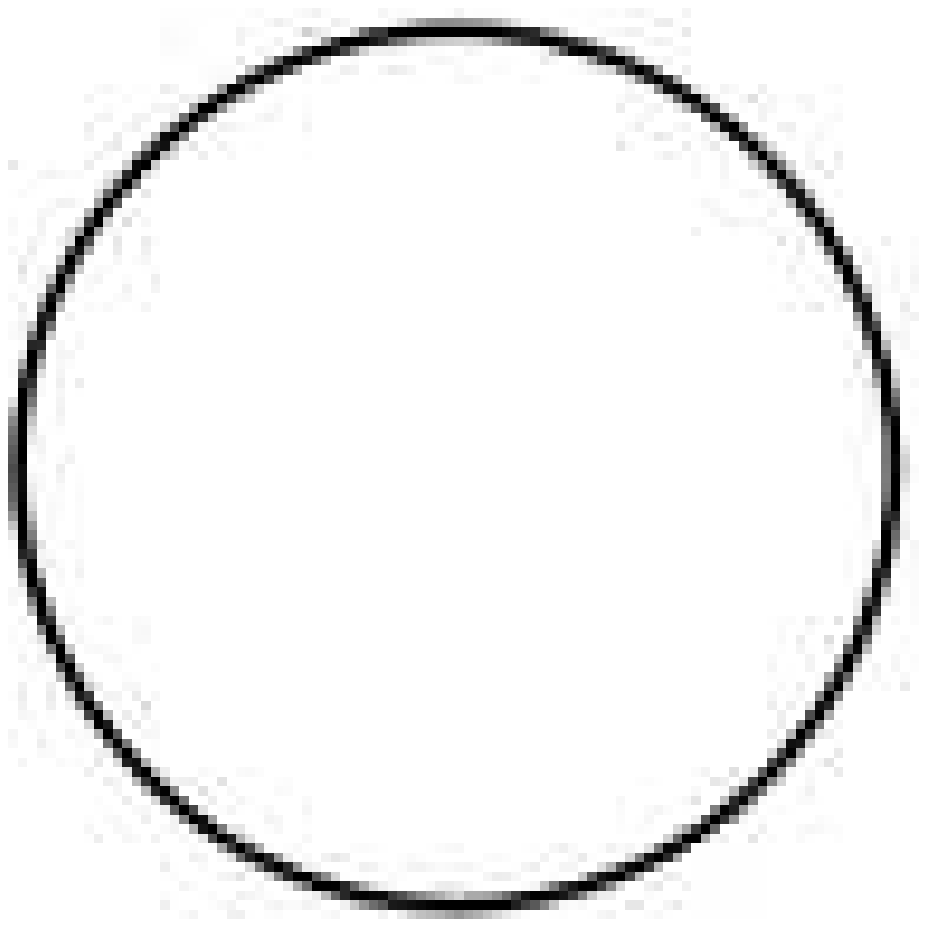}}\rangle=(-A^2-A^{-2})\langle D\rangle.
$
\end{figure}

\end{definition}

\begin{example}
When $F=\mathbb{R}^2$, there is a canonical isomorphism 
\begin{align}
    \langle\ \rangle: K(\mathbb{R}^2)\rightarrow \mathbb{C}
\end{align}
defined by sending the empty link to $1$. The image $\langle\mathcal{L}\rangle$ of the framed link $\mathcal{L}$ is called the Kauffman bracket of $\mathcal{L}$. 
\end{example}

\begin{example}
    $F=S^1\times [0,1]$, we let $\mathcal{B}=K(S^1\times [0,1])$. Actually, $\mathcal{B}$ is an algebra, which is called the Kauffman bracket skein algebra of $S^1\times I$. The algebraic structure (i.e. product) of $\mathcal{B}$ is induced by the gluing of two annulus. 
\end{example}
    
    For any link diagram $D$ in $\mathbb{R}^2$ with $k$ ordered components and $b_1,..,b_k\in \mathcal{B}$, let
    \begin{align}
        \langle b_1,..,b_k\rangle_D
    \end{align}
be the complex number obtained by cabling $b_1,...,b_k$ along the components of $D$ then taking the Kauffman bracket $\langle \ \rangle$.

Note that the skein algebra $\mathcal{B}$ is commutative and the empty diagram is the unit, hence denoted by $1$.
Let 
$z\in \mathcal{B}$ be the core curve of $S^1\times I$ as illustrated in  Figure \ref{figure-z}. 

 \begin{figure}[!htb] 
 \label{figure-z}
 \begin{align*} 
\raisebox{-15pt}{
\includegraphics[width=80 pt]{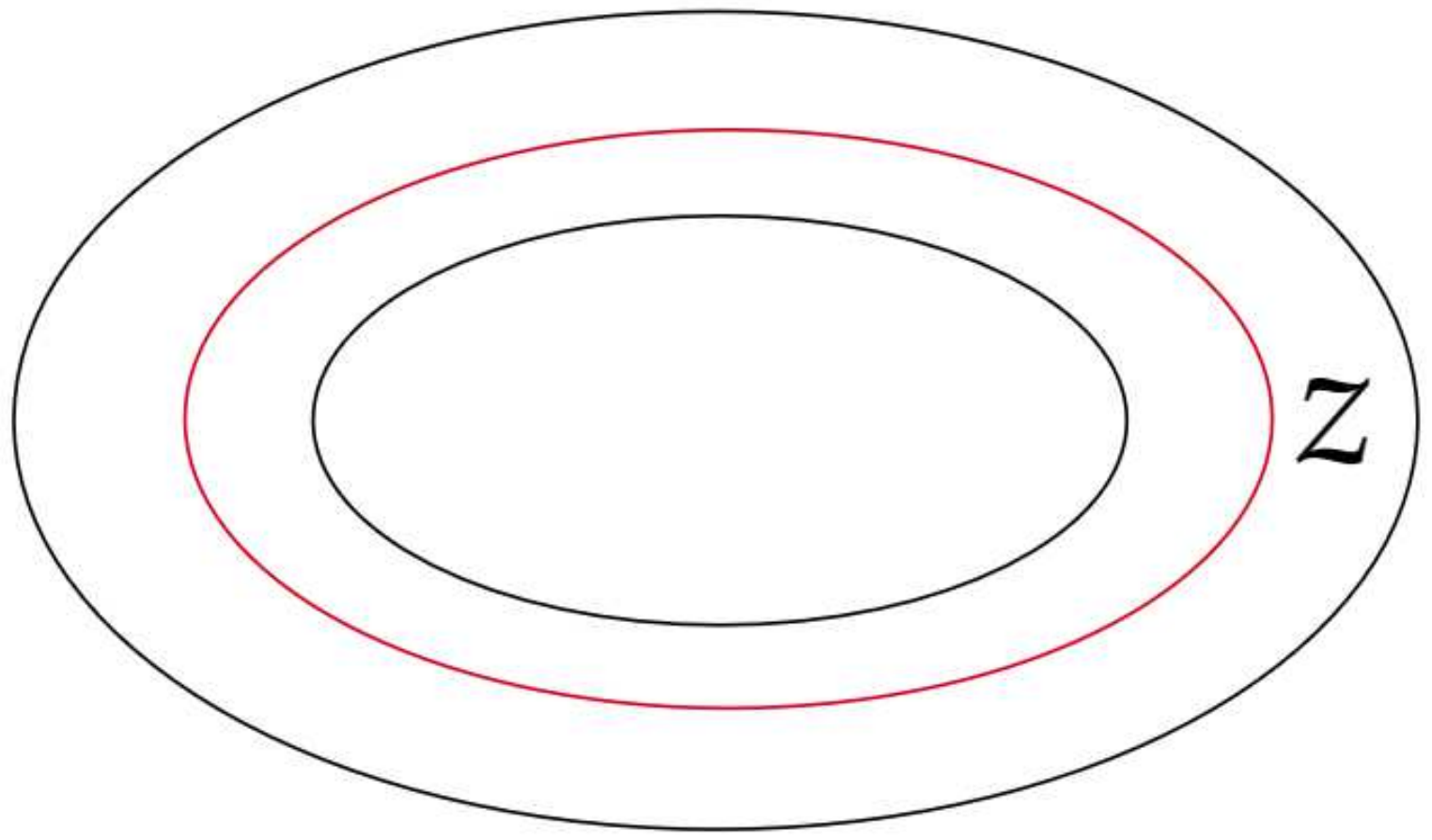}}
\end{align*}
\caption{The core curve $z$}
\end{figure}   
Then  $z^n$ means $n$-parallel copies of $z$.
Moreover, we have $\mathcal{B}=\mathbb{C}[z].$
We define the skein elements $e_n\in \mathcal{B}$ recursively by $e_0=1$, $e_1=z$ and 
$e_n=ze_{n-1}-e_{n-2}$
for $n\geq 2$. Let
\begin{align}
    \omega_r=\mu_r\sum_{n=0}^{r-2}(-1)^n[n+1]e_n,
\end{align}
the Kirby colored $\Omega_r\in \mathcal{B}$ is defined by 
\begin{align}
    \Omega_r=\mu_r\omega_r,
\end{align}
where $\mu_r=\frac{\sin\frac{2\pi}{r}}{\sqrt{r}}$, and $[n]$ is the quantum integer defined by 
\begin{align}
   [n]=\frac{t^{\frac{n}{2}}-t^{-\frac{n}{2}}}{t^{\frac{1}{2}}-t^{-\frac{1}{2}}}. 
\end{align}
Note that we fix the convention that $t=A^4=e^{\frac{4\pi\sqrt{-1}}{r}}$ throughout this paper.

 Suppose $M$ is obtained from $S^3$ by doing a surgery along a framed link $\mathcal{L}$, let $D_\mathcal{L}$ be the standard diagram of $\mathcal{L}$ which implies that the blackboard framing of $D_{\mathcal{L}}$ coincides with framing of $\mathcal{L}$. The signature of the linking matrix of $\mathcal{L}$ is denoted by $\sigma(\mathcal{L})$.  Then the $r$-th Reshetikhin-Turaev invariant of $M$  is defined as 
\begin{align}  \label{formula-RT}
RT_r(M)=\mu_r\langle\Omega_r,...,\Omega_r\rangle_{D_{\mathcal{L}}}\langle\Omega_r\rangle_{U_+}^{-\sigma(\mathcal{L})}.     \end{align}

\subsection{Dilogarithm and Lobachevsky functions}
Let $\log: \mathbb{C}\setminus (-\infty,0]\rightarrow \mathbb{C}$ be the standard logarithm function defined by 
\begin{align}
    \log z=\log |z|+\sqrt{-1}\arg z
\end{align}
with $-\pi <\arg z<\pi$. 

The dilogarithm function $\text{Li}_2: \mathbb{C}\setminus (1,\infty)\rightarrow \mathbb{C}$ is defined by 
\begin{align}
    \text{Li}_2(z)=-\int_0^{z}\frac{\log(1-x)}{x}dx
\end{align}
where the integral is along any path in $\mathbb{C}\setminus (1,\infty)$ connecting $0$ and $z$, which is holomorphic in $\mathbb{C}\setminus [1,\infty)$ and continuous in $\mathbb{C}\setminus (1,\infty)$. 

The dilogarithm function satisfies the following properties 
\begin{align}
    \text{Li}_2\left(\frac{1}{z}\right)=-\text{Li}_2(z)-\frac{\pi^2}{6}-\frac{1}{2}(\log(-z) )^2.
\end{align}
In the unit disk $\{z\in \mathbb{C}| |z|<1\}$,  $\text{Li}_2(z)=\sum_{n=1}^{\infty}\frac{z^n}{n^2}$, and on the unit circle 
\begin{align}
 \{z=e^{2\pi \sqrt{-1}\theta}|0 \leq \theta\leq 1\},    
\end{align}
we have
\begin{align}
    \text{Li}_2(e^{2\pi\sqrt{-1} \theta})=\frac{\pi^2}{6}+\pi^2\theta(\theta-1)+2\pi \sqrt{-1}\Lambda(\theta)
\end{align}
where 
\begin{align} \label{formula-Lambda(t)}
\Lambda(\theta)=\text{Re}\left(\frac{\text{Li}_2(e^{2\pi \sqrt{-1}\theta})}{2\pi \sqrt{-1}}\right)=-\int_{0}^{\theta}\log|2 \sin \pi \theta|d \theta 
\end{align}
for $\theta\in \mathbb{R}$. The function $\Lambda(\theta)$ is an odd function which has period $1$ and satisfies 
$
\Lambda(1)=\Lambda(\frac{1}{2})=0.
$

Furthermore, we have the follow estimation for the function $$\text{Re}\left(\frac{1}{2\pi\sqrt{-1}}\text{Li}_2\left(e^{2\pi\sqrt{-1}(\theta+X\sqrt{-1})}\right)\right)$$ with $\theta,X\in \mathbb{R}$.   
\begin{lemma} (see Lemma 2.2 in \cite{OhtYok18}) \label{lemma-Li2}
    Let $\theta$ be a real number with $0<\theta<1$. Then there exists a constant $C>0$ such that 
\begin{align}
    \left\{ \begin{aligned}
         &0  &  \ (\text{if} \ X\geq 0) \\
         &2\pi \left(\theta-\frac{1}{2}\right)X & \ (\text{if} \ X<0)
                          \end{aligned} \right.-C<\text{Re}\left(\frac{1}{2\pi\sqrt{-1}}\text{Li}_2\left(e^{2\pi\sqrt{-1}(\theta+X\sqrt{-1})}\right)\right)
\end{align}
\begin{align*}
    <\left\{ \begin{aligned}
         &0  &  \ (\text{if} \ X\geq 0) \\
         &2\pi \left(\theta-\frac{1}{2}\right)X & \ (\text{if} \ X<0)
                          \end{aligned}\right.+C.
\end{align*}.
\end{lemma}

\subsection{Quantum dilogrithm functions}
For a positive integer $N$, we set $t=e^{\frac{2\pi\sqrt{-1}}{N+\frac{1}{2}}}$. We introduce the holomorphic function $\varphi_N(\theta)$ for $\{\theta\in
\mathbb{C}| 0< \text{Re}(\theta) < 1\}$, by the following integral
\begin{align}
\varphi_N(\theta)=\int_{-\infty}^{+\infty}\frac{e^{(2\theta-1)x}dx}{4x \sinh x
\sinh\frac{x}{N+\frac{1}{2}}}.
\end{align}
Noting that this integrand has poles at $n\pi \sqrt{-1} (n\in
\mathbb{Z})$, where, to avoid the poles at $0$, we choose the
following contour of the integral
\begin{align}
\gamma=(-\infty,-1]\cup \{z\in \mathbb{C}||z|=1, \text{Im} z\geq 0\}
\cup [1,\infty).
\end{align}

\begin{lemma}  \label{lemma-varphixi}
The function $\varphi_N(\theta)$ satisfies 
\begin{align}
    (t)_n&=\exp \left(\varphi_N\left(\frac{1}{2N+1}\right)-\varphi_N\left(\frac{2n+1}{2N+1}\right)\right)   \  \   \left(0\leq n\leq N\right), \\
    (t)_n&=\exp \left(\varphi_N\left(\frac{1}{2N+1}\right)-\varphi_N\left(\frac{2n+1}{2N+1}-1\right)+\log 2\right)   \  \   \left(N< n\leq 2N\right).
\end{align}
\end{lemma}

\begin{lemma} \label{lemma-varphixi2}
    We have the following identities:
\begin{align}
    \varphi_N(\theta)+\varphi_N(1-\theta)&=2\pi \sqrt{-1}\left(-\frac{2N+1}{4}\left(\theta^2-\theta+\frac{1}{6}\right)+\frac{1}{12(2N+1)}\right),\\ 
    \varphi_N\left(\frac{1}{2N+1}\right)&=\frac{2N+1}{4\pi\sqrt{-1}}\frac{\pi^2}{6}+\frac{1}{2}\log \left(\frac{2N+1}{2}\right)+\frac{\pi \sqrt{-1}}{4}-\frac{\pi \sqrt{-1}}{6(2N+1)},\\
    \varphi_N\left(1-\frac{1}{2N+1}\right)&=\frac{2N+1}{4\pi\sqrt{-1}}\frac{\pi^2}{6}-\frac{1}{2}\log \left(\frac{2N+1}{2}\right)+\frac{\pi \sqrt{-1}}{4}-\frac{\pi \sqrt{-1}}{6(2N+1)}.
\end{align}
\end{lemma}
The function $\varphi_N (\theta)$ is closely related to the dilogarithm function as follows.
\begin{lemma} \label{lemma-varphixi3}
    (1) For every $\theta$ with $0<\text{Re}(\theta)<1$, 
    \begin{align}
        \varphi_N(\theta)=\frac{N+\frac{1}{2}}{2\pi \sqrt{-1}}\text{Li}_2(e^{2\pi\sqrt{-1}\theta})
 -\frac{\pi \sqrt{-1}e^{2\pi\sqrt{-1}\theta}}{6(1-e^{2\pi\sqrt{-1}\theta})}\frac{1}{2N+1}+O\left(\frac{1}{(N+\frac{1}{2})^3}\right).
    \end{align}
    (2) For every $\theta$ with $0<\text{Re}(\theta)<1$, 
    \begin{align}
        \varphi_N'(\theta)=-\frac{2N+1}{2}\log(1-e^{2\pi\sqrt{-1}\theta})+O\left(\frac{1}{N+\frac{1}{2}}\right)
    \end{align}
    (3) As $N\rightarrow \infty$, $\frac{1}{N+\frac{1}{2}}\varphi_N(\theta)$ uniformly converges to $\frac{1}{2\pi\sqrt{-1}}\text{Li}_2(e^{2\pi\sqrt{-1}\theta})$ and $\frac{1}{N+\frac{1}{2}}\varphi'_N(\theta)$ uniformly converges to $-\log(1-e^{2\pi\sqrt{-1}\theta})$ on any compact subset of $\{\theta\in \mathbb{C}|0<\text{Re}(\theta)<1\}$. 
\end{lemma}
See the literature, such as \cite{Oht16,CJ17,WongYang20-1} for the proof of Lemma \ref{lemma-varphixi}, \ref{lemma-varphixi2}, \ref{lemma-varphixi3}.

 \subsection{Saddle point method}
We need to use the following version of saddle point method as illustrated in \cite{Oht18}.
\begin{proposition}[\cite{Oht18}, Proposition 3.1]  \label{proposition-saddlemethod}
   Let $A$ be a non-singular symmetric complex $3\times 3$ matrix, and let $\Psi(z_1,z_2,z_3)$ and $r(z_1,z_2,z_3)$ be holomorphic functions of the forms, 
   \begin{align}
    \Psi(z_1,z_2,z_3)&=\mathbf{z}^{T}A\mathbf{z}+r(z_1,z_2,z_3), \\\nonumber
    r(z_1,z_2,z_3)&=\sum_{i,j,k}b_{ijk}z_iz_jz_k+\sum_{i,j,k,l}c_{ijkl}z_iz_jz_kz_l+\cdots
   \end{align}
   defined in a neighborhood of $\mathbf{0}\in \mathbb{C}^3$. The restriction of the domain 
   \begin{align} \label{formula-domain0}
       \{(z_1,z_2,z_3)\in \mathbb{C}^3| \text{Re}\Psi(z_1,z_2,z_3)<0\}  
   \end{align}
   to a neighborhood of $\mathbf{0}\in \mathbb{C}^3$ is homotopy equivalent to $S^2$. Let $D$ be an oriented disk embeded in $\mathbb{C}^3$ such that $\partial D$ is included in the domain (\ref{formula-domain0}) whose inclusion is homotopic to a homotopy equivalence to the above $S^2$ in the domain (\ref{formula-domain0}). Then we have the following asymptotic expansion
\begin{align}
    \int_{D}e^{N\psi(z_1,z_2,z_3)}dz_1dz_2dz_3=\frac{\pi^{\frac{3}{2}}}{N^{\frac{3}{2}}\sqrt{\det(-A)}}\left(1+\sum_{i=1}^d\frac{\lambda_i}{N^i}+O(\frac{1}{N^{d+1}})\right),
\end{align}
   for any $d$, where we choose the sign of $\sqrt{\det{(-A)}}$ as explained in Proposition \cite{Oht16}, and $\lambda_i$'s are constants presented by using coefficients of the expansion $\Psi(z_1,z_2,z_3)$, such presentations are obtained by formally expanding the following formula, 
\begin{align}
    &1+\sum_{i=1}^{\infty}\frac{\lambda_i}{N^i}\\\nonumber
    &=\exp\left(Nr\left(\frac{\partial }{\partial w_1},\frac{\partial }{\partial w_2},\frac{\partial }{\partial w_3}\right)\right)\exp\left(-\frac{1}{4N}(w_1,w_2,w_3)A^{-1}\begin{pmatrix}
      w_1 \\ w_2 \\ w_3  
    \end{pmatrix}\right)|_{w_1=w_2=w_3=0}.
\end{align}
\end{proposition}
For a proof of the Proposition \ref{proposition-saddlemethod},  see \cite{Oht16}. 
\begin{remark}[\cite{Oht18}, Remark 3.2]  \label{remark-saddle}
    As mentioned in Remark 3.6 of \cite{Oht16}, we can extend Proposition \ref{proposition-saddlemethod} to the case where $\Psi(z_1,z_2,z_3)$ depends on $N$ in such a way that $\Psi(z_1,z_2,z_3)$ is of the form 
    \begin{align}
        \Psi(z_1,z_2,z_3)=\Psi_0(z_1,z_2,z_3)+\Psi_1(z_1,z_2,z_3)\frac{1}{N}+R(z_1,z_2,z_3)\frac{1}{N^2}. 
    \end{align}
    where $\Psi_i(z_1,z_2,z_3)$'s are holomorphic functions independent of $N$, and we assume that $\Psi_0(z_1,z_2,z_3)$ satisfies the assumption of the Proposition and $|R(z_1,z_2,z_3)|$ is bounded by a constant which is independent of $N$.  
\end{remark}

\section{Computations of the potential function}  \label{Section-Potentialfunction}

Let $M_{p,q}$  be the 3-manifold obtained by the doing $q$-surgery along the twist knot $\mathcal{K}_p$, (see Figure \ref{figure-twistknot}) where the index $p$ represents $2p$ crossings (half-twists). For example, $\mathcal{K}_{-1}=4_1$, $\mathcal{K}_1=3_1$, $\mathcal{K}_2=5_2$.

\begin{figure}[!htb] \label{figure-twistknot} 
\raisebox{-15pt}{
\includegraphics[width=150 pt]{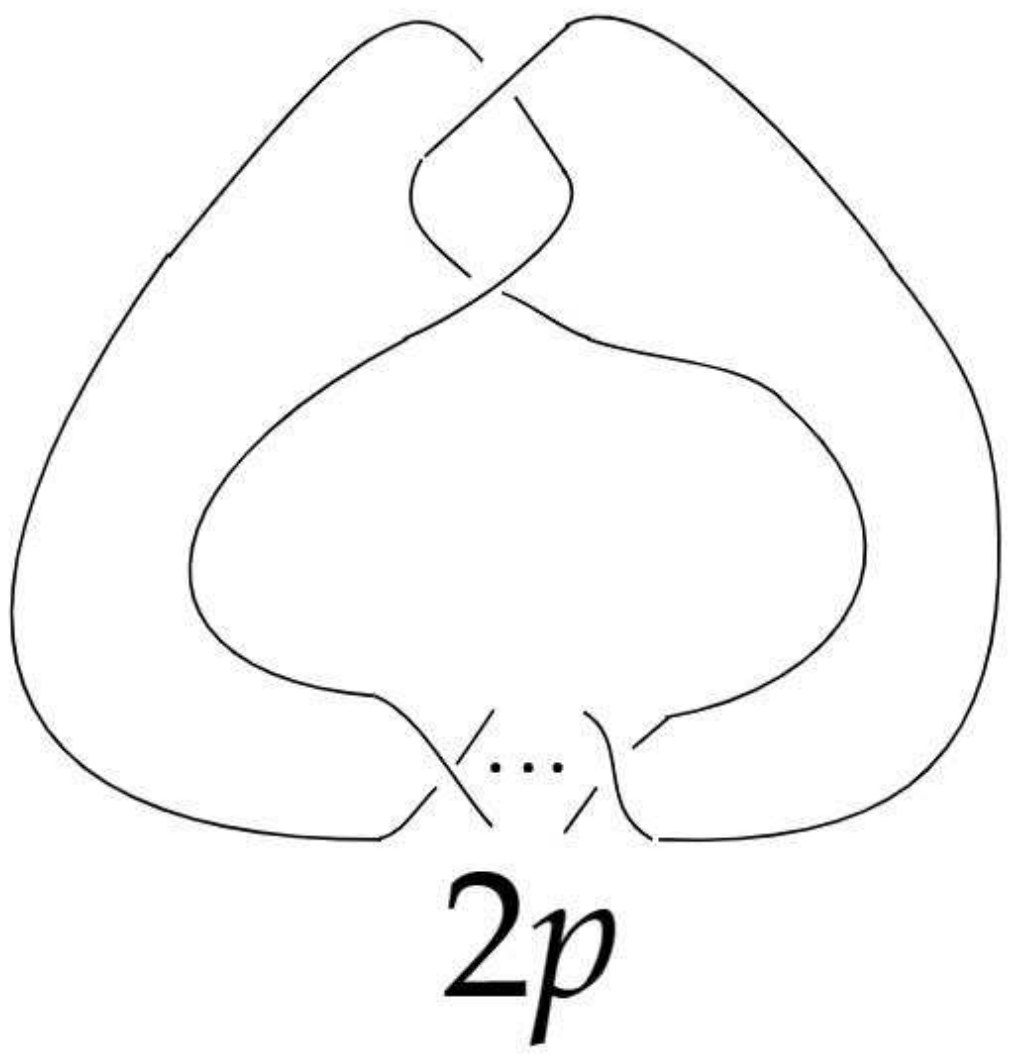}}.
\caption{Twist knot $\mathcal{K}_p$}
\end{figure}

Denoted by $(\mathcal{K}_p)_q$ the twist knot $\mathcal{K}_p$ with framing $q$, and $D_{(\mathcal{K}_p)_q}$ be its link diagram. 
By formula (\ref{formula-RT}), the $r$-th Reshetikhin-Turaev invariant of $M_{p,q}$ is given by
\begin{align}
RT_r(M)=\mu_r\langle\Omega_r\rangle_{D_{(\mathcal{K}_p)_q}}\langle\Omega_r\rangle_{U_+}^{-\sigma((\mathcal{K}_p)_q)}.    
\end{align}
A direct computation shows that
$\langle\Omega_r\rangle_{U_+}=e^{\left(-\frac{3}{r}-\frac{r+1}{4}\right)\pi \sqrt{-1}}$. 
Let $\sigma_{p,q}=\sigma((\mathcal{K}_p)_q)$ and 
\begin{align}
    \kappa_r=\mu_r^2\langle\Omega_r\rangle_{U_+}^{-\sigma_{p,q}}=\left(\frac{\sin\frac{2\pi}{r}}{\sqrt{r}}\right)^2e^{\sigma_{p,q}\left(\frac{3}{r}+\frac{r+1}{4}\right)\pi\sqrt{-1}}.
\end{align}

Then  we have 
\begin{align}
    RT_r(M_{p,q})&=\kappa_r\langle \omega_r\rangle_{D_{(\mathcal{K}_p)_q}}=\kappa_r\sum_{m=0}^{r-2}[m+1](-1)^{qm+1}t^{q\frac{m(m+2)}{4}}\langle e_{m}\rangle_{D_{(\mathcal{K}_p)_0}}.
\end{align}

Recall the definition of  colored Jones polynomial \cite{CZ23-1}, for a knot $\mathcal{K}$ with zero-framing, we have
\begin{align}
[m+1]J_{m+1}(\mathcal{K};t)=\langle (-1)^me_m\rangle_{D_{\mathcal{K}}}, 
\end{align}
then we get
\begin{align}
   RT_r(M)&=\kappa_r \sum_{m=0}^{r-2}[m+1]^2(-1)^{qm}t^{q\frac{m(m+2)}{4}}J_{m+1}(\mathcal{K}_p;t)\\\nonumber
   &=\frac{\kappa_r}{\sin^2 \left(\frac{\pi}{N+\frac{1}{2}}\right)}\sum_{m=0}^{2N-1}(-1)^{qm}\sin^2\left(\frac{(m+1)\pi}{N+\frac{1}{2}}\right)e^{\frac{qm(m+2)\pi\sqrt{-1}}{2(N+\frac{1}{2})}}J_{m+1}(\mathcal{K}_p;t)
\end{align}

And at the root of unity $t=e^{\frac{2\pi\sqrt{-1}}{N+\frac{1}{2}}}$, we have  
\begin{align}  \label{formula-sumkl}
 J_{m}(\mathcal{K}_p;t)&=\sum_{k=0}^{m-1}\sum_{l=0}^k\frac{(-1)^{k+l}\sin \frac{\pi(2l+1)}{N+\frac{1}{2}}}{\sin \frac{m\pi }{N+\frac{1}{2}}}t^{(p+\frac{1}{2})l(l+1)-m(k+\frac{1}{2})+\frac{1}{2}k^2+\frac{3}{2}k+\frac{1}{2}}\\\nonumber
 &\cdot\frac{(t)_k(t)_{m+k}}{(t)_{k+l+1}(t)_{k-l}(t)_{m-k-1}}. 
\end{align}

Therefore, let $a=N-1-m$, we obtain
\begin{align}
RT_{r}(M_{p,q})&=\kappa_r'\sum_{a=-N}^{N-1}\sum_{k=0}^{N-a-1}\sum_{l=0}^k  (-1)^{q(N-1-a)+k+l}\sin\left(\frac{(a+\frac{1}{2})\pi}{N+\frac{1}{2}}\right)\sin\left(\frac{2\pi(l+\frac{1}{2})}{N+\frac{1}{2}}\right)\\\nonumber
&\cdot t^{q\frac{(N-1-a)(N+1-a)}{4}+(p+\frac{1}{2})l(l+1)-(N-a)(k+\frac{1}{2})+\frac{1}{2}k^2+\frac{3}{2}k+\frac{1}{2}}  \frac{(t)_k(t)_{N-a+k}}{(t)_{k+l+1}(t)_{k-l}(t)_{N-a-k-1}},
\end{align}
where 
\begin{align} \label{formula-kappa'r}
\kappa'_r=\frac{\kappa_r}{\sin^2 \left(\frac{\pi}{N+\frac{1}{2}}\right)}.    
\end{align}

Let 
\begin{align}
\theta_1=\frac{a+\frac{1}{2}}{N+\frac{1}{2}}, \ 
\theta_2=\frac{k+\frac{1}{2}}{N+\frac{1}{2}},  \ \theta_3=\frac{l+\frac{1}{2}}{N+\frac{1}{2}}.       
\end{align}
Then we have
\begin{align}
    &t^{q\frac{(N-1-a)(N+1-a)}{4}+(p+\frac{1}{2})l(l+1)-(N-a)(k+\frac{1}{2})+\frac{1}{2}k^2+\frac{3}{2}k+\frac{1}{2}}\\\nonumber
    &=e^{(N+\frac{1}{2})2\pi\sqrt{-1}\left(\frac{q}{4}(1+\theta_1^2-2\theta_1-\frac{1}{(N+\frac{1}{2})^2})+(p+\frac{1}{2})(\theta_3^2-\frac{1}{4(N+\frac{1}{2})^2})-\theta_2+\theta_2\theta_1+\frac{1}{2}\theta_2^2+\frac{\theta_2}{(N+\frac{1}{2})}-\frac{1}{8(N+\frac{1}{2})^2}\right)}.
\end{align}
Moreover, by using Lemma \ref{lemma-varphixi},  we have
\begin{align}
    &\frac{(t)_k(t)_{N-a+k}}{(t)_{k+l+1}(t)_{k-l}(t)_{N-a-k-1}}\\\nonumber
    &=\exp\left(\varphi_N\left(\frac{(k+l+1)+\frac{1}{2}}{N+\frac{1}{2}}\right)+\varphi_N\left(\frac{(k-l)+\frac{1}{2}}{N+\frac{1}{2}}\right)-\varphi_N\left(\frac{\frac{1}{2}}{N+\frac{1}{2}}\right)\right.\\\nonumber
    &\left.-\varphi_N\left(\frac{k+\frac{1}{2}}{N+\frac{1}{2}}\right)-\varphi_N\left(\frac{(N-a+k)+\frac{1}{2}}{N+\frac{1}{2}}\right)+\varphi_N\left(\frac{(N-a-k-1)+\frac{1}{2}}{N+\frac{1}{2}}\right)\right)\\\nonumber
    &=e^{\varphi_N(\theta_2+\theta_3+\frac{\frac{1}{2}}{N+\frac{1}{2}})+\varphi_{N}(\theta_2-\theta_3+\frac{\frac{1}{2}}{N+\frac{1}{2}})+\varphi_N(-\theta_1-\theta_2+1)-\varphi_N(\theta_2)-\varphi_N(1+\theta_2-\theta_1)}\\\nonumber
    &\cdot e^{(N+\frac{1}{2})(\frac{\pi\sqrt{-1}}{12}-\frac{\pi\sqrt{-1}}{4(N+\frac{1}{2})}+\frac{\pi\sqrt{-1}}{12(N+\frac{1}{2})^2})}(N+\frac{1}{2})^{-\frac{1}{2}}
\end{align}
for $0<k+l+1\leq N$ and $0<N-a+k<N$, 

\begin{align}
    &\frac{(t)_k(t)_{N-a+k}}{(t)_{k+l+1}(t)_{k-l}(t)_{N-a-k-1}}\\\nonumber
    &=\exp\left(\varphi_N\left(\frac{(k+l+1)+\frac{1}{2}}{N+\frac{1}{2}}\right)+\varphi_N\left(\frac{(k-l)+\frac{1}{2}}{N+\frac{1}{2}}\right)+\varphi_N\left(\frac{(N-a-k-1)+\frac{1}{2}}{N+\frac{1}{2}}\right)\right.\\\nonumber
    &\left.-\varphi_N\left(\frac{k+\frac{1}{2}}{N+\frac{1}{2}}\right)-\varphi_N\left(\frac{(N-a+k)+\frac{1}{2}}{N+\frac{1}{2}}-1\right)-\varphi_N\left(\frac{\frac{1}{2}}{N+\frac{1}{2}}\right)+\log 2\right)\\\nonumber
    &=2e^{\varphi_N(\theta_2+\theta_3+\frac{\frac{1}{2}}{N+\frac{1}{2}})+\varphi_{N}(\theta_2-\theta_3+\frac{\frac{1}{2}}{N+\frac{1}{2}})+\varphi_N(-\theta_1-\theta_2+1)-\varphi_N(\theta_2)-\varphi_N(\theta_2-\theta_1)}\\\nonumber
    &\cdot e^{(N+\frac{1}{2})(\frac{\pi\sqrt{-1}}{12}-\frac{\pi\sqrt{-1}}{4(N+\frac{1}{2})}+\frac{\pi\sqrt{-1}}{12(N+\frac{1}{2})^2})}(N+\frac{1}{2})^{-\frac{1}{2}}
\end{align}
for $0<k+l+1\leq N$ and $N<N-a+k<2N$,

\begin{align}
    &\frac{(t)_k(t)_{N-a+k}}{(t)_{k+l+1}(t)_{k-l}(t)_{N-a-k-1}}\\\nonumber
    &=\exp\left(\varphi_N\left(\frac{(k+l+1)+\frac{1}{2}}{N+\frac{1}{2}}-1\right)+\varphi_N\left(\frac{(k-l)+\frac{1}{2}}{N+\frac{1}{2}}\right)+\varphi_N\left(\frac{(N-a-k-1)+\frac{1}{2}}{N+\frac{1}{2}}\right)\right.\\\nonumber
    &\left.-\varphi_N\left(\frac{k+\frac{1}{2}}{N+\frac{1}{2}}\right)-\varphi_N\left(\frac{(N-a+k)+\frac{1}{2}}{N+\frac{1}{2}}\right)-\varphi_N\left(\frac{\frac{1}{2}}{N+\frac{1}{2}}\right)-\log 2\right)\\\nonumber
    &=\frac{1}{2}e^{\varphi_N(\theta_2+\theta_3+\frac{\frac{1}{2}}{N+\frac{1}{2}}-1)+\varphi_{N}(\theta_2-\theta_3+\frac{\frac{1}{2}}{N+\frac{1}{2}})+\varphi_N(-\theta_1-\theta_2+1)-\varphi_N(\theta_2)-\varphi_N(1+\theta_2-\theta_1)}\\\nonumber
    &\cdot e^{(N+\frac{1}{2})(\frac{\pi\sqrt{-1}}{12}-\frac{\pi\sqrt{-1}}{4(N+\frac{1}{2})}+\frac{\pi\sqrt{-1}}{12(N+\frac{1}{2})^2})}(N+\frac{1}{2})^{-\frac{1}{2}}
\end{align}
for $N<k+l+1\leq 2N$ and $0<N-a+k<N$,

\begin{align}
    &\frac{(t)_k(t)_{N-a+k}}{(t)_{k+l+1}(t)_{k-l}(t)_{N-a-k-1}}\\\nonumber
    &=\exp\left(\varphi_N\left(\frac{(k+l+1)+\frac{1}{2}}{N+\frac{1}{2}}-1\right)+\varphi_N\left(\frac{(k-l)+\frac{1}{2}}{N+\frac{1}{2}}\right)+\varphi_N\left(\frac{(N-a-k-1)+\frac{1}{2}}{N+\frac{1}{2}}\right)\right.\\\nonumber
    &\left.-\varphi_N\left(\frac{k+\frac{1}{2}}{N+\frac{1}{2}}\right)-\varphi_N\left(\frac{(N-a+k)+\frac{1}{2}}{N+\frac{1}{2}}-1\right)-\varphi_N\left(\frac{\frac{1}{2}}{N+\frac{1}{2}}\right)\right)\\\nonumber
    &=e^{\varphi_N(\theta_2+\theta_3+\frac{\frac{1}{2}}{N+\frac{1}{2}}-1)+\varphi_{N}(\theta_2-\theta_3+\frac{\frac{1}{2}}{N+\frac{1}{2}})+\varphi_N(-\theta_1-\theta_2+1)-\varphi_N(\theta_2)-\varphi_N(\theta_2-\theta_1)}\\\nonumber
    &\cdot e^{(N+\frac{1}{2})(\frac{\pi\sqrt{-1}}{12}-\frac{\pi\sqrt{-1}}{4(N+\frac{1}{2})}+\frac{\pi\sqrt{-1}}{12(N+\frac{1}{2})^2})}(N+\frac{1}{2})^{-\frac{1}{2}}
\end{align}
for $N<k+l+1\leq 2N$ and $N<N-a+k<2N$.

Since
\begin{align}
     &e^{(N+\frac{1}{2})2\pi\sqrt{-1}\left(\frac{q}{2}(1-\theta_1-\frac{1}{N+\frac{1}{2}})+\frac{\theta_2+\theta_3}{2}-\frac{1}{2(N+\frac{1}{2})}\right)}\\\nonumber
     &
     \cdot e^{(N+\frac{1}{2})2\pi\sqrt{-1}\left(\frac{q}{4}(1+\theta_1^2-2\theta_1-\frac{1}{(N+\frac{1}{2})^2})+(p+\frac{1}{2})(\theta_3^2-\frac{1}{4(N+\frac{1}{2})^2})+\theta_2-\theta_2\theta_1+\frac{1}{2}\theta_2^2+\frac{\theta_2}{(N+\frac{1}{2})}-\frac{1}{8(N+\frac{1}{2})^2}\right)}
     \\\nonumber
     &\cdot e^{\varphi_N(\theta_2+\theta_3+\frac{\frac{1}{2}}{N+\frac{1}{2}}-1)+\varphi_{N}(\theta_2-\theta_3+\frac{\frac{1}{2}}{N+\frac{1}{2}})+\varphi_N(-\theta_1-\theta_2+1)-\varphi_N(\theta_2)-\varphi_N(\theta_2-\theta_1)}\\\nonumber
    &\cdot e^{(N+\frac{1}{2})(\frac{\pi\sqrt{-1}}{12}-\frac{\pi\sqrt{-1}}{4(N+\frac{1}{2})}+\frac{\pi\sqrt{-1}}{12(N+\frac{1}{2})^2})}(N+\frac{1}{2})^{-\frac{1}{2}}\\\nonumber
    &=(N+\frac{1}{2})^{-\frac{1}{2}}e^{(N+\frac{1}{2})\pi\sqrt{-1}\left(q(\frac{3}{2}-2\theta_1+\frac{\theta_1^2}{2}-\frac{1}{N+\frac{1}{2}}-\frac{1}{2(N+\frac{1}{2})^2})+(2p+1)(\theta_3^2-\frac{1}{4(N+\frac{1}{2})^2})-\theta_1+2\theta_2\theta_1+\theta_2^2+\theta_3+\frac{2\theta_2}{(N+\frac{1}{2})}\right)}\\\nonumber
    &\cdot e^{(N+\frac{1}{2})\pi\sqrt{-1}(-\frac{5}{4(N+\frac{1}{2})}-\frac{1}{6(N+\frac{1}{2})^2}+\frac{1}{12})},
\end{align}
finally, combining above formulas together, we obtain
\begin{align}
&RT_{r}(M_{p,q})\\\nonumber
&=\kappa_r'(N+\frac{1}{2})^{-\frac{1}{2}}\sum_{a=-N}^{N-1}\sum_{k=0}^{N-a-1}\sum_{l=0}^k  (-1)^{q(N-1-a)+k+l}\sin\left(\frac{(a+\frac{1}{2})\pi}{N+\frac{1}{2}}\right)\sin\left(\frac{2\pi(l+\frac{1}{2})}{N+\frac{1}{2}}\right)\\\nonumber
&\cdot t^{q\frac{(N-1-a)(N+1-a)}{4}+(p+\frac{1}{2})l(l+1)-(N-a)(k+\frac{1}{2})+\frac{1}{2}k^2+\frac{3}{2}k+\frac{1}{2}}  \frac{(t)_k(t)_{N-a+k}}{(t)_{k+l+1}(t)_{k-l}(t)_{N-a-k-1}}\\\nonumber
&=\kappa_r'(N+\frac{1}{2})^{-\frac{1}{2}}\sum_{a=-N}^{N-1}\sum_{k=0}^{N-a-1}\sum_{l=0}^k  \sin\left(\frac{(a+\frac{1}{2})\pi}{N+\frac{1}{2}}\right)\sin\left(\frac{2\pi(l+\frac{1}{2})}{N+\frac{1}{2}}\right)e^{(N+\frac{1}{2})\tilde{V}_N(\frac{a+\frac{1}{2}}{N+\frac{1}{2}},\frac{k+\frac{1}{2}}{N+\frac{1}{2}},\frac{l+\frac{1}{2}}{N+\frac{1}{2}})},
\end{align}
where 
$\tilde{V}_{N}(\theta_1,\theta_2,\theta_3)$
is given by 
\begin{itemize}
    \item 
    \begin{align*}
        &\pi\sqrt{-1}\left(q\left(\frac{3}{2}-2\theta_1+\frac{\theta_1^2}{2}-\frac{1}{N+\frac{1}{2}}-\frac{1}{2(N+\frac{1}{2})^2}\right)+(2p+1)\left(\theta_3^2-\frac{1}{4(N+\frac{1}{2})^2}\right)\right.\\\
    &\left.-\theta_2+2\theta_2\theta_1+\theta_2^2+\theta_3+\frac{2\theta_2}{N+\frac{1}{2}}-\frac{5}{4(N+\frac{1}{2})}-\frac{1}{6(N+\frac{1}{2})^2}+\frac{1}{12}\right)\\\nonumber
    &+\frac{1}{N+\frac{1}{2}}\left(\varphi_N(\theta_2+\theta_3+\frac{\frac{1}{2}}{N+\frac{1}{2}})+\varphi_{N}(\theta_2-\theta_3+\frac{\frac{1}{2}}{N+\frac{1}{2}})+\varphi_N(-\theta_1-\theta_2+1)\right. \\\nonumber
    &\left.-\varphi_N(\theta_2)-\varphi_N(1+\theta_2-\theta_1)\right),
    \end{align*}
when $0<\theta_2+\theta_3<1$ and $0<1-\theta_1+\theta_2<1$;
    \item 
    \begin{align*}
       &\pi\sqrt{-1}\left(q\left(\frac{3}{2}-2\theta_1+\frac{\theta_1^2}{2}-\frac{1}{N+\frac{1}{2}}-\frac{1}{2(N+\frac{1}{2})^2}\right)+(2p+1)\left(\theta_3^2-\frac{1}{4(N+\frac{1}{2})^2}\right)\right.\\\
    &\left.-\theta_2+2\theta_2\theta_1+\theta_2^2+\theta_3+\frac{2\theta_2}{N+\frac{1}{2}}-\frac{5}{4(N+\frac{1}{2})}-\frac{1}{6(N+\frac{1}{2})^2}+\frac{1}{12}\right)\\\nonumber
    &+\frac{1}{N+\frac{1}{2}}\left(\varphi_N(\theta_2+\theta_3+\frac{\frac{1}{2}}{N+\frac{1}{2}})+\varphi_{N}(\theta_2-\theta_3+\frac{\frac{1}{2}}{N+\frac{1}{2}})+\varphi_N(-\theta_1-\theta_2+1)-\varphi_N(\theta_2)\right.\\\nonumber
    &\left.-\varphi_N(\theta_2-\theta_1)\right),
    \end{align*}
    when $0<\theta_2+\theta_3<1$ and $1<1-\theta_1+\theta_2<2$;

    \item 
    \begin{align*}
       &\pi\sqrt{-1}\left(q\left(\frac{3}{2}-2\theta_1+\frac{\theta_1^2}{2}-\frac{1}{N+\frac{1}{2}}-\frac{1}{2(N+\frac{1}{2})^2}\right)+(2p+1)\left(\theta_3^2-\frac{1}{4(N+\frac{1}{2})^2}\right)\right.\\\
    &\left.-\theta_2+2\theta_2\theta_1+\theta_2^2+\theta_3+\frac{2\theta_2}{N+\frac{1}{2}}-\frac{5}{4(N+\frac{1}{2})}-\frac{1}{6(N+\frac{1}{2})^2}+\frac{1}{12}\right)\\\nonumber
    &+\frac{1}{N+\frac{1}{2}}\left(\varphi_N(\theta_2+\theta_3+\frac{\frac{1}{2}}{N+\frac{1}{2}}-1)+\varphi_{N}(\theta_2-\theta_3+\frac{\frac{1}{2}}{N+\frac{1}{2}})+\varphi_N(-\theta_1-\theta_2+1)\right.\\\nonumber
    &\left.-\varphi_N(\theta_2)-\varphi_N(1+\theta_2-\theta_1)\right),
    \end{align*}
when $1<\theta_2+\theta_3<2$ and $0<1-\theta_1+\theta_2<1$;

    \item
    \begin{align*}
       &\pi\sqrt{-1}\left(q\left(\frac{3}{2}-2\theta_1+\frac{\theta_1^2}{2}-\frac{1}{N+\frac{1}{2}}-\frac{1}{2(N+\frac{1}{2})^2}\right)+(2p+1)\left(\theta_3^2-\frac{1}{4(N+\frac{1}{2})^2}\right)\right.\\\
    &\left.-\theta_2+2\theta_2\theta_1+\theta_2^2+\theta_3+\frac{2\theta_2}{N+\frac{1}{2}}-\frac{5}{4(N+\frac{1}{2})}-\frac{1}{6(N+\frac{1}{2})^2}+\frac{1}{12}\right)\\\nonumber
    &+\frac{1}{N+\frac{1}{2}}\left(\varphi_N(t+u+\frac{\frac{1}{2}}{N+\frac{1}{2}}-1)+\varphi_{N}(\theta_2-\theta_3+\frac{\frac{1}{2}}{N+\frac{1}{2}})\right.\\\nonumber
    &\left.+\varphi_N(-\theta_1-\theta_2+1)-\varphi_N(\theta_2)-\varphi_N(\theta_2-\theta_1)\right),
    \end{align*}
     when $1<\theta_2+\theta_3<2$ and $1<1-\theta_1+\theta_2<2$.
\end{itemize}

By Lemma \ref{lemma-varphixi2}, we have
\begin{align}
    &\varphi_N(1-\theta_1-\theta_2)=-\varphi(\theta_1+\theta_2)\\\nonumber
    &+2\pi \sqrt{-1}\left(-\frac{2N+1}{4}\left((\theta_1+\theta_2)^2-(\theta_2+\theta_1)+\frac{1}{6}\right)+\frac{1}{12(2N+1)}\right)
\end{align}
since $\theta_1+\theta_2<1$.

Hence 
\begin{align}
&\tilde{V}_{N}(\theta_1,\theta_2,\theta_3)\\\nonumber
&=\pi\sqrt{-1}\left(q\left(\frac{3}{2}-2\theta_1+\frac{\theta_1^2}{2}-\frac{1}{N+\frac{1}{2}}-\frac{1}{2(N+\frac{1}{2})^2}\right)+(2p+1)\left(\theta_3^2-\frac{1}{4(N+\frac{1}{2})^2}\right)\right.\\\nonumber
    &\left.-\theta_1^2+\theta_1+\theta_3+\frac{2\theta_2}{N+\frac{1}{2}}-\frac{5}{4(N+\frac{1}{2})}-\frac{1}{12(N+\frac{1}{2})^2}-\frac{1}{12}\right)\\\nonumber
    &+\frac{1}{N+\frac{1}{2}}\left(\varphi_N(\theta_2+\theta_3+\frac{\frac{1}{2}}{N+\frac{1}{2}}-1)+\varphi_{N}(\theta_2-\theta_3+\frac{\frac{1}{2}}{N+\frac{1}{2}})\right.\\\nonumber
    &\left.-\varphi_N(\theta_2+\theta_1)-\varphi_N(\theta_2)-\varphi_N(\theta_2-\theta_1)\right),    
\end{align}
for $1<\theta_2+\theta_3<2$ and $1<1-\theta_1+\theta_2<2$, and we can write the expression for $\tilde{V}_N$ for other cases similarly.

Then we introduce the functions 
\begin{align} \label{formula-VN}
   &\hat{V}_N(\theta_1,\theta_2,\theta_3)\\\nonumber
   &=\tilde{V}_N(\theta_1,\theta_2,\theta_3)+2(q-1)\pi\sqrt{-1}\left(\theta_1-\frac{\frac{1}{2}}{N+\frac{1}{2}}\right)-2\pi\sqrt{-1}\left(\theta_2-\frac{\frac{1}{2}}{N+\frac{1}{2}}\right)\\\nonumber
   &-2(p+2)\pi\sqrt{-1}\left(\theta_3-\frac{\frac{1}{2}}{N+\frac{1}{2}}\right) \\\nonumber
   &=\pi\sqrt{-1}\left(\left(\frac{q}{2}-1\right)\theta_1^2-\theta_1+(2p+1)\theta_3^2-(2p+3)\theta_3-2\theta_2+\frac{3}{2}q-\frac{1}{12}\right.\\\nonumber
   &\left.+\frac{2\theta_2}{N+\frac{1}{2}}+\frac{p-2q+\frac{9}{4}}{N+\frac{1}{2}}-\frac{3(p+q)+2}{6(N+\frac{1}{2})^2}\right)\\\nonumber
    &+\frac{1}{N+\frac{1}{2}}\left(\varphi_N(\theta_2+\theta_3+\frac{\frac{1}{2}}{N+\frac{1}{2}}-1)+\varphi_{N}(\theta_2-\theta_3+\frac{\frac{1}{2}}{N+\frac{1}{2}})\right.\\\nonumber
    &\left.-\varphi_N(\theta_2+\theta_1)-\varphi_N(\theta_2)-\varphi_N(\theta_2-\theta_1)\right) 
\end{align}
and 
\begin{align}
&\hat{V}(\theta_1,\theta_2,\theta_3)\\\nonumber
&=\lim_{N\rightarrow \infty}V_{N}(\theta_1,\theta_2,\theta_3)\\\nonumber
&=\pi\sqrt{-1}\left(\left(\frac{q}{2}-1\right)\theta_1^2-\theta_1+(2p+1)\theta_3^2-(2p+3)\theta_3-2\theta_2+\frac{3}{2}q-\frac{1}{12}\right)\\\nonumber
   &+\frac{1}{2\pi\sqrt{-1}}\left(\text{Li}_2(e^{2\pi\sqrt{-1}(\theta_2+\theta_3)})+\text{Li}_2(e^{2\pi\sqrt{-1}(\theta_2-\theta_3)})-\text{Li}_2(e^{2\pi\sqrt{-1}(\theta_2+\theta_1)})\right.\\\nonumber
&\left.-\text{Li}_2(e^{2\pi\sqrt{-1}\theta_2})-\text{Li}_2(e^{2\pi\sqrt{-1}(\theta_2-\theta_1)})\right)  
\end{align}
for $1<\theta_2+\theta_3<2$ and $1<1-\theta_1+\theta_2<2$, and one can write the expression for $V_N$  and  $V$ for other cases similarly.

\begin{proposition} \label{Prop-RTN}
Let $r=2N+1$ with $N\geq 1$, the $r$-th Reshetikhin-Turaev invariant of $M_{p,q}$ can be written as 
  \begin{align}
   RT_{r}(M_{p,q})=\kappa_r'(N+\frac{1}{2})^{-\frac{1}{2}}\sum_{a=0}^{N-1}\sum_{k=0}^{N-a-1}\sum_{l=0}^kg_{N}(a,k,l)   
  \end{align}  
with
\begin{align}
 g_{N}(a,k,l)= \sin\left(\frac{(a+\frac{1}{2})\pi}{N+\frac{1}{2}}\right)\sin\left(\frac{2\pi(l+\frac{1}{2})}{N+\frac{1}{2}}\right)e^{(N+\frac{1}{2})\hat{V}_N\left(\frac{a+\frac{1}{2}}{N+\frac{1}{2}},\frac{k+\frac{1}{2}}{N+\frac{1}{2}},\frac{l+\frac{1}{2}}{N+\frac{1}{2}}\right)}, 
\end{align}
  where the function $\hat{V}_N(\theta_1,\theta_2,\theta_3)$ is given by formula (\ref{formula-VN}), and $\kappa'_r$ is given by (\ref{formula-kappa'r}).
\end{proposition}

\section{Poisson summation formula }  \label{Section-Poissonsummation}
By Proposition \ref{Prop-RTN}, we have
\begin{align}
   g_{N}(a,k,l)=\sin\left(\frac{(a+\frac{1}{2})\pi}{N+\frac{1}{2}}\right)\sin\left(\frac{2\pi(l+\frac{1}{2})}{N+\frac{1}{2}}\right)e^{(N+\frac{1}{2})\hat{V}_N\left(\frac{a+\frac{1}{2}}{N+\frac{1}{2}},\frac{k+\frac{1}{2}}{N+\frac{1}{2}},\frac{l+\frac{1}{2}}{N+\frac{1}{2}}\right)}
\end{align}
We consider the real part of the potential function 
\begin{align}
&v(\theta_1,\theta_2,\theta_3)\\\nonumber
&=\text{Re}\hat{V}(\theta_1,\theta_2,\theta_3)\\\nonumber
&= \text{Re}\left(\frac{1}{2\pi\sqrt{-1}}\left(\text{Li}_2(e^{2\pi\sqrt{-1}(\theta_2+\theta_3)})+\text{Li}_2(e^{2\pi\sqrt{-1}(\theta_2-\theta_3)})-\text{Li}_2(e^{2\pi\sqrt{-1}(\theta_2+\theta_1)})\right.\right.\\\nonumber
&\left.\left.-\text{Li}_2(e^{2\pi\sqrt{-1}\theta_2})-\text{Li}_2(e^{2\pi\sqrt{-1}(\theta_2-\theta_1)})\right)   \right)\\\nonumber
&=\Lambda(\theta_2+\theta_3)+\Lambda(\theta_2-\theta_3)-\Lambda(\theta_2+\theta_1)-\Lambda(\theta_2)-\Lambda(\theta_2-\theta_1)
\end{align}
where $D$ is given by the following domain 
\begin{align}
D'=\{(\theta_1,\theta_2,\theta_3)\in \mathbb{R}^3|-1 \leq \theta_1\leq 1, 0\leq \theta_3\leq \theta_2\leq 1-\theta_1 \}    
\end{align}

Let 
\begin{align}
D=\{(\theta_1,\theta_2,\theta_3)\in D'| \frac{1}{2}<\theta_2<1\},    
\end{align}
we have

\begin{lemma}
There exists $\epsilon>0$, such that 
the following domain 
\begin{align}
   \text{Re}\hat{V}(\theta_1,\theta_2,\theta_3)>0.324-\epsilon   
\end{align}
is included in the region $D$. 
\end{lemma}
\begin{proof}
We only need to show that when $0\leq \theta_2\leq \frac{1}{2}$, we have 
\begin{align} \label{formula-ReV<}
   \text{Re}\hat{V}(\theta_1,\theta_2,\theta_3)\leq 0.324-\epsilon. 
\end{align}
Clearly, when $\theta_2=0, \frac{1}{2}$, $\text{Re}\hat{V}(\theta_1,\theta_2,\theta_3)=0\leq 0.324-\epsilon$ holds.

For a fixed $0<\theta_2< \frac{1}{2}$, we consider the function 
\begin{align}
 f(\theta_1,\theta_2)=2\Lambda(\theta_2)-\Lambda(\theta_2+\theta_1)-\Lambda(\theta_2-\theta_1).   
\end{align}
Then
\begin{align}
\text{Re}\hat{V}(\theta_1,\theta_2,\theta_3)=\Lambda(\theta_2+\theta_3)+\Lambda(\theta_2-\theta_3)-3\Lambda(\theta_2)+f(\theta_1,\theta_2).
\end{align}

Clearly, 
\begin{align}
f_{\theta_1}(\theta_1,\theta_2)=\log \left|\frac{\sin(\pi(\theta_2+\theta_1)) }{\sin(\pi(\theta_2-\theta_1))}\right|.     
\end{align}
since $0<\theta_2<\frac{1}{2}$ and $\theta_2+\theta_1<1$, we obtain that $f_{\theta_1}(\theta_1,\theta_2)=0$ has a unique solution $\theta_1=\frac{1}{2}$. Moreover, when $\theta_1\in [0,\frac{1}{2}]$, we have $f_{\theta_1}(\theta_1,\theta_2)\geq 0$ and when $\theta_1\in [\frac{1}{2},1-\theta_2]$, we have $f_{\theta_1}(\theta_1,\theta_2)\leq 0$.  Hence, for $0\leq \theta_2\leq \frac{1}{2}$ and $0\leq \theta_1\leq 1-\theta_2$, we have
\begin{align}
   f(\theta_1,\theta_2)\leq f\left(\frac{1}{2},\theta_2\right)=2\Lambda(\theta_2).  
\end{align}

Suppose formula (\ref{formula-ReV<}) does not holds, then
\begin{align}
\text{Re}\hat{V}(\theta_1,\theta_2,\theta_3)=\Lambda(\theta_2+\theta_3)+\Lambda(\theta_2-\theta_3)-3\Lambda(\theta_2)+f(\theta_1,\theta_2)\geq 0.324-\epsilon  
\end{align}
So we obtain 
\begin{align}
-3\Lambda(\theta_2)&\geq 0.324-\epsilon-f(\theta_1,\theta_2)-2\Lambda\left(\frac{1}{6}\right)\\\nonumber
&\geq 0.324-\epsilon-2\Lambda(\theta_2)-2\Lambda\left(\frac{1}{6}\right). 
\end{align}
It implies that 
\begin{align}
   -\Lambda(\theta_2)\geq 0.324-\epsilon-0.3230659473> 0,  
\end{align}
which is impossible, since $-\Lambda(\theta_2)\leq 0$ when $\theta_2\in [0,\frac{1}{2}]$.  
\end{proof}

Furthermore, as in \cite{CZ23-1}, we introduce the domain  
\begin{align}
D_0=&\left\{ (\theta_1,\theta_2,\theta_3)\in D|0.02 \leq \theta_2-\theta_3\leq 0.7, 1.02 \leq \theta_2+\theta_3\leq 1.7, \right. \\\nonumber & \left. 0.2 \leq \theta_3\leq 0.8,0.5\leq \theta_2\leq 0.909\right\}\cap D'
\end{align}
for this case.  

Note that the region $D_0$ lies in the tetrahedron $FEBC$ which is symmetry with respect to the plane EFH and the $\theta_2\theta_3$-plane. Note that, we only illustrated the upper half of this region in Figure \ref{figure-D0}.    

\begin{figure}[!htb] \label{figure-D0} 
\begin{align} 
\raisebox{-15pt}{
\includegraphics[width=300 pt]{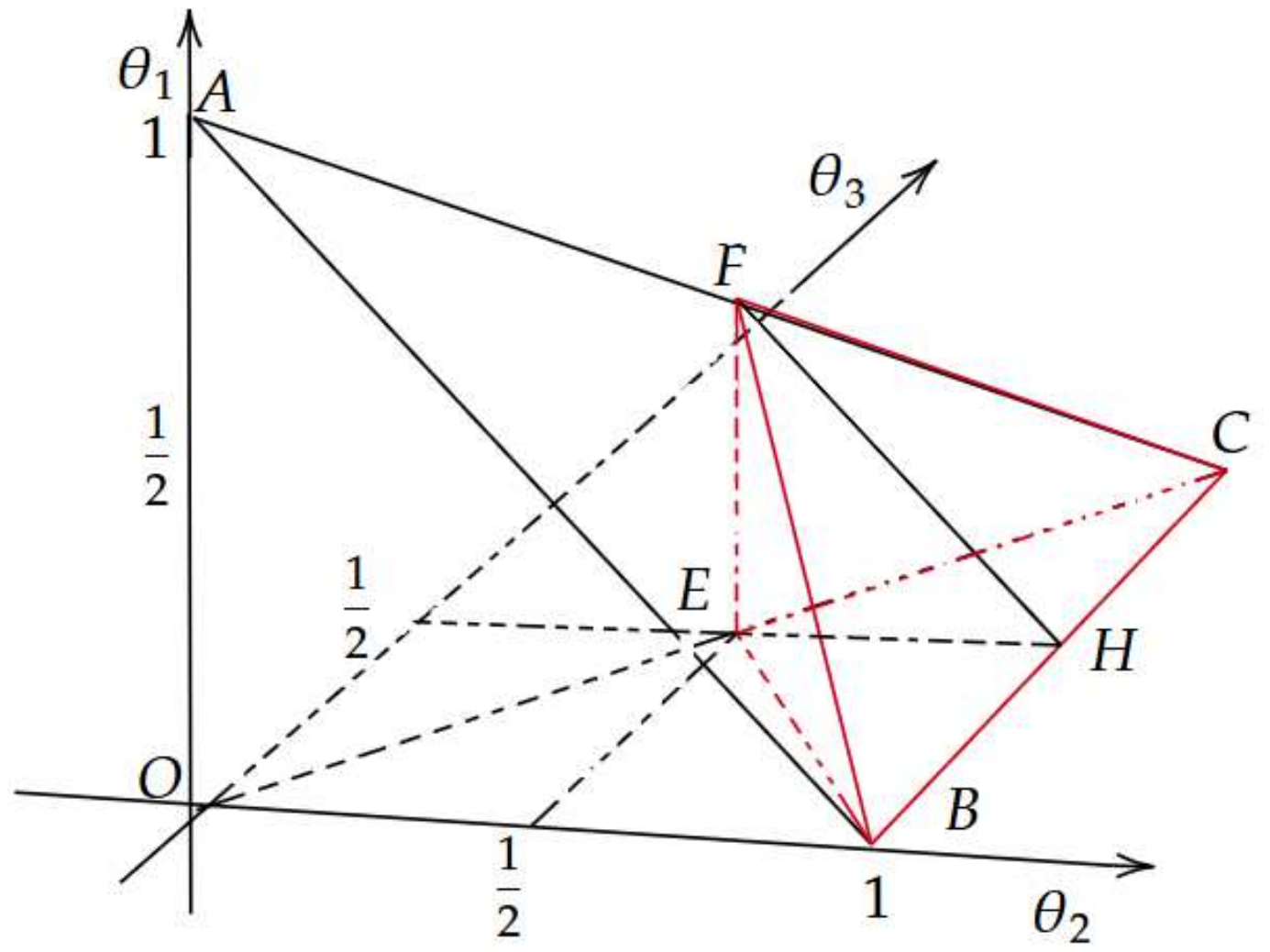}}.
\end{align}
\caption{The region $D_0$ lies in the tetrahedron FEBC}
\end{figure}

Then we have
\begin{lemma} \label{lemma-regionD'0}
The following domain
    \begin{align} \label{formula-domain}
        \left\{(\theta_1,\theta_3,\theta_3)\in D| v(\theta_1,\theta_2,\theta_3)> \frac{3.56337}{2\pi }\right\}
    \end{align}
    is included in the region $D_0$.
\end{lemma}
\begin{proof}
By using formulas (\ref{formula-vtheta1}) and (\ref{formula-vtheta2}), and together with Lemma 4.1 in \cite{CZ23-1}.     
\end{proof}

By Corollary \ref{corollary-zetaR},  for $p,q$ satisfies the condition (\ref{formula-pqconditions}), we have $\zeta_{\mathbb{R}}(p,q)\geq \frac{3.56337}{2\pi }$. Hence
\begin{proposition} \label{prop-gkl}
For $(p,q)\in S$ and  $(\frac{a+\frac{1}{2}}{N+\frac{1}{2}},\frac{k+\frac{1}{2}}{N+\frac{1}{2}},\frac{l+\frac{1}{2}}{N+\frac{1}{2}})\in D\setminus D_0$,  we have
\begin{align}
    |g_{N}(k,l)|<O\left(e^{(N+\frac{1}{2})\left(\zeta_{\mathbb{R}}(p,q)-\epsilon\right)}\right)
\end{align}
for some sufficiently small $\epsilon>0$.
\end{proposition}

For a sufficiently small $\varepsilon$, we take a smooth bump function $\psi$ on $\mathbb{R}^3$ such that
$\psi(\theta_1,\theta_2,\theta_3)=1$ on $(\theta_1,\theta_2,\theta_3)\in D_{\varepsilon}$,  $0<\psi(\theta_1,\theta_2,\theta_3)<1$ on $(\theta_1,\theta_2,\theta_3)\in D_0\setminus D_{\varepsilon}$, $\psi(\theta_1,\theta_2,\theta_3)=0$ for $(\theta_1,\theta_2,\theta_3)\notin D_0$.  
Let 
\begin{align}
    h_N(a,k,l)=\psi\left(\frac{2a+1}{2N+1},\frac{2k+1}{2N+1},\frac{2l+1}{2N+1}\right)g_N(a,k,l).
\end{align}

Then  by Proposition \ref{prop-gkl}, for $(p,q)$ satisfies (\ref{formula-pqconditions})  we obtain
\begin{align}  \label{formula-JN}
    RT_r(M)=\sum_{(a,k,l)\in \mathbb{Z}^3}h_N(a,k,l)+O\left(e^{(N+\frac{1}{2})\left(\zeta_{\mathbb{R}}(p,q)-\epsilon\right)}\right).
\end{align}
Recall the Poisson summation formula \cite{SS03} in 3-dimensional case which states that for any function $h$ in the Schwartz space on $\mathbb{R}^2$, we have 
\begin{align} \label{formula-Poisson}
    \sum_{(a,k,l)\in \mathbb{Z}^3}h(a,k,l)=\sum_{(m_1,m_2,m_3)\in \mathbb{Z}^3}\hat{h}(m_1,m_2,m_3)
\end{align}
where 
\begin{align}
    \hat{h}(m_1,m_2,m_3)=\int_{\mathbb{R}^3}h(\theta_1,\theta_2,\theta_3)e^{-2\pi \sqrt{-1}m_1 \theta_1-2\pi \sqrt{-1}m_2 \theta_2-2\pi\sqrt{-1}m_3 \theta_3}d\theta_1d\theta_2d\theta_3.
\end{align}

Note that $h_N$ is $C^{\infty}$-smooth and equals zero outside $D_0$, it is in the Schwartz space on $\mathbb{R}^3$. The Poisson summation formula (\ref{formula-Poisson}) holds for $h_N$. 

By using change of variables $\theta_1=\frac{a+\frac{1}{2}}{N+\frac{1}{2}}, \theta_2=\frac{k+\frac{1}{2}}{N+\frac{1}{2}}, \theta_3=\frac{l+\frac{1}{2}}{N+\frac{1}{2}}$, 
we compute the Fourier coefficient $\hat{h}_{N}(m_1,m_2,m_3)$ as follows
\begin{align}
    &\int_{\mathbb{R}^3}h_N(a,k,l)e^{-2\pi \sqrt{-1}m_1 a-2\pi \sqrt{-1}m_2 k-2\pi\sqrt{-1}m_3 l}dadkdl\\\nonumber
    &=(-1)^{m_1+m_2+m_3}\left(N+\frac{1}{2}\right)^3\\\nonumber
    &\cdot\int_{\mathbb{R}^3}h_N\left((N+\frac{1}{2})\theta_1-\frac{1}{2},(N+\frac{1}{2})\theta_2-\frac{1}{2},(N+\frac{1}{2})\theta_3-\frac{1}{2}\right)e^{-2\pi \sqrt{-1}\sum_{i=1}^3\frac{(2N+1)m_i(\theta_1+\theta_2+\theta_3)}{2}}d\theta_1d\theta_2d\theta_3\\\nonumber
    &=(-1)^{m_1+m_2+m_3}\left(N+\frac{1}{2}\right)^3\kappa_r'\\\nonumber
    &\cdot\int_{D_0}\sin(\pi \theta_1)\sin(2\pi \theta_3)e^{(N+\frac{1}{2})\left(V_N\left(\theta_1,\theta_2,\theta_3\right)-2\pi\sqrt{-1}m_1\theta_1-2\pi\sqrt{-1}m_2\theta_2-2\pi\sqrt{-1}m_3 \theta_3\right)}d\theta_1d\theta_2d\theta_3,
\end{align}
Therefore, applying the Poisson summation formula (\ref{formula-Poisson}) to (\ref{formula-JN}), we obtain
\begin{proposition} \label{prop-fouriercoeff}
For $(p,q)\in S$ , the Reshetinkhin-Turaev invariant $RT_r(M)$ is given by 
\begin{align} \label{formula-fouriercoeff}
     RT_r(M)=\sum_{(m_1,m_2,m_3)\in \mathbb{Z}^3}\hat{h}_N(m_1,m_2,m_3)+O\left(e^{(N+\frac{1}{2})\left(\zeta_{\mathbb{R}}(p,q)-\epsilon\right)}\right),
\end{align}
where
  \begin{align}
      \hat{h}_N(m_1,m_2,m_3)=&(-1)^{m_1+m_2+m_3}\left(N+\frac{1}{2}\right)^{\frac{5}{2}}\kappa_r' \\\nonumber
      &\cdot\int_{D_0}\sin(\pi \theta_1)\sin(2\pi \theta_3)e^{(N+\frac{1}{2})\hat{V}_N\left(\theta_1,\theta_2,\theta_3;m_1,m_2,m_3\right)}dsdtdu
\end{align}
with
\begin{align}
&\hat{V}_N\left(\theta_1,\theta_2,\theta_3;m_1,m_2,m_3\right)\\\nonumber
&=\hat{V}_N\left(\theta_1,\theta_2,\theta_3\right)-2\pi\sqrt{-1}m_1\theta_1-2\pi\sqrt{-1}m_2\theta_2-2\pi\sqrt{-1}m_3 \theta_3,
\end{align}
and $\hat{V}_{N}(\theta_1,\theta_2,\theta_3)$ is given by formula (\ref{formula-VN}). 

\end{proposition}

\begin{lemma}
The following identity holds
    \begin{align} \label{formula-potientalsym}
          &\hat{V}_{N}(\theta_1,\theta_2,1-\theta_3;m_1,m_2,m_3)
          &=\hat{V}_{N}(\theta_1,\theta_2,\theta_3;m_1,m_2,-m_3-2)-2\pi \sqrt{-1}(m_3+1). 
    \end{align}
\end{lemma}
\begin{proof}
   Since we have the following identity
    \begin{align}
        &\pi\sqrt{-1}\left((2p+1)(1-\theta_3)^2-(2p+2m_3+3)(1-\theta_3)\right)\\\nonumber
        &=\pi\sqrt{-1}\left((2p+1)\theta_3^2-(2p+2(-m_3-2)+3)\theta_3)\right)\\\nonumber
        &-2\pi \sqrt{-1}(m_3+1).
    \end{align}
    which immediately gives the formula (\ref{formula-potientalsym}).
\end{proof}

\begin{proposition} \label{prop-hmn}
    For any $m_1,m_2,m_3\in \mathbb{Z}$, we have
    \begin{align}
        \hat{h}_{N}(m_1,m_2,-m_3-2)=(-1)^{m_3}\hat{h}_{N}(m_1,m_2,m_3).
    \end{align}
\end{proposition}
\begin{proof}
    Since 
    \begin{align}
      &\sin(\pi \theta_1)\sin(2\pi \theta_3)e^{(N+\frac{1}{2})\hat{V}_N\left(\theta_1,\theta_2,1-\theta_3;m_1,m_2,m_3\right)} \\\nonumber
      &=\sin(\pi \theta_1)\sin(2\pi \theta_3)e^{(N+\frac{1}{2})\left(\hat{V}_{N}\left(\theta_1,\theta_2,\theta_3;m_1,m_2,-m_3-2\right)-2\pi\sqrt{-1}(m_3+1)\right)}\\\nonumber
      &=\sin(\pi \theta_1)\sin(2\pi \theta_3)e^{(N+\frac{1}{2})\hat{V}_N\left(\theta_1,\theta_2,\theta_3;m_1,m_2,-m_3-2\right)}(-1)^{m_3+1},
    \end{align}
  since the region $D_0$ is symmetric with respect to the plane $\theta_3=\frac{1}{2}$. 

Then we have
\begin{align}
    &\int_{D_0}\sin(\pi s)\sin(2\pi u)e^{(N+\frac{1}{2})\hat{V}_N\left(\theta_1,\theta_2,\theta_3;m_1,m_2,-m_3-2\right)}d\theta_1d\theta_2d\theta_3\\\nonumber
    &=(-1)^{m_3+1}\int_{D_0}\sin (\pi \theta_1)\sin(2\pi \theta_3)e^{(N+\frac{1}{2})\hat{V}_N\left(\theta_1,\theta_2,1-\theta_3;m_1,m_2,m_3\right)}d\theta_1d\theta_2d\theta_3\\\nonumber
    &=(-1)^{m_3}\int_{D_0}\sin(\pi s)\sin(2\pi \tilde{u})\exp^{(N+\frac{1}{2})\hat{V}_N\left(\theta_1,\theta_2,\tilde{\theta}_3;m_1,m_2,m_3\right)}d\theta_1 d\theta_2 d\tilde{\theta}_3,
\end{align}
where in the third ``=", we have let $\tilde{\theta}_3=1-\theta_3$.  It follows that 
\begin{align}
        \hat{h}_{N}(m_1,m_2,-m_3-2)=(-1)^{m_3}\hat{h}_{N}(m_1,m_2,m_3).
    \end{align}
\end{proof}
\begin{corollary} 
    We have  
    \begin{align} \label{formula-bigcancel}
        \hat{h}_{N}(m_1,m_2,-1)=0,
    \end{align}
    and 
    \begin{align}
    \hat{h}_{N}(m_1,m_2,-2)=\hat{h}_{N}(m_1,m_2,0).    
    \end{align}
\end{corollary}

Note that
\begin{align} \label{formula-hatV}
  \hat{V}_{N}(-\theta_1,\theta_2,\theta_3,m_1,m_2,m_3)=\hat{V}_N(\theta_1,\theta_2,\theta_3,-m_1-1,m_2,m_3).  
\end{align}

Since the region $D_0$ is symmetry with respect to the plane $\theta_1=0$.   We have
\begin{align}
&\int_{D_0}\sin(\pi \theta_1)\sin(2\pi\theta_3)e^{(N+\frac{1}{2})\hat{V}_N(\theta_1,\theta_2,\theta_3,m_1,m_2,m_3)}d\theta_1d\theta_2d\theta_3\\\nonumber
&=-\int_{D_0}\sin(\pi\tilde{\theta}_1)\sin(2\pi\theta_3)e^{(N+\frac{1}{2})\hat{V}_N(-\tilde{\theta}_1,\theta_2,\theta_3,m_1,m_2,m_3)}d\tilde{\theta}_1d\theta_2d\theta_3\\\nonumber
&=-\int_{D_0}\sin(\pi\theta_1)\sin(2\pi\theta_3)e^{(N+\frac{1}{2})\hat{V}_N(\theta_1,\theta_2,\theta_3,-m_1-1,m_2,m_3)}d\theta_1d\theta_2d\theta_3.
\end{align}

Hence, we have 
\begin{align}
\hat{h}_N(m_1,m_2,m_3)= \hat{h}_N(-m_1-1,m_2,m_3).   
\end{align}

So we obtain 
\begin{align} \label{formula-RTN}
   RT_r(M)&=2\sum_{m_1\geq 0}\sum_{m_2\in \mathbb{Z}}\sum_{m_3\in \mathbb{Z}}\hat{h}_{N}(m_1,m_2,m_3)+O(e^{(N+\frac{1}{2})(\zeta_\mathbb{R}(p,q))-\epsilon})\\\nonumber
   &=4\sum_{m_1\geq 0}\sum_{m_2\in \mathbb{Z}}\sum_{m_3\geq 0}\hat{h}_{N}(m_1,m_2,m_3)+O(e^{(N+\frac{1}{2})(\zeta_\mathbb{R}(p,q))-\epsilon}),
\end{align}
where in the second ``=" we have used the Proposition \ref{prop-hmn}. Note that 
\begin{align} \label{formula-heN}
 \hat{h}_N(m_1,m_2,m_3)&=(-1)^{m_1+m_2+m_3}\kappa''_r\\\nonumber
 &\cdot\int_{D_0}\sin(\pi \theta_1)\sin(2\pi\theta_3)e^{(N+\frac{1}{2})\hat{V}_N(\theta_1,\theta_2,\theta_3,m_1,m_2,m_3)}d\theta_1d\theta_2d\theta_3,   
\end{align}
where 
\begin{align}
  \kappa''_r= \left(N+\frac{1}{2}\right)^\frac{5}{2}\kappa'_r=\left(N+\frac{1}{2}\right)^\frac{3}{2}e^{\sigma_{p,q}\left(\frac{3}{2N+1}+\frac{N+1}{2}\right)\pi\sqrt{-1}}. 
\end{align}

\section{Asymptotic expansions}  \label{Section-asympticexpansion}
The goal of this section is to estimate each Fourier coefficients $\hat{h}_N(m_1,m_2,m_3)$ appearing in (\ref{formula-RTN}).

\subsection{Preparations}
Recall the functions
\begin{align} \label{formula-PotentialV}
   &\hat{V}_N(\theta_1,\theta_2,\theta_3)\\\nonumber
   &=\pi\sqrt{-1}\left(\left(\frac{q}{2}-1\right)\theta_1^2-\theta_1+(2p+1)\theta_3^2-(2p+3)\theta_3-2\theta_2+\frac{3}{2}q-\frac{1}{12}\right.\\\nonumber
   &\left.+\frac{2\theta_2}{N+\frac{1}{2}}+\frac{p-2q+\frac{7}{4}}{N+\frac{1}{2}}-\frac{3(p+q)+2}{6(N+\frac{1}{2})^2}\right)\\\nonumber
    &+\frac{1}{N+\frac{1}{2}}\left(\varphi_N(\theta_2+\theta_3+\frac{\frac{1}{2}}{N+\frac{1}{2}}-1)+\varphi_{N}(\theta_2-\theta_3+\frac{\frac{1}{2}}{N+\frac{1}{2}})\right.\\\nonumber
    &\left.-\varphi_N(\theta_2+\theta_1)-\varphi_N(\theta_2)-\varphi_N(\theta_2-\theta_1)\right)
\end{align}
and
\begin{align}
\hat{V}(\theta_1,\theta_2,\theta_3)&=\pi\sqrt{-1}\left(\frac{3q}{2}+(\frac{q}{2}-1)\theta_1^2-\theta_1+(2p+1)\theta_3^2-(2p+3)\theta_3-2\theta_2\right)\\\nonumber
&+\frac{1}{2\pi\sqrt{-1}}\left(\frac{\pi^2}{6}+\text{Li}_2(e^{2\pi\sqrt{-1}(\theta_2+\theta_3)})+\text{Li}_2(e^{2\pi\sqrt{-1}(\theta_2-\theta_3)}))\right.\\\nonumber
&\left.-\text{Li}_2(e^{2\pi\sqrt{-1}(\theta_2+\theta_1)}-\text{Li}_2(e^{2\pi\sqrt{-1}\theta_2})-\text{Li}_2(e^{2\pi\sqrt{-1}(\theta_2-\theta_1)})\right).
\end{align}

\begin{proposition} \label{prop-VNexpansion}
For any $L>0$, in the region
\begin{align}
\{(\theta_1,\theta_2,\theta_3)\in \mathbb{C}^3|(\text{Re}(\theta_1),\text{Re}(\theta_2),\text{Re}(\theta_3))\in D_{0}, |\text{Im}(\theta_1)|<L, |\text{Im} \theta_2|<L, |\text{Im} \theta_3|<L\},    
\end{align}
we have 
\begin{align} \label{formula-VNexpansion}
    &\hat{V}_{N}(\theta_1,\theta_2,\theta_3)=\hat{V}(\theta_1,\theta_2,\theta_3)+\frac{1}{N+\frac{1}{2}}\left(\pi\sqrt{-1}\left(2\theta_2+(p-2q+\frac{7}{4})\right)\right.\\\nonumber
    &\left.-\frac{1}{2}\log(1-e^{2\pi\sqrt{-1}(\theta_2+\theta_3)})-\frac{1}{2}\log(1-e^{2\pi\sqrt{-1}(\theta_2-\theta_3)})\right)+\frac{1}{(N+\frac{1}{2})^2}w_{N}(\theta_1,\theta_2,\theta_3),
\end{align}
with $|w_{N}(\theta_1,\theta_2,\theta_3)|$  bounded from above by a constant independent of $N$. 
\end{proposition}
\begin{proof}
By using Taylor expansion, together with Lemma \ref{lemma-varphixi2}, we have

\begin{align}
    &\varphi_N\left(\theta_2+\theta_3-1+\frac{1}{N+\frac{1}{2}}\right)\\\nonumber
    &=\varphi_N(\theta_2+\theta_3-1)+\varphi'_N(\theta_2+\theta_3-1)\frac{1}{N+\frac{1}{2}}\\\nonumber&+\frac{\varphi''_{N}(\theta_2+\theta_3-1)}{2}\frac{1}{(N+\frac{1}{2})^2}+O\left(\frac{1}{(N+\frac{1}{2})^2}\right)\\\nonumber
    &=\frac{N+\frac{1}{2}}{2\pi\sqrt{-1}}\text{Li}_2(e^{2\pi\sqrt{-1}(\theta_2+\theta_3)})-\frac{\pi\sqrt{-1}}{6(2N+1)}\frac{e^{2\pi\sqrt{-1}(\theta_2+\theta_3)}}{1-e^{2\pi\sqrt{-1}(\theta_2+\theta_3)}}\\\nonumber
    &-\frac{1}{2}\log(1-e^{2\pi\sqrt{-1}(\theta_2+\theta_3)})+\frac{\pi\sqrt{-1}}{2(2N+1)}\frac{e^{2\pi\sqrt{-1}(\theta_2+\theta_3)}}{1-e^{2\pi\sqrt{-1}(\theta_2+\theta_3)}}+O\left(\frac{1}{(N+\frac{1}{2})^2}\right). 
    \end{align}
  Expanding  $\varphi_N\left(\theta_2-\theta_3+\frac{1}{N+\frac{1}{2}}\right)$, $\varphi_N(\theta_2+\theta_1)$, $\varphi_N(\theta_2)$ and $\varphi_N(\theta_2-\theta_1)$ similarly,
we obtain
\begin{align}
    &\frac{1}{N+\frac{1}{2}}\left(\varphi_N(\theta_2+\theta_3+\frac{\frac{1}{2}}{N+\frac{1}{2}}-1)+\varphi_{N}(\theta_2-\theta_3+\frac{\frac{1}{2}}{N+\frac{1}{2}})\right.\\\nonumber
    &\left.-\varphi_N(\theta_2+\theta_1)-\varphi_N(\theta_2)-\varphi_N(\theta_2-\theta_1)\right)\\\nonumber
    &=\frac{1}{2\pi\sqrt{-1}}\left(\text{Li}_2(e^{2\pi\sqrt{-1}(\theta_2+\theta_3)})+\text{Li}_2(e^{2\pi\sqrt{-1}(\theta_2-\theta_3)})-\text{Li}_2(e^{2\pi\sqrt{-1}(\theta_2+\theta_1)})\right.\\\nonumber
&\left.-\text{Li}_2(e^{2\pi\sqrt{-1}\theta_2})-\text{Li}_2(e^{2\pi\sqrt{-1}(\theta_2-\theta_1)})\right)-\frac{1}{2}\log(1-e^{2\pi\sqrt{-1}(\theta_2+\theta_3)})\\\nonumber
&-\frac{1}{2}\log(1-e^{2\pi\sqrt{-1}(\theta_2-\theta_3)})+w'_{N}(\theta_1,\theta_2,\theta_3),
\end{align}
where
\begin{align}
    &w'_{N}(\theta_1,\theta_2,\theta_3)\\\nonumber
    &=\frac{\pi\sqrt{-1}}{12}\left(2\frac{e^{2\pi\sqrt{-1}(\theta_2+\theta_3)}}{1-e^{2\pi\sqrt{-1}(\theta_2+\theta_3)}}+2\frac{e^{2\pi\sqrt{-1}(\theta_2-\theta_3)}}{1-e^{2\pi\sqrt{-1}(\theta_2-\theta_3)}}+\frac{e^{2\pi\sqrt{-1}(\theta_2+\theta_1)}}{1-e^{2\pi\sqrt{-1}(\theta_2+\theta_1)}}\right.\\\nonumber
    &\left.+\frac{e^{2\pi\sqrt{-1}(\theta_2+\theta_1)}}{1-e^{2\pi\sqrt{-1}(\theta_2+\theta_1)}}+\frac{e^{2\pi\sqrt{-1}\theta_2}}{1-e^{2\pi\sqrt{-1}\theta_2}}+O\left(\frac{1}{N+\frac{1}{2}}\right)\right).
\end{align}
Let 
\begin{align}
   w_{N}(\theta_1,\theta_2,\theta_3)=  w'_{N}(\theta_1,\theta_2,\theta_3)-\frac{3(p+q)+2}{6},
\end{align}
then we obtain 
the formula (\ref{formula-VNexpansion}).
\end{proof}

We consider the critical point of $\hat{V}(\theta_1,\theta_2,\theta_3)$ which is given by the solution of the following equations

\begin{align} \label{equation-critical1}
\hat{V}_{\theta_1}=\pi\sqrt{-1}((q-2)\theta_1+1)+\log(1-e^{2\pi\sqrt{-1}(\theta_2+\theta_1)})-\log(1-e^{2\pi\sqrt{-1}(\theta_2-\theta_1)})=0,   
\end{align}
\begin{align} \label{equation-critical2}
\hat{V}_{\theta_2}&=-2\pi\sqrt{-1}-\log(1-e^{2\pi\sqrt{-1}(\theta_2+\theta_3)})-\log(1-e^{2\pi\sqrt{-1}(\theta_2-\theta_3)})\\\nonumber
&+\log(1-e^{2\pi\sqrt{-1}(\theta_2+\theta_1)})+\log(1-e^{2\pi\sqrt{-1}\theta_2})+\log(1-e^{2\pi\sqrt{-1}(\theta_2-\theta_1)})=0,    
\end{align}
\begin{align} \label{equation-critical3}
\hat{V}_{\theta_3}&=\pi\sqrt{-1}((4p+2)\theta_3-(2p+3))-\log(1-e^{2\pi\sqrt{-1}(\theta_2+\theta_3)})\\\nonumber
&+\log(1-e^{2\pi\sqrt{-1}(\theta_2-\theta_3)})=0.    
\end{align}

\begin{proposition}  \label{prop-critical}
 The critical point equations (\ref{equation-critical1}), (\ref{equation-critical2}) and (\ref{equation-critical3})  has a unique solution $(\theta_1^0,\theta_2^0,\theta_3^0)$ in the region $D_{0\mathbb{C}}=\{(\theta_1,\theta_2,\theta_3)\in \mathbb{C}^3| (\text{Re}(\theta_1),\text{Re}(\theta_2),\text{Re}(\theta_3))\in D_0\}$.
\end{proposition}
\begin{proof}
    See Appendix \ref{Appendix-criticalexists} for a proof.   
\end{proof}
Now we set 
$\zeta(p,q)$ to be the critical value of the potential function $\hat{V}(p,q;\theta_1,\theta_2,\theta_3)$, i.e. 
\begin{align}
    \zeta(p,q)=V(p,q,\theta_1^0,\theta_2^0,\theta_3^0),
\end{align}
and set
\begin{align} \label{formula-zetaR(p)}
    \zeta_\mathbb{R}(p,q)=\text{Re} \zeta(p,q)=\text{Re} \hat{V}(p,q,\theta_1^0,\theta_2^0,\theta_3^0). 
\end{align}

\begin{lemma} \label{lemma-volumeestimate}
When $(p,q)\in S$, we have the following estimation for $\zeta_{\mathbb{R}}(p,q)$
\begin{align}
    2\pi \zeta_{\mathbb{R}}(p,q)\geq 3.56337.
\end{align}
\end{lemma}
\begin{proof}
    See Appendix \ref{Appendix-geometry} for a proof. 
\end{proof}

Set $z_1=e^{2\pi\sqrt{-1}\theta_1}, z_2=e^{2\pi\sqrt{-1}\theta_2}$, and $z_3=e^{2\pi\sqrt{-1}\theta_3}$, by straightforward computations, we have

\begin{align}
    \hat{V}_{\theta_1\theta_1}&=(q-2)\pi\sqrt{-1}-\frac{2\pi\sqrt{-1}e^{2\pi\sqrt{-1}(\theta_2+\theta_1)}}{1-e^{2\pi\sqrt{-1}(\theta_2+\theta_1)}}-\frac{2\pi\sqrt{-1}e^{2\pi\sqrt{-1}(\theta_2-\theta_1)}}{1-e^{2\pi\sqrt{-1}(\theta_2-\theta_1)}}\\\nonumber
    &=2\pi\sqrt{-1}\left(\frac{q}{2}-1-\frac{1}{z_1/z_2-1}-\frac{1}{1/(z_1z_2)-1}\right)
\end{align}
\begin{align}
    \hat{V}_{\theta_1\theta_2}&=-\frac{2\pi\sqrt{-1}e^{2\pi\sqrt{-1}(\theta_2+\theta_1)}}{1-e^{2\pi\sqrt{-1}(\theta_2+\theta_1)}}+\frac{2\pi\sqrt{-1}e^{2\pi\sqrt{-1}(\theta_2-\theta_1)}}{1-e^{2\pi\sqrt{-1}(\theta_2-\theta_1)}}\\\nonumber
    &=2\pi\sqrt{-1}\left(\frac{1}{z_1/z_2-1}-\frac{1}{1/(z_1z_2)-1}\right),
\end{align}
\begin{align}
    \hat{V}_{\theta_1\theta_3}=0.
\end{align}

\begin{align}
    \hat{V}_{\theta_2\theta_2}&=\frac{2\pi\sqrt{-1}e^{2\pi\sqrt{-1}(\theta_2+\theta_3)}}{1-e^{2\pi\sqrt{-1}(\theta_2+\theta_3)}}+\frac{2\pi\sqrt{-1}e^{2\pi\sqrt{-1}(\theta_2-\theta_3)}}{1-e^{2\pi\sqrt{-1}(\theta_2-\theta_3)}}-\frac{2\pi\sqrt{-1}e^{2\pi\sqrt{-1}(\theta_2+\theta_1)}}{1-e^{2\pi\sqrt{-1}(\theta_2+\theta_1)}}\\\nonumber
    &-\frac{2\pi\sqrt{-1}e^{2\pi\sqrt{-1}\theta_2}}{1-e^{2\pi\sqrt{-1}\theta_2}}-\frac{2\pi\sqrt{-1}e^{2\pi\sqrt{-1}(\theta_2-\theta_1)}}{1-e^{2\pi\sqrt{-1}(\theta_2-\theta_1)}}\\\nonumber
    &=2\pi\sqrt{-1}\left(-\frac{1}{z_1/z_2-1}-\frac{1}{1/z_2-1}-\frac{1}{1/(z_1z_2)-1}\right.\\\nonumber
    &\left.+\frac{1}{1/(z_2z_3)-1}+\frac{1}{z_3/z_2-1}\right).
\end{align}
\begin{align}
    \hat{V}_{\theta_2\theta_3}&=\frac{2\pi\sqrt{-1}e^{2\pi\sqrt{-1}(\theta_2+\theta_3)}}{1-e^{2\pi\sqrt{-1}(\theta_2+\theta_3)}}-\frac{2\pi\sqrt{-1}e^{2\pi\sqrt{-1}(\theta_2-\theta_3)}}{1-e^{2\pi\sqrt{-1}(\theta_2-\theta_3)}}\\\nonumber
    &=2\pi\sqrt{-1}\left(\frac{1}{1/(z_2z_3)-1}-\frac{1}{z_3/z_2-1}\right).
\end{align}

\begin{align}
    \hat{V}_{\theta_3\theta_3}&=\pi\sqrt{-1}(4p+2)+\frac{2\pi\sqrt{-1}e^{2\pi\sqrt{-1}(\theta_2+\theta_3)}}{1-e^{2\pi\sqrt{-1}(\theta_2+\theta_3)}}+\frac{2\pi\sqrt{-1}e^{2\pi\sqrt{-1}(\theta_2-\theta_3)}}{1-e^{2\pi\sqrt{-1}(\theta_2-\theta_3)}}\\\nonumber
    &=2\pi\sqrt{-1}\left((2p+1)+\frac{1}{1/(z_2z_3)-1}+\frac{1}{z_3/z_2-1}\right).
\end{align}

We let
\begin{align}
 \hat{f}(\theta_1,X_1,\theta_2,X_2,\theta_3,X_3)=\text{Re}\hat{V}(p,q,\theta_1+\sqrt{-1}X_1,\theta_2+\sqrt{-1}X_2,\theta_3+\sqrt{-1}X_3),  
\end{align}
then for $i\in \{1,2,3\}$, 
\begin{align}
\hat{f}_{X_i}&=\text{Re}(\sqrt{-1}\hat{V}(p,q,\theta_1+\sqrt{-1}X_1,\theta_2+\sqrt{-1}X_i,\theta_3+\sqrt{-1}X_3)\\\nonumber
&=-\text{Im}\hat{V}_{\theta_i}(p,q,\theta_1+\sqrt{-1}X_1,\theta_2+\sqrt{-1}X_2,\theta_3+\sqrt{-1}X_3)
\end{align}
and
\begin{align}
\hat{f}_{X_iX_i}=-\text{Im}\left(\sqrt{-1}\hat{V}_{\theta_i\theta_i}(p,q,\theta_1+\sqrt{-1}X,\theta_2+\sqrt{-1}X_2,\theta_3+\sqrt{-1}X_3)\right).    
\end{align}

Hence the Hessian matrix of $\hat{f}$ is given by 
\begin{align}
   Hess(\hat{f})&=\begin{pmatrix}
     \hat{f}_{X_1X_1} & \hat{f}_{X_1X_2} & \hat{f}_{X_1X_3} \\
     \hat{f}_{X_2X_1} & \hat{f}_{X_2X_2} & \hat{f}_{X_2X_3} \\ 
     \hat{f}_{X_3X_1} & \hat{f}_{X_2X_2} & \hat{f}_{X_3X_3}
    \end{pmatrix}\\\nonumber
    &=2\pi \text{Im}\begin{pmatrix}
   a+c   & -a+c & 0 \\
     -a+c & a+b+c+d+e & d-e \\ 
     0 & d-e & d+e
    \end{pmatrix}
\end{align}
where 
\begin{align}
a&=-\text{Im}\frac{1}{1-z_2/z_1}=-\frac{\sin(2\pi(\theta_2-\theta_1))}{e^{2\pi(X_2-X_1)}+e^{-2\pi(X_2-X_1)}-2\cos(2\pi(\theta_2-\theta_1))} \\\nonumber
b&=-\text{Im}\frac{1}{1-z_2}=-\frac{\sin(2\pi \theta_2)}{e^{2\pi X_2}+e^{-2\pi X_2}-2\cos(2\pi \theta_2)}  \\\nonumber
c&=-\text{Im}\frac{1}{1-z_1z_2}=-\frac{\sin(2\pi( \theta_2+\theta_1))}{e^{2\pi (X_2+X_1)}+e^{-2\pi (X_2+X_1)}-2\cos(2\pi (\theta_2+\theta_1))} \\\nonumber
d&=\text{Im}\frac{1}{1-z_2z_3}=\frac{\sin(2\pi (\theta_2+\theta_3))}{e^{2\pi (X_2+X_3)}+e^{-2\pi (X_2+X_3)}-2\cos(2\pi (\theta_2+\theta_3))}\\\nonumber
e&=\text{Im}\frac{1}{1-z_2/z_3}=\frac{\sin(2\pi (\theta_2-\theta_3))}{e^{2\pi (X_2-X_3)}+e^{-2\pi (X_2-X_3)}-2\cos(2\pi (\theta_2-\theta_3))}.  
\end{align}

Obviously, when $\frac{1}{2}<\theta_2-\theta_1,\theta_2,\theta_2+\theta_1<1$, $1<\theta_2+\theta_3<\frac{3}{2}$ and $0<\theta_2-\theta_3<\frac{1}{2}$, we have that $a,b,c,d,e>0$.

Set 
\begin{align}
    D_{H}=\{(\theta_1,\theta_2,\theta_3)\in D_0| \frac{1}{2}<\theta_2\pm \theta_1,\theta_2<1, 1<\theta_2+\theta_3<\frac{3}{2}, 0<\theta_2-\theta_3<\frac{1}{2}\}.
\end{align}
We obtain
\begin{proposition} \label{proposition-Hessianf}
 For $(\theta_1,\theta_2,\theta_3)\in D_{H}$, the Hessian matrix
\begin{align}
    Hess(\hat{f})
\end{align}
is positive definite. 
\end{proposition}
\begin{proof}
Since
    \begin{align}
&    \begin{vmatrix}
   a+c   & -a+c & 0 \\
     -a+c & a+b+c+d+e & d-e \\ 
     0 & d-e & d+e
\end{vmatrix}\\\nonumber
&=abd+abe+4acd+4ace+4ade+cbd+cbe+4cde>0,
\end{align}
\begin{align}
&    \begin{vmatrix}
   a+c   & -a+c \\
     -a+c & a+b+c+d+e 
\end{vmatrix}\\\nonumber
&=ab+4ac+ad+ae+cb+cd+ce>0
\end{align}
and $a+c>0$ by $a,b,c,d,e>0$ when $(\theta_1,\theta_2,\theta_3)\in D_{H}$. 
\end{proof}

\subsection{The case $m_2\neq 0$}

We say that a Fourier coefficient $\hat{h}_{N}(m_1,m_2,m_3)$ can be neglected if exists some $\epsilon>0$, such that
\begin{align}
    \hat{h}_{N}(m_1,m_2,m_3)=O(e^{(N+\frac{1}{2})(\zeta_{\mathbb{R}}(p,q)-\epsilon)}). 
\end{align}

We introduce
\begin{align}
&\hat{V}(\theta_1,\theta_2,\theta_3,m_1,m_2,m_3)=\hat{V}(\theta_1,\theta_2,\theta_3)-2\pi\sqrt{-1}\sum_{i=1}^3m_i\theta_i\\\nonumber
&=\pi\sqrt{-1}\left(\frac{3q}{2}+\left(\frac{q}{2}-1\right)\theta_1^2-(2m_1+1)\theta_1+2(m_2+1)\theta_2+(2p+1)\theta_3^2-(2p+3+2m_3)\theta_3\right)\\\nonumber
&+\frac{1}{2\pi\sqrt{-1}}\left(\frac{\pi^2}{6}+\text{Li}_2(e^{2\pi\sqrt{-1}(\theta_2+\theta_3)})+\text{Li}_2(e^{2\pi\sqrt{-1}(\theta_2-\theta_3)})-\text{Li}_2(e^{2\pi\sqrt{-1}(\theta_2+\theta_1)})\right.\\\nonumber
&\left.-\text{Li}_2(e^{2\pi\sqrt{-1}\theta_2})-\text{Li}_2(e^{2\pi\sqrt{-1}(\theta_2-\theta_1)})\right).
\end{align}

First, when $m_2=-1$, it is easy to see that
\begin{align} \label{formula-ReV-1}
&\lim_{X_2\rightarrow +\infty}\text{Re}\hat{V}
(\theta_1+\sqrt{-1}X_1,\theta_2+\sqrt{-1}X_2,\theta_3+\sqrt{-1}X_3;m_1,m_2,m_3)|_{X_3=X_1=0}\\\nonumber
&=\lim_{X_2\rightarrow +\infty}\text{Re}\left(\frac{1}{2\pi\sqrt{-1}}\left(\text{Li}_2(e^{2\pi\sqrt{-1}(\theta_2+\theta_3+\sqrt{-1}X_2)})+\text{Li}_2(e^{2\pi\sqrt{-1}(\theta_2-\theta_3+\sqrt{-1}X_2)})\right.\right.\\\nonumber &\left.\left.
-\text{Li}_2(e^{2\pi\sqrt{-1}(\theta_2-\theta_1+\sqrt{-1}X_2)})-\text{Li}_2(e^{2\pi\sqrt{-1}(\theta_2+\sqrt{-1}X_2)})-\text{Li}_2(e^{2\pi\sqrt{-1}(\theta_2+\theta_1+\sqrt{-1}X_2)})\right)\right)\\\nonumber
&=0.
\end{align}

It follows that the term
$
    \hat{h}_{N}(m_1,-1,m_3)
$
can be neglected for any $m_1,m_3\in \mathbb{Z}$.

Next, we consider the case $m_2\neq-1$. Motivated by Lemma \ref{lemma-Li2}, for $(\theta_1,\theta_2,\theta_3)\in D_{0}$,  we introduce the following  function
\begin{equation} \label{eq:2}
F(X_1,X_2,X_3;m_1,m_2,m_3)=\left\{ \begin{aligned}
         &0  &  \ (\text{if} \ X_2+X_3\geq 0) \\
         &\left(\theta_2+\theta_3-\frac{3}{2}\right)(X_2+X_3) & \ (\text{if} \ X_2+X_3<0)
                          \end{aligned} \right.
                          \end{equation}
\begin{equation*}
 +\left\{ \begin{aligned}
         &0  &  \ (\text{if} \ X_2-X_3\geq 0) \\
         &\left(\theta_2-\theta_3-\frac{1}{2}\right)(X_2-X_3) & \ (\text{if} \ X_2-X_3<0)
                          \end{aligned} \right.
                          \end{equation*}
\begin{equation*}
-\left\{ \begin{aligned}
         &0  &  \ (\text{if} \ X_2\geq 0) \\
         &\left(\theta_2-\frac{1}{2}\right)X_2 & \ (\text{if} \ X_2<0)
                          \end{aligned} \right.  
                          \end{equation*}
\begin{equation*}
-\left\{ \begin{aligned}
         &0  &  \ (\text{if} \ X_2+X_1\geq 0) \\
         &\left(\theta_2+\theta_1-\frac{1}{2}\right)(X_2+X_1) & \ (\text{if} \ X_2+X_1<0)
                          \end{aligned} \right.  
                          \end{equation*}  
\begin{equation*}
-\left\{ \begin{aligned}
         &0  &  \ (\text{if} \ X_2-X_1\geq 0) \\
         &\left(\theta_2-\theta_1-\frac{1}{2}\right)(X_2-X_1) & \ (\text{if} \ X_2-X_1<0)
                          \end{aligned} \right.  
                          \end{equation*}                           

\begin{equation*}
    +\left(-\left(\frac{q}{2}-1\right)\theta_1+m_1+\frac{1}{2}\right)X_1+(m_2+1)X_2+\left(p+\frac{3}{2}+m_3-(2p+1)\theta_3\right)X_3
\end{equation*}                       
where we use $\theta_2+\theta_3-\frac{3}{2}$ instead of $\theta_2+\theta_3-\frac{1}{2}$ in the first summation since in our situation $1< \theta_2+\theta_3<2$.

The following simple observation is useful. 
\begin{lemma} \label{lemma-X1X2X3}
Given a linear function on  $F(X_1,X_2,X_3)=AX_1+BX_2+CX_3$ over $\mathbb{R}^3$ with $X_1\geq X_2\geq X_3\geq 0$,  we have

(1) If $A<0$ or $A+B<0$ or $A+B+C<0$, then we can find a direction  with  $X_1^2+X_2^2+X_3^2\rightarrow \infty$, such that $F(X_1,X_2,X_3)\rightarrow -\infty$; 

(2)  If $A>0$ and $A+B>0$  and $A+B+C>0$, then we have  
$F(X_1,X_2,X_3)\rightarrow +\infty$ as $X_1^2+X_2^2+X_3^2\rightarrow +\infty$.
\end{lemma}

Then we have
\begin{proposition} \label{prop-inequalities26}
If $(\theta_1,\theta_2,\theta_3)\in D_0$ satisfies the following 26 inequalities: 
\begin{align}
&(1): m_2+1>0 \nonumber \\\nonumber 
&(2): (2p+1)\theta_3<p+m_2+m_3+\frac{5}{2}\\\nonumber
&(3): 2p\theta_3+\theta_2<p+m_3+2\\\nonumber
&(4): (2p-1)\theta_3-\theta_2<p-m_2+m_3\\\nonumber
&(5): \theta_2>m_2+\frac{1}{2}\\\nonumber
&(6): (2p-1)\theta_3+\theta_2>p+m_2+m_3+1\\\nonumber
&(7):   2p\theta_3-\theta_2>p+m_3 \\\nonumber 
&(8): (2p+1)\theta_3>p-m_2+m_3+\frac{1}{2} \\\nonumber
&(9):  (2p+1)\theta_3+(\frac{q}{2}-1)\theta_1<p+m_2+m_3+m_1+3
\end{align}    
\begin{align}
&(10): (\frac{q}{2}-1)\theta_1<m_2+m_1+\frac{3}{2} \nonumber \\\nonumber
&(11): \theta_2-\frac{q}{2}\theta_1>-m_1\\\nonumber
&(12): (2p+1)\theta_3-(\frac{q}{2}-1)\theta_1<p+m_2+m_3-m_1+2\\\nonumber
&(13): (\frac{q}{2}-1)\theta_1>-m_2+m_1-\frac{1}{2}\\\nonumber
&(14): \theta_2+\frac{q}{2}\theta_1>m_1+1 \\\nonumber
&(15): 2p\theta_3+\frac{q}{2}\theta_1<p+m_3+m_1+2\\\nonumber
&(16): 2p\theta_3-\frac{q}{2}\theta_1<p+m_3-m_1+1\\\nonumber
&(17): (2p-1)\theta_3-\theta_2+(\frac{q}{2}+1)\theta_1<p-m_2+m_3+m_1+\frac{1}{2}\\\nonumber
&(18): (2p-1)\theta_3-\theta_2-(\frac{q}{2}+1)\theta_1<p-m_2+m_2-m_1-\frac{1}{2}\\\nonumber
\end{align}
\begin{align}
&(19): \theta_2-(\frac{q}{2}+1)\theta_1>m_2-m_1\nonumber\\\nonumber
&(20): \theta_2+(\frac{q}{2}+1)\theta_1>m_2+m_1+1\\\nonumber
&(21): (2p-1)\theta_3+\theta_2-(\frac{q}{2}+1)\theta_1>p+m_2+m_3-m_1+\frac{1}{2}\\\nonumber
&(22): (2p-1)\theta_3+\theta_2+(\frac{q}{2}+1)\theta_1>p+m_2+m_3+m_1+\frac{3}{2}\\\nonumber
&(23): 2p\theta_3-\frac{q}{2}\theta_1>p+m_3-m_1\\\nonumber
&(24): 2p\theta_3+\frac{q}{2}\theta_1>p+m_3+m_1+1\\\nonumber
&(25): (2p+1)\theta_3-(\frac{q}{2}-1)\theta_1>p-m_1+m_3-m_1\\\nonumber
&(26): (2p+1)\theta_3+(\frac{q}{2}-1)\theta_1>p-m_2+m_3+m_1+1
\end{align}
then we have $F(X_1,X_2,X_3,m_1,m_2,m_3)\rightarrow +\infty$ as $X_1^2+X_2^2+X_3^2\rightarrow +\infty$. 
\end{proposition}

We prove the proposition by subdividing $(X_1,X_2,X_3)\in \mathbb{R}^3$ into 48 regions. 

Case I: $X_2\geq X_3\geq 0$, 

I-A: $X_2\geq X_3\geq X_1\geq 0$, then 
\begin{align}
 &F(X_1,X_2,X_3;m_1,m_2,m_3)\\\nonumber
 &=(m_2+1)X_2+(p+m_3+\frac{3}{2}-(2p+1)\theta_3)X_3+(m_1+\frac{1}{2}-(\frac{q}{2}-1)\theta_1)X_1   
\end{align}
By Lemma \ref{lemma-X1X2X3}, if 
\begin{align}
m_1+1>0, \  \ (2p+1)\theta_3<p+m_2+m_3+\frac{5}{2},    
\end{align}
and 
\begin{align}
(2p+1)\theta_3+(\frac{q}{2}-1)\theta_1<p+m_2+m_3+m_1+3    
\end{align}
then we have $F(X_1,X_2,X_3)\rightarrow +\infty$ as $X_1^2+X_2^2+X_3^2\rightarrow +\infty$. 

I-B: $X_2\geq X_1\geq X_3\geq 0$, 
\begin{align}
 &F(X_1,X_2,X_3;m_1,m_2,m_3)\\\nonumber
 &=(m_2+1)X_2+(m_1+\frac{1}{2}-(\frac{q}{2}-1)\theta_1)X_1   +(p+m_3+\frac{3}{2}-(2p+1)\theta_3)X_3
\end{align}
By Lemma \ref{lemma-X1X2X3}, if 
\begin{align}
m_2+1>0, \ (\frac{q}{2}-1)\theta_1<m_2+m_1+\frac{3}{2}    
\end{align}
and 
\begin{align}
(2p+1)\theta_3+(\frac{q}{2}-1)\theta_1<p+m_2+m_3+m_1+3    
\end{align}
then we have $F(X_1,X_2,X_3)\rightarrow +\infty$ as $X_1^2+X_2^2+X_3^2\rightarrow +\infty$. 

I-C: $X_2\geq X_1\geq X_3\geq 0$, 
\begin{align}
 &F(X_1,X_2,X_3;m_1,m_2,m_3)\\\nonumber
 &=(m_1+\theta_2-\frac{q}{2}\theta_1)X_1+(m_2+\frac{3}{2}-\theta_2+\theta_1)X_2+(p+m_3+\frac{3}{2}-(2p+1)\theta_3)X_3   
\end{align}
By Lemma \ref{lemma-X1X2X3}, if 
\begin{align}
m_1+\theta_2-\frac{q}{2}\theta_1>0, \ (\frac{q}{2}-1)\theta_1<m_2+m_1+\frac{3}{2}    
\end{align}
and 
\begin{align}
(2p+1)\theta_3+(\frac{q}{2}-1)\theta_1<p+m_2+m_3+m_1+3    
\end{align}
then we have $F(X_1,X_2,X_3)\rightarrow +\infty$ as $X_1^2+X_2^2+X_3^2\rightarrow +\infty$. 

I-D: $X_2\geq X_3\geq -X_1\geq 0$, 
\begin{align}
 &F(X_1,X_2,X_3;m_1,m_2,m_3)\\\nonumber
 &=(m_2+1)X_2+(p+m_3+\frac{3}{2}-(2p+1)\theta_3)X_3+(-m_1-\frac{1}{2}+(\frac{q}{2}-1)\theta_1)(-X_1)   
\end{align}
By Lemma \ref{lemma-X1X2X3}, if 
\begin{align}
m_2+1>0, \ (2p+1)\theta_3<p+m_2+m_3+\frac{5}{2}    
\end{align}
and 
\begin{align}
(2p+1)\theta_3-(\frac{q}{2}-1)\theta_1<p+m_2+m_3-m_1+2    
\end{align}
then we have $F(X_1,X_2,X_3)\rightarrow +\infty$ as $X_1^2+X_2^2+X_3^2\rightarrow +\infty$. 

I-E: $X_2\geq -X_1\geq X_3\geq 0$, 
\begin{align}
 &F(X_1,X_2,X_3;m_1,m_2,m_3)\\\nonumber
 &=(m_2+1)X_2+(-m_1-\frac{1}{2}+(\frac{q}{2}-1)\theta_1)(-X_1)+(p+m_3+\frac{3}{2}-(2p+1)\theta_3)X_3   
\end{align}
By Lemma \ref{lemma-X1X2X3}, if 
\begin{align}
m_2+1>0, \ (\frac{q}{2}-1)\theta_1>-m_2+m_1-\frac{1}{2}    
\end{align}
and 
\begin{align}
(2p+1)\theta_3-(\frac{q}{2}-1)\theta_1<p+m_2+m_3-m_1+2    
\end{align}
then we have $F(X_1,X_2,X_3)\rightarrow +\infty$ as $X_1^2+X_2^2+X_3^2\rightarrow +\infty$. 

I-F: $-X_1\geq X_2\geq X_3\geq 0$, 
\begin{align}
 &F(X_1,X_2,X_3;m_1,m_2,m_3)\\\nonumber
 &=(-m_1-1+\theta_2+\frac{q}{2}\theta_1)(-X_1)+(m_2+\frac{3}{2}-\theta_2-\theta_1)X_2+(p+m_3+\frac{3}{2}-(2p+1)\theta_3)X_3   
\end{align}
By Lemma \ref{lemma-X1X2X3}, if 
\begin{align}
\theta_2+\frac{q}{2}\theta_1>m_1+1, \ (\frac{q}{2}-1)\theta_1>-m_2+m_1-\frac{1}{2}    
\end{align}
and 
\begin{align}
(2p+1)\theta_3-(\frac{q}{2}-1)\theta_1<p+m_2+m_3-m_1+2    
\end{align}
then we have $F(X_1,X_2,X_3)\rightarrow +\infty$ as $X_1^2+X_2^2+X_3^2\rightarrow +\infty$. 

Hence from the above 6 cases I-A,B,C,D, E and F, we obtain $8$ inequalities: 
\begin{align}
    & m_1+1>0, \\\nonumber 
    & (2p+1)\theta_3<p+m_2+m_3+\frac{5}{2} \\\nonumber
    & (2p+1)\theta_3+(\frac{q}{2}-1)\theta_1<p+m_2+m_3+m_1+3\\\nonumber
    & (\frac{q}{2}-1)\theta_1<m_2+m_1+\frac{3}{2} \\\nonumber
    & \theta_2-\frac{q}{2}\theta_1>-m_1\\\nonumber
    & (2p+1)\theta_1-(\frac{q}{2}-1)\theta_2<p+m_1+m_3-m_1+2 \\\nonumber
    & (\frac{q}{2}-1)\theta_1>-m_2+m_1-\frac{1}{2}\\\nonumber
    &\theta_2+\frac{q}{2}\theta_1>m_1+1
\end{align}

Similarly, by considering the other $7$ cases II:$Y\geq X\geq 0$, III: $Y\geq -X\geq 0$, IV: $-X\geq Y\geq 0$, V: $-X\geq -Y\geq 0$, VI: $-Y\geq -X\geq 0$, VII:$-Y\geq X\geq 0$, and VIII: $X\geq -Y\geq 0$,  we finally obtain the 26 inequalities in Proposition \ref{prop-inequalities26}.

From the inequalities (1) and (5), we have 
\begin{align}
   -1<m_2<\theta_2-\frac{1}{2}<\frac{1}{2} 
\end{align}
since $\theta_2<1$. Hence for any $(\theta_1,\theta_2,\theta_3)\in D_0$, when $m_2< -1$ or $m_2\geq 1$, $F(X_1,X_2,X_3;m_1,m_2,m_3)$ goes to $-\infty$ in some direction of $X_1^2+X_2^2+X_3^2\rightarrow +\infty$. 

In conclusion, when $m_2\neq 0$, the Fourier coefficients 
$
\hat{h}_{N}(m_1,m_2,m_3)    
$
 can be neglected. 
In the following, we consider the Fourier coefficient $\hat{h}_{N}(m_1,m_2,m_3)$ with $m_2=0$.
\subsection{The case $m_2=0$ but  $(m_1,m_3)\neq (0,0)$}
\begin{theorem} \label{theorem-m_1}
  When $m_1=0$ and $m_3\geq 1$, then 
\begin{align}
\hat{h}_{N}(0,0,m_3)    
\end{align}
 can be neglected.
\end{theorem}
\begin{proof}
For $c_3\geq \frac{1}{2}+\frac{7}{8p}$, then 
$\tilde{c}_3\geq \frac{7}{8p}$, by Lemmas \ref{theorem-inprincipal1} and \ref{theorem-inprincipal2},  we have
\begin{align}
    &2\pi \text{Re}\hat{V}(p,q; \theta_1,\theta_2,c_3=\frac{1}{2}+\tilde{c}_3;0,0,m_3)\\\nonumber
    &\leq 2\pi\text{Re}\hat{V}\left(p,q,\theta_1\left(\frac{1}{2}\right),\theta_2\left(\frac{1}{2}\right),\frac{1}{2}\right)-\pi^2 \tilde{c}_3^2\\\nonumber
    &\leq 2\pi\text{Re}\hat{V}\left(p,q,\theta_1\left(\frac{1}{2}\right),\theta_2\left(\frac{1}{2}\right),\frac{1}{2}\right)-\pi^2\frac{49}{64}\frac{1}{p^2}\\\nonumber
    &< Vol\left(M_{p,q}\right).  
\end{align}
Then the integral over the region  
$\{(\theta_1,\theta_2,\theta_3)\in D_0|\theta_3\leq \frac{1}{2}+\frac{7}{8p}\}$
can be  neglected by using Theorem \ref{theorem-usedinm1}.

Moreover, when $m_3\geq 1$,  from the inequalities (6) and (7), we obtain 
\begin{align}
(4p-1)\theta_3>2p+2m_3+1    
\end{align}
Hence
\begin{align}
\theta_3>\frac{1}{2}+\frac{4m_3+3}{2(4p-1)}>\frac{1}{2}+\frac{7}{8p}.    
\end{align}

Therefore, over the region $\{(\theta_1,\theta_2,\theta_3)\in D_0|\theta_3\leq \frac{1}{2}+\frac{7}{8p}\}$, we can find a direction of 
$X_1^2+X_2^2+X_3^2\rightarrow +\infty $ such that 
\begin{align}
\text{Re}\hat{V}(p,q;0,0,m_3,\theta_1,\theta_2,\theta_3)\rightarrow -\infty.     
\end{align}
It follows that the integral over this region can be neglected.   
Therefore, the whole integral can be neglected. 
\end{proof}

\begin{theorem} \label{theorem-m3}
  When $m_1\geq 1$ and $m_3=0$,  then 
\begin{align}
\hat{h}_{N}(m_1,0,0)    
\end{align}
  can be neglected.
\end{theorem}
\begin{proof} 
For $\frac{2}{q}\leq c_1<\frac{1}{4}$, by Lemmas \ref{theorem-inprincipal3} and \ref{theorem-inprincipal4},  we have
\begin{align} \label{formula-inequality}
    2\pi \text{Re}\hat{V}(p,q; c_1,\theta_2,\theta_3;m_1,0,0)&\leq Vol\left(S^3\setminus \mathcal{K}_p\right)-\frac{3}{2}\pi^2 c_1^2\\\nonumber
    &\leq Vol\left(S^3\setminus \mathcal{K}_p\right)-6\pi^2\frac{1}{q^2}\\\nonumber
    &<Vol\left(M_{p,q}\right).  
\end{align}
Then the integral over the region  
$\{(\theta_1,\theta_2,\theta_3)\in D_0|\theta_1> \frac{2}{q}\}$ can be  neglected by using Theorem \ref{theorem-usedinm3}.

Moreover, when $m_1\geq 1$,  from the inequalities (14), we obtain 
\begin{align}
\theta_1>\frac{2(2-\theta_2)}{q}>\frac{2}{q},
\end{align}

Therefore, over the region
$\{(\theta_1,\theta_2,\theta_3)\in D_0|\theta_1\leq \frac{2}{q}\}$, we can find a direction of 
$X_1^2+X_2^2+X_3^2\rightarrow +\infty $ such that 
\begin{align}
\text{Re}\hat{V}(p,q;\theta_1,\theta_2,\theta_3;m_1,0,0)\rightarrow -\infty.     
\end{align}
It follows that the integral over this region can be neglected. 

Therefore, the whole integral can be neglected. 
\end{proof}

\begin{theorem}
  When $m_1\geq 1$ and $m_3\geq 1$, then 
\begin{align}
\hat{h}_{N}(m_1,0,m_3)    
\end{align}
 can be neglected.
\end{theorem}
\begin{proof}
 First,   when $m_1\geq 1$,  from the inequalities (14), we obtain 
\begin{align}
\theta_1>\frac{2(2-\theta_2)}{q}>\frac{2}{q},
\end{align}

Therefore, over the region
when $\theta_1< \frac{2}{q}$, we can find a direction of 
$X_1^2+X_2^2+X_3^2\rightarrow +\infty $ such that 
\begin{align}
\text{Re}\hat{V}(p,q;\theta_1,\theta_2,\theta_3;m_1,0,m_3)\rightarrow -\infty.     
\end{align}

It follows that the integral over the region $\{\theta_1< \frac{2}{q}\}$ can be neglected. 

For $\frac{2}{q}\leq c_1\leq \frac{1}{4}$, on the region $D_0(c_1)=\{\theta_1=c_1\} \cap D_0$, we estimate  the integral 
\begin{align}
    \int_{D_0(c_1)}\sin(\pi c_1)\sin(2\pi\theta_3)e^{(N+\frac{1}{2})\hat{V}_N(c_1,\theta_2,\theta_3;m_1,0,m_3)}d\theta_2d\theta_3
\end{align}
as follows. 

For $\theta_3<\frac{1}{2}+\frac{7}{8p}$, then
$\theta_3<\frac{1}{2}+\frac{7}{2(4p-1)}$, in this case,  we can find a direction of 
$X_2^2+X_3^2\rightarrow +\infty $ such that 
\begin{align}
\text{Re}\hat{V}(p,q;c_1,\theta_2,\theta_3;m_1,0,m_3)\rightarrow -\infty.     
\end{align}
So the integral over $\theta_3<\frac{1}{2}+\frac{7}{8p}$ can be neglected.  

For $\theta_3\geq \frac{1}{2}+\frac{7}{8p}$, we use the one-dimensional saddle point method.  Let $\theta_3=c_3$, 
for $c_3\geq \frac{1}{2}+\frac{7}{8p}$, then by Lemma \ref{Lemma-inequ} and formula (\ref{formula-inequality}), we have 
\begin{align}
 &2\pi \text{Re}\hat{V}(p,q;m_1,0,m_3,c_1,T_2(c_1,\frac{1}{2}+\tilde{c}_3),c_3=\frac{1}{2}+\tilde{c}_3)\\\nonumber
 &<2\pi \text{Re} \hat{V}(p,q;m_1,0,0,c_1,\theta_2(c_1),\theta_3(c_1))\\\nonumber
 &<Vol\left(M_{p,q}\right).
\end{align}
By using the one-dimensional saddle method, see Appendix \ref{appendix-onedim}, 
we can show that integral over the region $\theta_3\geq \frac{1}{2}+\frac{7}{8p}$ can also be neglected. 

Hence the whole integral can be neglected. 
\end{proof}

Therefore, 
\begin{proposition}
There exists $\epsilon>0$, such that
\begin{align} \label{formula-propRT}
RT_{r}(M_{p,q})=4\hat{h}_{N}(0,0,0)+O(e^{(N+\frac{1}{2})(\zeta_{\mathbb{R}}(p,q)-\epsilon)}).     
\end{align}
\end{proposition}
Finally, we only need to estimate the integral $\hat{h}_{N}(0,0,0)$.

\subsection{The contribution of $\hat{h}_{N}(0,0,0)$}

We let $U_0'$ be the region determined by inequalities illustrated in Proposition \ref{prop-inequalities26}  when $m_1=m_2=m_3=0$. Let $U_0=U'_0\cap D_0$, as a consequence of Proposition \ref{prop-inequalities26}, 

\begin{corollary}\label{coro-U0noU0}
For any $(\theta_1,\theta_2,\theta_3)\in D_0$, 

(i) when $(\theta_1,\theta_2,\theta_3)\in U_0$, we have 
\begin{align}
    f(X_1,X_2,X_3)\rightarrow +\infty \ \text{as } \ X_1^2+X_2^2+X_3^2\rightarrow +\infty.
\end{align}

(ii) when $(\theta_1,\theta_2,\theta_3)\notin U_0$, 
\begin{align}
     f(X_1,X_2,X_3)\rightarrow -\infty \ \text{in some directions of} \  X_1^2+X_2^2+X_3^2\rightarrow +\infty.
\end{align}
\end{corollary}

We introduce the cube MCB as shown in figure \ref{figure-cube}, the integrand is symmetric with respect to $\theta_3=\frac{1}{2}$, $\theta_1=0$ and by Lemma \ref{lemma-regionD'0}, 
we have
\begin{figure}[!htb] \label{figure-cube} 
\begin{align*} 
\raisebox{-15pt}{
\includegraphics[width=450 pt]{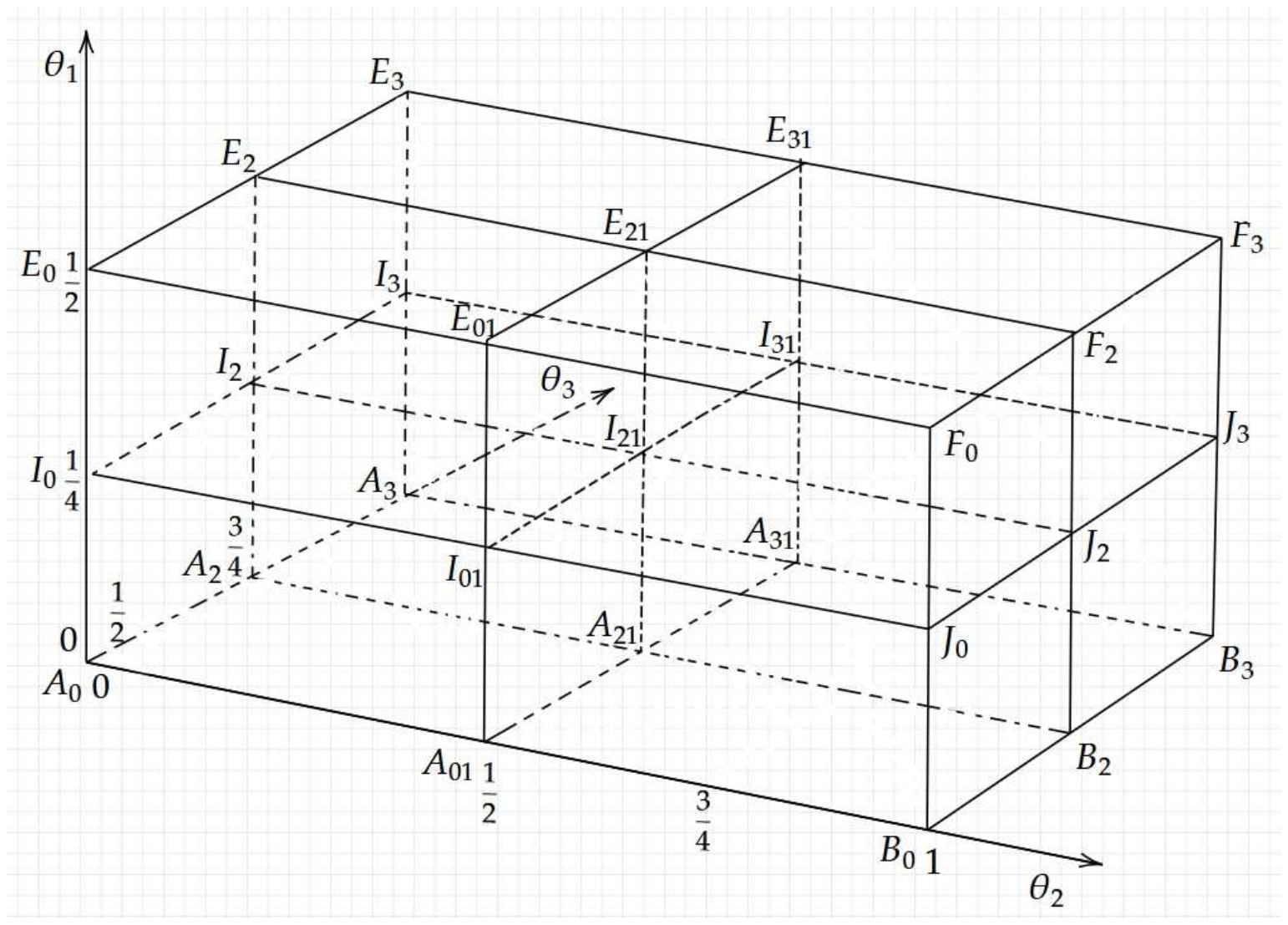}}.
\end{align*}
\caption{The cube MCB}
\end{figure}

\begin{align}
 &\int_{D_0}\sin(\pi \theta_1)\sin(2\pi\theta_3)e^{(N+\frac{1}{2})\hat{V}_N(\theta_1,\theta_2,\theta_3)}d\theta_1d\theta_2d\theta_3\\\nonumber
 &=4\int_{MCB}\sin(\pi \theta_1)\sin(2\pi\theta_3)e^{(N+\frac{1}{2})\hat{V}_N(\theta_1,\theta_2,\theta_3)}d\theta_1d\theta_2d\theta_3 +O(e^{\frac{N+\frac{1}{2}}{2\pi }(3.56337-\epsilon)})   
\end{align}

So we only need to consider the integral over the cube MCB 
\begin{align}
 \int_{MCB}\sin(\pi \theta_1)\sin(2\pi\theta_3)e^{(N+\frac{1}{2})\hat{V}_N(\theta_1,\theta_2,\theta_3)}d\theta_1d\theta_2d\theta_3.    
\end{align}

Consider the function 
\begin{align}
  v(\theta_1,\theta_3,\theta_3)&=\text{Re}\hat{V}(p,q,\theta_1,\theta_2,\theta_3)\\\nonumber
  &=\Lambda(\theta_2+\theta_3)+\Lambda(\theta_2-\theta_3)-\Lambda(\theta_2+\theta_1)-\Lambda(\theta_2)-\Lambda(\theta_2-\theta_1). 
\end{align}

we have
\begin{align} \label{formula-vtheta1}
\frac{dv(\theta_1,\theta_2,\theta_3)}{d\theta_1}=\Lambda'(\theta_2-\theta_1)-\Lambda'(\theta_2+\theta_1)=2\log \frac{|\sin \pi(\theta_2+\theta_1)|}{|\sin \pi(\theta_2-\theta_1)|}<0    
\end{align}
and 
\begin{align} \label{formula-vtheta2}
\frac{dv(\theta_1,\theta_2,\theta_3)}{d\theta_3}=\Lambda'(\theta_2+\theta_3)-\Lambda'(\theta_2-\theta_3)=2\log \frac{|\sin \pi(\theta_2-\theta_3)|}{|\sin \pi(\theta_2+\theta_3)|}<0    
\end{align}
for $0<\theta_1< \frac{1}{2}$,  $\frac{1}{2}<\theta_2<1$ and $\frac{1}{2}<\theta_3<1$. 

Therefore, on the cube $CB_1: I_{21}J_2J_3I_{31}-E_{21}F_2F_3E_{31} $, for $(\theta_1,\theta_2,\theta_3)\in CB_1$,  
we have
\begin{align} \label{formula-CB1}
&2\pi v(\theta_1,\theta_3,\theta_3)\leq  2\pi \max_{\theta_2\in (\frac{1}{2},1)} v\left(\frac{1}{4},\theta_2,\frac{3}{4}\right)\\\nonumber
&=2\pi v\left(\frac{1}{4},\frac{5}{6},\frac{3}{4}\right)=v_3 <3.56337.   
\end{align}
On the cube $CB_2: A_{21}B_2B_3A_{31}-I_{21}J_2J_3I_{31}$, for $(\theta_1,\theta_2,\theta_3)\in CB_2$, we have
\begin{align} \label{formula-CB2}
&2\pi v(\theta_1,\theta_3,\theta_3)\leq 2\pi\max_{\theta_2\in (\frac{1}{2},1)}v\left(0,\theta_2,\frac{3}{4}\right)\\\nonumber
&=2\pi v\left(0,\frac{5}{6},\frac{3}{4}\right)=3.552296<3.56337.   
\end{align}
On the cube $CB_3: I_{01}J_0J_2I_{21}-E_{01}F_0F_2E_{21}$, for $(\theta_1,\theta_2,\theta_3)\in CB_3$, we have
\begin{align}  \label{formula-CB3}
&2\pi v(\theta_1,\theta_3,\theta_3)\leq 2\pi\max_{\theta_2\in (\frac{1}{2},1)}v\left(\frac{1}{4},\theta_2,\frac{1}{2}\right)\\\nonumber
&=2\pi v\left(\frac{1}{4},0.688635,\frac{1}{2}\right)=3.252728<3.56337.   
\end{align}

Now we consider the cube $SPM_1: A_{01}B_0B_2A_{21}-I_{01}J_0J_2I_{21}$ as shown in Figure \ref{figure2}, it is obvious that 
\begin{align}
MCB=SPM_1\cup CB_1\cup CB_2\cup CB_3.     
\end{align}
By the formulas (\ref{formula-CB1}), (\ref{formula-CB2}) and (\ref{formula-CB3}),  we obtain 
\begin{align}
 &\int_{MCB}\sin(\pi \theta_1)\sin(2\pi\theta_3)e^{(N+\frac{1}{2})\hat{V}_N(\theta_1,\theta_2,\theta_3)}\\\nonumber
 &=\int_{SPM_1}\sin(\pi \theta_1)\sin(2\pi\theta_3)e^{(N+\frac{1}{2})\hat{V}_N(\theta_1,\theta_2,\theta_3)}+O(e^{\frac{N+\frac{1}{2}}{2\pi }(3.56337-\epsilon)}).  
\end{align}

\begin{figure}[!htb] \label{figure2} 
\begin{align*} 
\raisebox{-15pt}{
\includegraphics[width=500 pt]{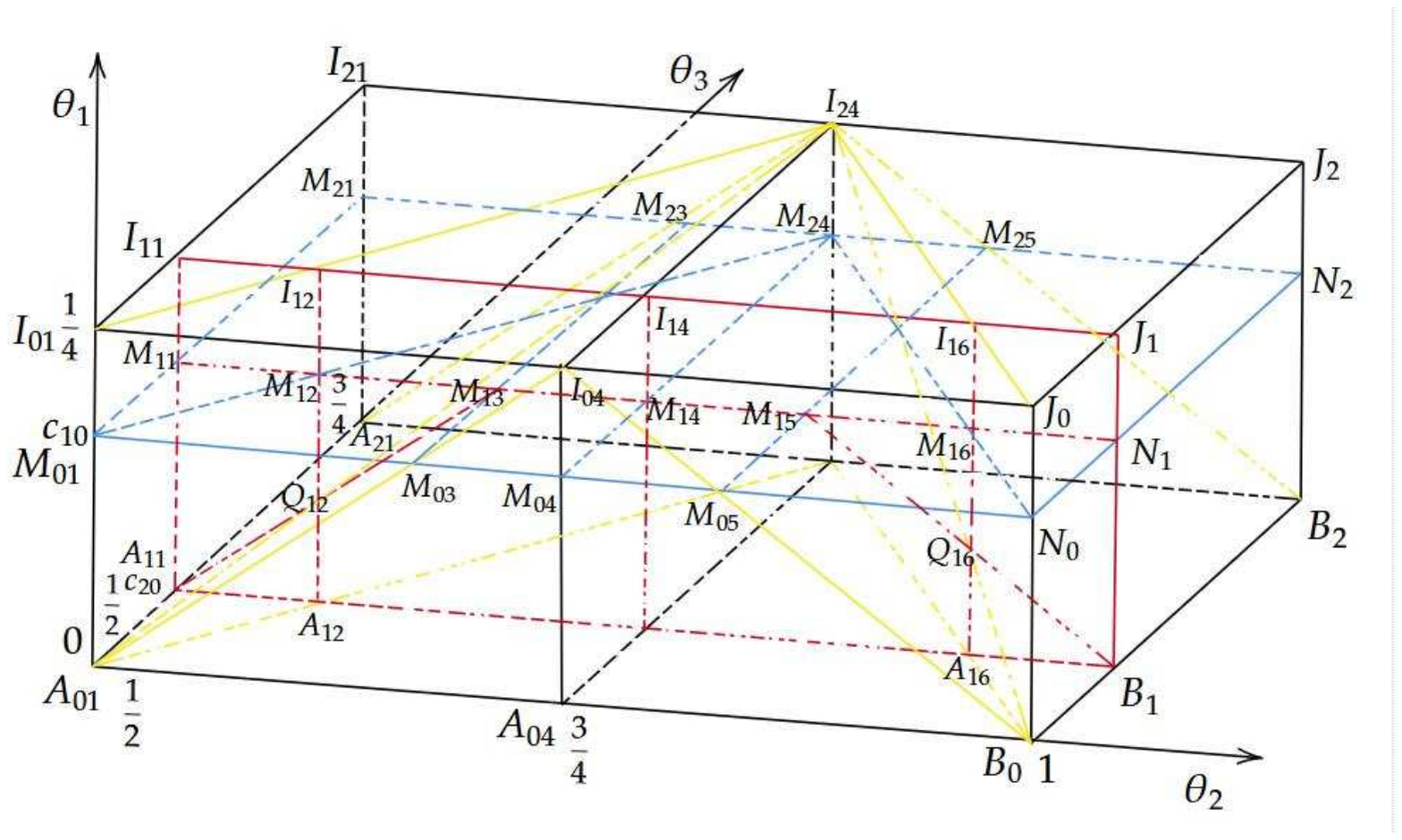}}.
\end{align*}
\caption{The cube $SPM_1$}
\end{figure}

Now we consider the integral over $SPM_1$. 
For the cube $CB_4: M_{11}N_1N_2M_{21}-I_{11}J_1J_2I_{21}$.

For any fixed $(c_1,c_3)\in M_{11}M_{21}J_{21}I_{11}$, i.e. $c_{10}\leq c_1<\frac{1}{4}$ and $c_{20}\leq c_3<\frac{3}{4}$, 
the critical point equation $\hat{V}_{\theta_2}(p,q,c_1,\theta,c_3)=0$ has a unique solution $\theta_2(c_1,c_3)$ with $\text{Re}(\theta_2(c_1,c_3))\in (\frac{1}{2},1)$, see Appendix \ref{appendix-onedim}. By Proposition \ref{prop-Rec1c3decrease},
we obtain
\begin{align}
&2\pi \text{Re}\hat{V}(p,q,c_1,\theta_2(c_1,c_3),c_3)\\\nonumber
&\leq 2\pi\text{Re}\hat{V}(p,q,c_{10},\theta_2(c_{10},c_{30}),c_{30}))\\\nonumber
&=2\pi\text{Re}\hat{V}(p,q,0.1225320,0.808058-0.111315\sqrt{-1},0.600484))\\\nonumber
& =3.25806<3.56337
\end{align}

By using the one-dimensional saddle point method, see Appendix \ref{appendix-onedim} for checking the condition, we estimate the integral 
\begin{align}
\int_{\frac{1}{2}}^1 \sin(\pi \theta_1)\sin(2\pi\theta_3)e^{(N+\frac{1}{2})\hat{V}_N(c_1,\theta_2,c_3)}d\theta_2  =O(e^{\frac{N+\frac{1}{2}}{2\pi}(3.56337-\epsilon)}) 
\end{align}
for some $\epsilon>0$. 
Therefore, 
\begin{align}
&\int_{CB4}\sin(\pi \theta_1)\sin(2\pi\theta_3)e^{(N+\frac{1}{2})\hat{V}_N(\theta_1,\theta_2,\theta_3)}\\\nonumber
&=\int_{(\theta_1,\theta_3)\in M_{11}M_{21}J_{21}I_{11}} \left(\int_{\frac{1}{2}}^{1}\sin(\pi \theta_1)\sin(2\pi\theta_3)e^{(N+\frac{1}{2})\hat{V}_N(\theta_1,\theta_2,\theta_3)}d\theta_2\right)d\theta_1d\theta_3\\\nonumber
&=O(e^{\frac{N+\frac{1}{2}}{2\pi }(3.55337-\epsilon)}).
\end{align}

On the Triangular Prism $TPL_1: A_{01}M_{03}M_{01}-A_{2}M_{23}M_{21}$, for $(\theta_1,\theta_2,\theta_3)\in TPL_1$, 
by formula (\ref{formula-vtheta2}), we have 
\begin{align} \label{formula-TPL1}
&2\pi \max_{(\theta_1,\theta_2,\theta_3)\in TPL_1} v(\theta_1,\theta_2,\theta_3)\\\nonumber 
   &=2\pi \max_{(\theta_1,\theta_2)\in A_{01}M_{03}M_{01}}v(\theta_1,\theta_2,\frac{1}{2})\\\nonumber
   &=2\pi \max_{0\leq c\leq c_{10}}v\left(c,c+\frac{1}{2},\frac{1}{2}\right)\\\nonumber
   &=2\pi v\left(c_{10},c_{10}+\frac{1}{2},\frac{1}{2}\right)=3.37448<3.56337. 
\end{align}
Similarly, on the Triangular Prism $TPR_1: B_{0}M_{06}N_{0}-B_{2}M_{25}N_{2}$, we have
\begin{align} \label{formula-TPR1}
&2\pi \max_{(\theta_1,\theta_2,\theta_3)\in TPR_1} v(\theta_1,\theta_2,\theta_3)\\\nonumber
&=2\pi \max_{(\theta_1,\theta_3)\in B_{0}M_{05}N_0} v\left(\theta_1,\theta_2,\frac{1}{2}\right)\\\nonumber
&=2\pi \max_{0\leq c\leq c_{10}}v(c,1-c,\frac{1}{2})\\\nonumber
&=2\pi v\left(c_{10},1-c_{10},\frac{1}{2}\right)=2.93274<3.56337.
\end{align}
On the Triangular Prism $TPL_2: A_{01}A_{12}A_{11}-I_{01}I_{12}I_{11}$,  we have
\begin{align}
&2\pi \max_{(\theta_1,\theta_2,\theta_3)\in TPL_2} v(\theta_1,\theta_2,\theta_3)\\\nonumber
&=2\pi \max_{(\theta_2,\theta_3)\in A_{01}A_{12}A_{11}} v(0,\theta_2,\theta_3)\\\nonumber
&=2\pi \max_{0\leq c\leq c_{30}}v(0,c,c)\\\nonumber
&=2\pi v(0, c_{30},c_{30})=2.27746<3.56337.
\end{align}
On the Triangular Prism $TPR_2: B_{0}B_{1}A_{16}-J_{0}J_{1}I_{16}$, we have
\begin{align}
&2\pi \max_{(\theta_1,\theta_2,\theta_3)\in TPR_2} v(\theta_1,\theta_2,\theta_3)\\\nonumber
&=2\pi \max_{(\theta_2,\theta_3)\in B_{0}B_{1}A_{16}} v(0,\theta_2,\theta_3)\\\nonumber
&=2\pi \max_{0\leq c\leq c_{30}}v\left(0,\frac{3}{2}-c,c\right)\\\nonumber
&=2\pi v\left(0, \frac{3}{2}-c_{30},c_{30}\right)=3.56337.
\end{align}

Note that the intersections of $TPL_1$ and $TPL_2$ is a truncated triangular prism $TTPL: A_{01}Q_{12}A_{11}-M_{01}M_{12}M_{11}$, and the intersections of $TPR_1$ and $TPR_2$ is a truncated triangular prisim $TTPR: B_0Q_{16}B_1-N_{01}M_{16}N_1$.

Now, we let $SPM_2=SPM_1-TPL_1-TPR_1-(TPL_2\setminus TTPL)-(TPR_2\setminus TTPR)$, by previous formulas, we have
\begin{align}
&\int_{SPM_1}\sin(\pi \theta_1)\sin(2\pi\theta_3)e^{(N+\frac{1}{2})\hat{V}_N(\theta_1,\theta_2,\theta_3)}\\\nonumber
&=\int_{SPM_2}\sin(\pi \theta_1)\sin(2\pi\theta_3)e^{(N+\frac{1}{2})\hat{V}_N(\theta_1,\theta_2,\theta_3)}+O(e^{\frac{N+\frac{1}{2}}{2\pi}(3.55337-\epsilon)}).  
\end{align}

For the Trapezoidal Prism $TZP_1: A_{11}B_1M_{15}M_{13}-A_{21}B_2M_{25}M_{23}$, first, we add two triangular prisms  $A_{11}M_{13}M_{11}-A_{21}M_{23}M_{21}$ and 
 $B_1M_{15}N_1-B_2M_{25}N_2$ to $TZP_1$, we obtain a cube $CB_5: A_{11}B_1B_2A_{21}-M_{11}N_1N_2M_{21}$, 

By formulas (\ref{formula-TPL1}) and (\ref{formula-TPR1}), we obtain that 
\begin{align}
&\int_{TZP_1}\sin(\pi \theta_1)\sin(2\pi\theta_3)e^{(N+\frac{1}{2})\hat{V}_N(\theta_1,\theta_2,\theta_3)}\\\nonumber
&=\int_{CB_5}\sin(\pi \theta_1)\sin(2\pi\theta_3)e^{(N+\frac{1}{2})\hat{V}_N(\theta_1,\theta_2,\theta_3)}+O(e^{\frac{N+\frac{1}{2}}{2\pi}(3.55337-\epsilon)})
\end{align}

For any fixed $(c_1,c_3)\in A_{11}A_{21}M_{21}M_{11}$, i.e. $0\leq c_{1}\leq c_{10}$ and $c_{20}\leq c_3<\frac{3}{4}$,  
the critical point equation $\hat{V}_{\theta_2}(p,q,c_1,\theta,c_3)=0$ has a unique solution $\theta_2(c_1,c_3)$ with $\text{Re}(\theta_2(c_1,c_3))\in (\frac{1}{2},1)$, see Appendix \ref{appendix-onedim}. By Proposition \ref{prop-Rec1c3decrease}
we obtain
\begin{align}
&2\pi \text{Re}\hat{V}(p,q,c_1,\theta_2(c_1,c_3),c_3)\\\nonumber
&\leq 2\pi\text{Re}\hat{V}(p,q,c_{10},\theta_2(c_{10},c_{30}),c_{30}))\\\nonumber
&=2\pi\text{Re}\hat{V}(p,q,0,0.8270666460-0.1216893136\sqrt{-1},0.600484))\\\nonumber
& =3.563367<3.56337.
\end{align}

By using the one-dimensional saddle point method, see Appendix \ref{appendix-onedim} for checking the condition, we estimate the integral 
\begin{align}
\int_{\frac{1}{2}}^1 \sin(\pi \theta_1)\sin(2\pi\theta_3)e^{(N+\frac{1}{2})\hat{V}_N(c_1,\theta_2,c_3)}d\theta_2  =O(e^{\frac{N+\frac{1}{2}}{2\pi}(3.56337-\epsilon)}) 
\end{align}
for some $\epsilon>0$. 
Therefore, 
\begin{align}
&\int_{CB_5}\sin(\pi \theta_1)\sin(2\pi\theta_3)e^{(N+\frac{1}{2})\hat{V}_N(\theta_1,\theta_2,\theta_3)}\\\nonumber
&=\int_{(\theta_1,\theta_3)\in A_{11}A_{21}M_{21}M_{11}}\left(\int_{\frac{1}{2}}^{1}\sin(\pi \theta_1)\sin(2\pi\theta_3)e^{(N+\frac{1}{2})\hat{V}_N(\theta_1,\theta_2,\theta_3)}d\theta_2\right)d\theta_1d\theta_3\\\nonumber
&=O(e^{\frac{N+\frac{1}{2}}{2\pi }(3.56337-\epsilon)}).
\end{align}

For the Trapezoidal Prism $TZP_2: M_{01}N_0M_{16}M_{12}-I_{01}J_0I_{16}I_{12}$, we can estimates its integral similarly as the computation for $TZP_1$.

Finally, we obtain a polyhedron $SPM_2=MCB-\cup_{i=1}^6CB_i-\cup_{i=1}^2 TPL_i-\cup_{i=1}^2 TPR_i$.  
Then we use the polyhedron $SPM_2$ to do the reflection along the plane $\theta_3=\frac{1}{2}$ and $\theta_1=0$ to obtain a new polyhedron denoted by $P$, the key point is that $P\subset D_{H}$. Let $D_0''=P\cap D_{0}$,  the above analysis tell us that
\begin{align}
&\int_{D_0}\sin(\pi \theta_1)\sin(2\pi\theta_3)e^{(N+\frac{1}{2})\hat{V}_N(\theta_1,\theta_2,\theta_3)}d\theta_1d\theta_2d\theta_3\\\nonumber
&=\int_{D''_0}\sin(\pi \theta_1)\sin(2\pi\theta_3)e^{(N+\frac{1}{2})\hat{V}_N(\theta_1,\theta_2,\theta_3)}d\theta_1d\theta_2d\theta_3+O(e^{(N+\frac{1}{2})\zeta_{\mathbb{R}}(p,q)-\epsilon}).    
\end{align}

\begin{proposition} \label{prop-m0n0} 
For $(p,q)$ satisfies the condition (\ref{formula-pqconditions}),  we have
\begin{align} \label{formula-D''}
    \hat{h}_N(0,0,0)&=-\kappa''_r\int_{D_0}\sin(\pi \theta_1)\sin(2\pi\theta_3)e^{(N+\frac{1}{2})\hat{V}_N(\theta_1,\theta_2,\theta_3)}d\theta_1d\theta_2d\theta_3   \\\nonumber
    &=\kappa''_r\omega(p,q)e^{(N+\frac{1}{2})\zeta_{\mathbb{R}}(p,q)}\\\nonumber
&\cdot\left(1+\sum_{i=1}^d\kappa_i(p,q)\left(\frac{2\pi\sqrt{-1}}{N+\frac{1}{2}}\right)^i+O\left(\frac{1}{(N+\frac{1}{2})^{d+1}}\right)\right),
    \end{align}
    for $d\geq 1$, where $\omega(p,q)$ and $\kappa_i(p,q)$ are constants determined by $M_{p,q}$.
\end{proposition}

\begin{proof}
We need to estimate the following integral
  \begin{align}
 &\int_{D_0}\sin(\pi \theta_1)\sin(2\pi\theta_3)e^{(N+\frac{1}{2})\hat{V}_N(\theta_1,\theta_2,\theta_3)}d\theta_1d\theta_2d\theta_3. 
\end{align}  

According to formula (\ref{formula-D''}), we only need to estimate the integral  
\begin{align}  \label{formula-integralD''}
\int_{D''_0}\sin(\pi \theta_1)\sin(2\pi\theta_3)e^{(N+\frac{1}{2})\hat{V}_N(\theta_1,\theta_2,\theta_3)}d\theta_1d\theta_2d\theta_3. 
\end{align}

We will verify the conditions of Proposition \ref{proposition-saddlemethod} for saddle point method in Proposition \ref{propostion-checksaddle}. By Proposition \ref{prop-VNexpansion} and  Remark \ref{remark-saddle},   
we can apply the Proposition \ref{proposition-saddlemethod} to the above integral (\ref{formula-integralD''}). Let $(\theta_1^0,\theta_2^0,\theta_3^0)$ be the critical point of $\hat{V}(p,q,\theta_1,\theta_2,\theta_3)$, we obtain that

\begin{align}
  &\int_{D''_0}\sin(\pi \theta_1)\sin(2\pi\theta_3)e^{(N+\frac{1}{2})\hat{V}_N(\theta_1,\theta_2,\theta_3)}d\theta_1d\theta_2d\theta_3\\\nonumber
        &=\left(\frac{\pi}{N+\frac{1}{2}}\right)^{\frac{3}{2}}\frac{\alpha(\theta_1^0,\theta_2^0,\theta_3^0)}{\sqrt{\det \left(-\frac{1}{2}\text{Hess}(\hat{V})(\theta_1^0,\theta_2^0,\theta_3^0)\right)}}e^{(N+\frac{1}{2})\zeta(p,q)}\\\nonumber
        &\left(1+\sum_{i=1}^d\kappa_i(p,q)\left(\frac{2\pi\sqrt{-1}}{N+\frac{1}{2}}\right)^i+O\left(\frac{1}{(N+\frac{1}{2})^{d+1}}\right)\right)   
\end{align}
 where the function  
\begin{align}
    &\alpha(\theta_1^0,\theta_2^0,\theta_3^0)\\\nonumber
    &=\sin(\pi\theta_1^0)\sin(2\pi\theta_3^0)e^{\pi\sqrt{-1}\left(2\theta_2^0+(p-2q+\frac{7}{4})\right)-\frac{1}{2}\log(1-e^{2\pi\sqrt{-1}(\theta_2^0+\theta_3^0)})-\frac{1}{2}\log(1-e^{2\pi\sqrt{-1}(\theta_2^0-\theta_3^0)})}\\\nonumber
    &=\frac{(-1)^{p-\frac{1}{4}}z_2^0(z_3^0-\frac{1}{z_3^{0}})(\sqrt{z_1^0}-\frac{1}{\sqrt{z_1^0}})}{4\sqrt{(1-z_2^0z_3^0)(1-z_2^0\frac{1}{z_{3}^0})}}.
\end{align}
the determinant of the Hessian matrix at $(\theta_1^0,\theta_2^0,\theta_3^0)$ is given by 
\begin{align}
    &\det \left(-\frac{1}{2}\text{Hess}(\hat{V})(\theta_1^0,\theta_2^0,\theta_3^0)\right)=(-\frac{1}{2})^3(2\pi\sqrt{-1})^3H(p,q;z_1^0,z_2^0,z_3^0)
\end{align}
where 
\begin{align}
  &H(p,q;z_1^0,z_2^0,z_3^0)=4\frac{z_2^0}{z_1^0-z_2^0}\cdot \frac{z_1^0z_2^0}{1-z_1^0z_2^0}\left(\frac{z_2^0z_3^0}{1-z_2^0z_3^0}+\frac{z_2^0}{z_3^0-z_2^0}\right)\\\nonumber
  &-4\left(\frac{z_2^0}{z_1^0-z_2^0}+\frac{z_1^0z_2^0}{1-z_1^0z_2^0}\right)\frac{z_2^0z_3^0}{1-z_2^0z_3^0}\cdot\frac{z_2^0}{z_3^0-z_2^0}+\left(\frac{z_2^0}{z_1^0-z_2^0}+\frac{z_1^0z_2^0}{1-z_1^0z_2^0}\right)\frac{z_2^0}{1-z_2^0}\left(\frac{z_2^0z_3^0}{1-z_2^0z_3^0}+\frac{z_2^0}{z_3^0-z_2^0}\right)\\\nonumber
  &+(8p+4)\frac{z_2^0}{z_1^0-z_2^0}\cdot \frac{z_1^0z_2^0}{1-z_1^0z_2^0}+(2p+1)\left(\frac{z_2^0}{z_1^0-z_2^0}+\frac{z_1^0z_2^0}{1-z_1^0z_2^0}\right)\frac{z_2^0}{1-z_2^0}\\\nonumber
  &-(2p+\frac{q}{2})\left(\frac{z_2^0}{z_1^0-z_2^0}+\frac{z_1^0z_2^0}{1-z_1^0z_2^0}\right)\left(\frac{z_2^0z_3^0}{1-z_2^0z_3^0}+\frac{z_2^0}{z_3^0-z_2^0}\right)-(\frac{q}{2}-1)\frac{z_2^0}{1-z_2^0}\left(\frac{z_2^0z_3^0}{1-z_2^0z_3^0}+\frac{z_2^0}{z_3^0-z_2^0}\right)\\\nonumber
  &+4(\frac{q}{2}-1)\frac{z_2^0z_3^0}{1-z_2^0z_3^0}\cdot\frac{z_2^0}{z_3^0-z_2^0}-(2p+1)(\frac{q}{2}-1)\left(\frac{z_2^0}{z_1^0-z_2^0}+\frac{z_1^0z_2^0}{1-z_1^0z_2^0}+\frac{z_2^0}{1-z_{2}^0}-\frac{z_2^0z_3^0}{1-z_2^0z_3^0}-\frac{z_2^0}{z_3^0-z_2^0}\right).
\end{align}

Therefore, we have

\begin{align}
  &\int_{D''_0}\sin(\pi \theta_1)\sin(2\pi\theta_3)e^{(N+\frac{1}{2})\hat{V}_N(\theta_1,\theta_2,\theta_3)}d\theta_1d\theta_2d\theta_3\\\nonumber
        &=\frac{(-1)^{p}\sqrt{-1}z_2(z_3-z_3^{-1})(\sqrt{z_1}-\frac{1}{\sqrt{z_1}})}{4(N+\frac{1}{2})^{\frac{3}{2}}\sqrt{(1-z_2z_3)(1-z_2z_{3}^{-1})}\sqrt{H(p,q;z_1^0,z_2^0,z_3^0)}}e^{(N+\frac{1}{2})\zeta(p,q)}\\\nonumber
        &\cdot\left(1+\sum_{i=1}^d\kappa_i(p,q)\left(\frac{2\pi\sqrt{-1}}{N+\frac{1}{2}}\right)^i+O\left(\frac{1}{(N+\frac{1}{2})^{d+1}}\right)\right). 
\end{align}
Hence
\begin{align}
     \hat{h}_N(0,0,0)&=-\kappa''_r\int_{D_0}\sin(\pi \theta_1)\sin(2\pi\theta_3)e^{(N+\frac{1}{2})\hat{V}_N(\theta_1,\theta_2,\theta_3)}d\theta_1d\theta_2d\theta_3\\\nonumber
     &=(-1)^{p+1}\frac{\sqrt{-1}}{4}e^{\sigma((\mathcal{K}_p)_q)\left(\frac{3}{2N+1}+\frac{N+1}{2}\right)\pi\sqrt{-1}}
     \omega(p,q)e^{(N+\frac{1}{2})\zeta(p,q)}\\\nonumber
        &\cdot\left(1+\sum_{i=1}^d\kappa_i(p,q)\left(\frac{2\pi\sqrt{-1}}{N+\frac{1}{2}}\right)^i+O\left(\frac{1}{(N+\frac{1}{2})^{d+1}}\right)\right)
\end{align}
where 
\begin{align} \label{formula-omega}
    \omega(p,q)=\frac{z_2^0(z_3^0-(z_3^0)^{-1})(\sqrt{z_1^0}-\frac{1}{\sqrt{z_1^0}})}{\sqrt{(1-z_2^0z_3^0)(1-z_2^0(z_{3}^0)^{-1})}\sqrt{H(p,q;z_1^0,z_2^0,z_3^0)}}.
\end{align}
\end{proof}

\begin{proposition} \label{propostion-checksaddle}
    When we apply Proposition \ref{proposition-saddlemethod} (saddle point method)  to the integral (\ref{formula-integralD''}), the assumptions of Proposition \ref{proposition-saddlemethod} holds. 
\end{proposition}
\begin{proof}
We note that, by Lemma \ref{lemma-varphixi3}, $\hat{V}_N(p,q;\theta_1,\theta_2,\theta_3)$ uniformly converges to the $\hat{V}(p,q;\theta_1,\theta_2,\theta_3)$ on $D''_0$ as $N\rightarrow \infty$. Hence,  we only need to verify the assumptions of the saddle point method for $\hat{V}(p,q;\theta_1,\theta_2,\theta_3)$.  
 We show that there exists a homotopy $D''_\delta$ ($0\leq \delta\leq 1 $) between $D''_0$ and $D''_1$ such that 
\begin{align}
 &(\theta_1^0,\theta_2^0,\theta_3^0)\in D''_1,  \label{saddle-1} \\ 
    &D''_1-\{(\theta_1^0,\theta_2^0,\theta_3^0)\}\subset \{(\theta_1,\theta_2,\theta_3)\in \mathbb{C}^3|Re V(p,q;\theta_1,\theta_2,\theta_3)<\zeta_\mathbb{R}(p,q)\}, \label{saddle-2} \\
    & \int_{\partial D''_{\delta}} \sin(\pi \theta_1)\sin(2\pi\theta_3)e^{(N+\frac{1}{2})\hat{V}_N(\theta_1,\theta_2,\theta_3)}d\theta_1d\theta_2d\theta_3=O\left(e^{(N+\frac{1}{2})(\zeta_{\mathbb{R}}(p,q)-\epsilon)}\right). \label{saddle-3}
\end{align}

   In the fiber of the projection $\mathbb{C}^3\rightarrow \mathbb{R}^3$ at $(\theta_{1R},\theta_{2R},\theta_{3R})\in D_0''$, we consider the flow from $(X_1,X_2,X_3)=(0,0,0)$ determined by the vector field $(-\frac{\partial \hat{f}}{\partial X_1},-\frac{\partial \hat{f}}{\partial X_2},-\frac{\partial \hat{f}}{\partial X_3})$. By the construction of the region $D''_0\subset D_{H}$, together with Corollary \ref{coro-U0noU0} and Proposition \ref{proposition-Hessianf}, the convex neighborhood $U_0$ of $(\theta^0_{1R},\theta^0_{2R},\theta^0_{3R})$ satisfies the following holds. 
 \begin{itemize}
        \item[(1)] If $(\theta_{1R},\theta_{2R},\theta_{3R})\in U_0$, then $\hat{f}$ has a unique minimal point, and the flow goes there.
        \item[(2)] If $(\theta_{1R},\theta_{2R},\theta_{3R})\in D''_0\setminus U_0$, then the flow goes to infinity. 
    \end{itemize}
We put $\mathbf{g}(\theta_{1R},\theta_{2R},\theta_{3R})=(g_1(\theta_{1R},\theta_{2R},\theta_{3R}),g_2(\theta_{1R},\theta_{2R},\theta_{3R}),g_2(\theta_{1R},\theta_{2R},\theta_{3R})$ to be the minimal point of $(1)$. In particular, $|\mathbf{g}(\theta_{1R},\theta_{2R},\theta_{3R})|\rightarrow \infty$ as $(\theta_{1R},\theta_{2R},\theta_{3R})$ goes to $\partial U_0$. Further, for a sufficiently large $R>0$, we stop the flow when $|\mathbf{g}(\theta_{1R},\theta_{2R},\theta_{3R})|=R$. We construct the revised flow $\hat{\mathbf{g}}(\theta_{1R},\theta_{2R},\theta_{3R})$, by putting $\hat{\mathbf{g}}(\theta_{1R},\theta_{2R},\theta_{3R})=\mathbf{g}(\theta_{1R},\theta_{2R},\theta_{3R})$ for $(\theta_{1R},\theta_{2R},\theta_{3R})\in U_0$ with $|\mathbf{g}(\theta_{1R},\theta_{2R},\theta_{3R})|<R$, otherwise, by putting $|\hat{\mathbf{g}}(\theta_{1R},\theta_{2R},\theta_{3R})|=(R,R,R)$. 

 We define the ending of the homotopy by
\begin{align}
    D''_1=\{(\theta_{1R},\theta_{2R},\theta_{3R})+\hat{\mathbf{g}}(\theta_{1R},\theta_{2R},\theta_{3R})\sqrt{-1}|(\theta_{1R},\theta_{2R},\theta_{3R})\in D''_0\}. 
\end{align}
Further, we define the internal part of the homotopy by setting it along the flow from $(\theta_{1R},\theta_{2R},\theta_{3R})$ determined by the vector field $\left(-\frac{\partial \hat{f}}{\partial X_1},-\frac{\partial \hat{f}}{\partial X_2},-\frac{\partial \hat{f}}{\partial X_3}\right)$. 

We show (\ref{saddle-1}) and (\ref{saddle-2}) as follows. We consider the function
\begin{align}
    h(\theta_{1R},\theta_{2R},\theta_{3R})=\text{Re} V(\theta_{1R}+\sqrt{-1}\hat{g}_1,\theta_{2R}+\sqrt{-1}\hat{g}_2,\theta_{3R}+\sqrt{-1}\hat{g}_3). 
\end{align}
If $(\theta_{1R},\theta_{2R},\theta_{3R})\notin U_0$, by (2), $-h(\theta_{1R},\theta_{2R},\theta_{3R})$ is sufficiently large (because we let $R$ be sufficiently large), hence (\ref{saddle-2}) holds in this case. Otherwise, $(\theta_{1R},\theta_{2R},\theta_{3R})\in U_0$, in this case, $\hat{\mathbf{g}}(\theta_{1R},\theta_{2R},\theta_{3R})=\mathbf{g}(\theta_{1R},\theta_{2R},\theta_{3R})$. It is shown from the definition of $\hat{\mathbf{g}}(\theta_{1R},\theta_{2R},\theta_{3R})$ that 
\begin{align}
    \frac{\partial \text{Re} \hat{V}}{\partial X_1}=\frac{\partial \text{Re} \hat{V}}{\partial X_2}=\frac{\partial \text{Re} \hat{V}}{\partial X_3}=0 \ \text{at} \ (X_1,X_2,X_3)=\mathbf{g}(\theta_{1R},\theta_{2R},\theta_{3R}),
\end{align}
which implies 
\begin{align}
    \text{Im}\frac{\partial \hat{V}}{\partial \theta_1}=\text{Im} \frac{\partial \hat{V}}{\partial \theta_2}=\text{Im} \frac{\partial \hat{V}}{\partial \theta_3}=0 \ \text{at} \ (\theta_{1R},\theta_{2R},\theta_{3R})+\mathbf{g}(\theta_{1R},\theta_{2R},\theta_{3R})\sqrt{-1}.
\end{align}
On the other hand, for $i=1,2,3$, 
    \begin{align}
        \frac{\partial h}{\partial \theta_{iR}}=\text{Re}\frac{\partial \hat{V}}{\partial \theta_i}, \ \text{at} \ (\theta_{1R},\theta_{2R},\theta_{3R})+\mathbf{g}(\theta_{1R},\theta_{2R},\theta_{3R})\sqrt{-1}.
    \end{align}
Therefore, when  $(\theta_{1R},\theta_{2R},\theta_{3R})+\mathbf{g}(\theta_{1R},\theta_{2R},\theta_{3R})\sqrt{-1}$ is a critical point of $\hat{V}$, $(\theta_{1R},\theta_{2R},\theta_{3R})$ is a critical point of $h(\theta_{1R},\theta_{2R},\theta_{3R})$. Hence by Proposition \ref{prop-critical}, $h(\theta_{1R},\theta_{2R},\theta_{3R})$ has a unique maximal point at $(\theta_{1R}^0,\theta_{2R}^0,\theta_{3R}^0)$ which is equal to $\zeta_{\mathbb{R}}(p,q)$. Therefore, (\ref{saddle-1}) and (\ref{saddle-2}) holds.

We show (\ref{saddle-3}) as follows. Note that the boundary of $\partial D''_0$ consists of $D_0(c_{10})$, $D_0(c_{30})$ and the partial boundaries of $D_0$ denoted by $D_{0b}$.  Hence $\partial D''_{\delta}$ consists of three parts denoted by 
\begin{align}
    \partial D''_{\delta}=A_1\cup A_2\cup B,
\end{align}
where $A_1$ and $A_2$ comes from the flows start at $(\theta_{1R},\theta_{2R},\theta_{3R})\in D_0(c_{10})$ and  $(\theta_{1R},\theta_{2R},\theta_{3R})\in D_0(c_{30})$ respectively, while $B$ comes from the flows start at $(\theta_{1R},\theta_{2R},\theta_{3R})\in D_{0b}$.

By its definition, $D_{0b}\subset \partial D_{0}\subset \{(\theta_{1},\theta_{2},\theta_{3})\in \mathbb{C}^3|\text{Re} \hat{V}(p,q;\theta_{1},\theta_{2},\theta_{3})<\zeta_\mathbb{R}(p,q)-\epsilon\}$, and the function $\text{Re} \hat{V}(p,q;\theta_1,\theta_2,\theta_3)$ decreases under the flow, so we have 
\begin{align} \label{formula-integralB2}
    \int_{B}\sin(\pi \theta_1)\sin(2\pi\theta_3)e^{(N+\frac{1}{2})\hat{V}_N(\theta_1,\theta_2,\theta_3)}d\theta_1d\theta_2d\theta_3=O\left(e^{(N+\frac{1}{2})(\zeta_{\mathbb{R}}(p,q)-\epsilon)}\right).
\end{align}

By Theorems \ref{theorem-usedinm1} and \ref{theorem-usedinm3}, the integrals on $D_0(c_{10})$ and $D_0(c_{30})$ 
is also of order $O(e^{(N+\frac{1}{2})(\zeta_\mathbb{R}(p,q)-\epsilon}))$. By applying the saddle point method to the slices of the region $A_1\cup A_2$ as shown in Appendix \ref{Appendix-twodimsaddle}, we can prove that
\begin{align} \label{formula-integralA1A22}
    \int_{A_1\cup A_2}\sin(\pi \theta_1)\sin(2\pi\theta_3)e^{(N+\frac{1}{2})\hat{V}_N(\theta_1,\theta_2,\theta_3)}d\theta_1d\theta_2d\theta_3=O\left(e^{(N+\frac{1}{2})(\zeta_{\mathbb{R}}(p,q)-\epsilon)}\right).
\end{align}
Combining formulas (\ref{formula-integralB2}) and (\ref{formula-integralA1A22}) together, we prove (\ref{saddle-3}).

By (\ref{saddle-1}) (\ref{saddle-2}) and (\ref{saddle-3}), the required homotopy exists. Hence the assumptions of Proposition \ref{proposition-saddlemethod} holds when we apply the saddle point method to the integral (\ref{formula-integralD''}).  
\end{proof}

\subsection{Final proof} \label{subsection-final}

Now we can finish the proof of Theorem \ref{theorem-main} as follows.
\begin{proof}
Using formula (\ref{formula-propRT}), together with Proposition \ref{prop-m0n0}, we obtain
    \begin{align}
        RT_{r}(M_{p,q})&=4\hat{h}_{N}(0,0,0)+O(e^{(N+\frac{1}{2})(\zeta_\mathbb{R}(p,q))-\epsilon}),
        \\\nonumber
        &=(-1)^{p+1}\sqrt{-1}e^{\sigma_{p,q}\left(\frac{3}{2N+1}+\frac{N+1}{2}\right)\pi\sqrt{-1}}\omega(p,q)e^{(N+\frac{1}{2})\zeta_{\mathbb{R}}(p,q)}\\\nonumber
&\cdot\left(1+\sum_{i=1}^d\kappa_i(p,q)\left(\frac{2\pi\sqrt{-1}}{N+\frac{1}{2}}\right)^i+O\left(\frac{1}{(N+\frac{1}{2})^{d+1}}\right)\right),
    \end{align}
    for $d\geq 1$, where $\omega(p,q)$ is given by formula (\ref{formula-omega}) and $\kappa_i(p,q)$ are constants determined by $M_{p,q}$.
\end{proof}

\section{Appendices} \label{Section-Appendices}

\subsection{Geometry of the critical point}  \label{Appendix-geometry}

\subsubsection{The geometric equation}
Note that the hyperbolic geometric equations (which include the gluing equations and Dehn filling equations) for the $M_{p,q}$ can be written as follows:
\begin{equation} \label{equantion-geom1} 
\left\{ \begin{aligned}
         &\log w+\log x+\log y+\log z = 2\pi\sqrt{-1}, \\
          &        \log(1-w)+\log(1-x)-\log(1-y)-\log(1-z)=0,\\
          &qu_1+v_1=-2\pi\sqrt{-1},  \\
          &u_2-pv_2=-2\pi\sqrt{-1},
                          \end{aligned} \right.
                          \end{equation}
where 
\begin{align}
    u_1&=\log (w-1)+\log x+\log y-\log(y-1)-\pi\sqrt{-1},\\\nonumber
    v_1&=2\log x +2\log y-2\pi\sqrt{-1}, \\\nonumber
    u_2&=\log (w-1)+\log x+\log z-\log(z-1)-\pi\sqrt{-1},\\\nonumber
    v_2&=2\log x +2\log z-2\pi\sqrt{-1}.
\end{align}

It follows that
\begin{align}
&xyzw=1, \label{geomequation-1}\\
&(1-w)(1-x)=(1-y)(1-z), \label{geomequation-2} \\
&(xz)^{2p-1}=-\frac{w-1}{z-1}, \label{geomequation-3} \\
&\left(\frac{-(w-1)xy}{y-1}\right)^q(xy)^2=1. \label{geomequation-4}
\end{align}

Then, we have
\begin{proposition}
    \begin{align}
      w&=1-(xz)^{2p-1}(z-1) \\\nonumber
      y&=1-(xz)^{2p-1}(x-1)
    \end{align}
    where $x,z$ is determined by the following two equations
\begin{align}
    ((xz)^{2p-1}(x-1)-1)((xz)^{2p-1}(z-1)-1)xz&=1\\\nonumber
    \left(-\frac{z-1}{x-1}\right)^q(x(1-(xz)^{2p-1}(x-1)))^{q+2}=1
\end{align}
\end{proposition}
\begin{proof}
From equation (\ref{geomequation-3}), we obtain the first equation,
 \begin{align}
      w&=1-(xz)^{2p-1}(z-1). 
    \end{align}
From equation (\ref{geomequation-1}), we obtain 
\begin{align}
    y=\frac{1}{xzw}=\frac{1}{xz(1-(xz)^{2p-1}(z-1))}.
\end{align}
From equation (\ref{geomequation-2}), we obtian
\begin{align}
    y=\frac{z+xw-w-x}{z-1}=1-(xz)^{2p-1}(x-1).
\end{align}

Moreover, suppose $(x,y,z,w)$ is a solution to the geometric equations (\ref{equantion-geom1}), then the complex volume of the $M_{p,q}$ is given by the following formula 
\begin{align}
&Vol(M_{p,q})+\sqrt{-1}CS(M_{p,q})\\\nonumber&=-\frac{1}{\sqrt{-1}}\left(R(w)+R(x)+R\left(\frac{1}{1-y}\right)+R\left(\frac{1}{1-z}\right)\right)\\\nonumber
&+\frac{\pi}{2}\left(3\pi\sqrt{-1}\log(w-1)+\log(x)+\log(y)-\log(y-1)\right.\\\nonumber
&\left.+\frac{1}{p}(3\pi\sqrt{-1}+\log(w-1)+\log(x)+\log(z)-\log(z-1))\right) \mod \pi^2\sqrt{-1}\mathbb{Z},  
\end{align}
where 
\begin{align}
R(x)=\frac{1}{2}\log(x)\log(1-x)+\text{Li}_2(x).
\end{align}

\end{proof}

\subsection{Critical value gives the complex volume}
Recall that
\begin{align}
    \hat{V}&=\pi\sqrt{-1}\left(\frac{3}{2}q+\frac{q}{2}\theta_1^2-\theta_1^2-\theta_1-2\theta_2-(2p+3)\theta_3+(2p+1)\theta_3^2\right)\\\nonumber
    &+\frac{1}{2\pi\sqrt{-1}}\left(\frac{\pi^2}{6}-\text{Li}_2(e^{2\pi\sqrt{-1}(\theta_2-\theta_1)})-\text{Li}_2(e^{2\pi\sqrt{-1}(\theta_2+\theta_1)})-\text{Li}_2(e^{2\pi\sqrt{-1}\theta_2})\right.\\\nonumber
    &\left.+\text{Li}_2(e^{2\pi\sqrt{-1}(\theta_2+\theta_3)})+\text{Li}_2(e^{2\pi\sqrt{-1}(\theta_2-\theta_3)})\right).
\end{align}
The critical point equations of $\hat{V}$ are given by 
\begin{align} \label{equation-Vs}
\hat{V}_{\theta_1}=\pi\sqrt{-1}(q\theta_1-2\theta_1-1)+\log(1-e^{2\pi\sqrt{-1}(\theta_2+\theta_1)})-\log(1-e^{2\pi\sqrt{-1}(\theta_2-\theta_1)})=0,
\end{align}
\begin{align}  \label{equation-Vt}
\hat{V}_{\theta_2}&=-2\pi\sqrt{-1}-\log(1-e^{2\pi\sqrt{-1}(\theta_2+\theta_3)})-\log(1-e^{2\pi\sqrt{-1}(\theta_2-\theta_3)})\\\nonumber
&+\log(1-e^{2\pi\sqrt{-1}(\theta_2+\theta_1)})+\log(1-e^{2\pi\sqrt{-1}\theta_2})+\log(1-e^{2\pi\sqrt{-1}(\theta_2-\theta_1)})=0,    
\end{align}
and
\begin{align} \label{equation-Vu}
\hat{V}_{\theta_3}&=\pi\sqrt{-1}((4p+2)\theta_3-(2p+3))-\log(1-e^{2\pi\sqrt{-1}(\theta_2+\theta_3)})\\\nonumber&+\log(1-e^{2\pi\sqrt{-1}(\theta_2-\theta_3)})=0.    
\end{align}

By equation (\ref{equation-Vs}), we have 
\begin{align}
&q(-2\pi\sqrt{-1}\theta_1)+4\pi\sqrt{-1}\theta_1+2\log(1-e^{2\pi\sqrt{-1}(\theta_2-\theta_1)})\\\nonumber
&-2\log(1-e^{2\pi\sqrt{-1}(\theta_2+\theta_1)})=-2\pi\sqrt{-1}.    
\end{align}
Comparing to the equation (\ref{equantion-geom1}), we must have
\begin{align} \label{equation-s}
    \log (w-1)+\log x+\log y-\log(y-1)=-\pi\sqrt{-1}(2\theta_1+1)
\end{align}
and 
\begin{align}
\log x+\log y=\pi\sqrt{-1}(2\theta_1-1)+\log(1-e^{2\pi\sqrt{-1}(\theta_2-\theta_1)})-\log(1-e^{2\pi\sqrt{-1}(\theta_2+\theta_1)})    
\end{align}
By equation (\ref{equation-Vu}), we have
\begin{align}
  &p2\pi\sqrt{-1}(2\theta_3-1)+\pi\sqrt{-1}(2\theta_3-1)-\log(1-e^{2\pi\sqrt{-1}(\theta_2+\theta_3)})\\\nonumber
  &+\log(1-e^{2\pi\sqrt{-1}(\theta_2-\theta_3)})=2\pi\sqrt{-1}.  
\end{align}
Comparing to equation (\ref{equantion-geom1}), we have 
\begin{align} \label{equation-u}
 \log x+\log z=2\pi\sqrt{-1}(\theta_3-1)
\end{align}
and 
\begin{align}
  &-\log(w-1)-\log x-\log z+\log(z-1)\\\nonumber
  &=2\pi\sqrt{-1}\theta_3-\log(1-e^{2\pi\sqrt{-1}(\theta_2+\theta_3)})+\log(1-e^{2\pi\sqrt{-1}(\theta_2-\theta_3)}).   
\end{align}
In conclusion, we have the correspondence 
\begin{equation}   \label{equation-corresponding}
\left\{ \begin{aligned}
         &\log (w-1)+\log x+\log y-\log(y-1)=-\pi\sqrt{-1}(2\theta_1+1), \\
          &       \log x+\log y=\pi\sqrt{-1}(2\theta_1-1)+\log(1-e^{2\pi\sqrt{-1}(\theta_2-\theta_1)})-\log(1-e^{2\pi\sqrt{-1}(\theta_2+\theta_1)}), \\
          &  \log x+\log z=2\pi\sqrt{-1}(\theta_3-1),\\
          &-\log(w-1)-\log x-\log z+\log(z-1),\\\nonumber
  &=2\pi\sqrt{-1}\theta_3-\log(1-e^{2\pi\sqrt{-1}(\theta_2+\theta_3)})+\log(1-e^{2\pi\sqrt{-1}(\theta_2-\theta_3)}). 
                          \end{aligned} \right.
                          \end{equation}

Then we will prove the following identity. 
\begin{theorem} \label{theorem-critical=volume}
Suppose $(\theta_1^0,\theta_2^0,\theta_3^0)$ is a solution of the critical point equations (\ref{equation-Vs}), (\ref{equation-Vt}) and  (\ref{equation-Vu}), and for $(p,q)\in S$, then  we have
\begin{align}
2\pi \hat{V}(p,q;\theta_1^0,\theta_2^0,\theta_3^0)=Vol(M_{p,q})+\sqrt{-1}CS(M_{p,q})+(3q-p-7)\pi^2\sqrt{-1}.    
\end{align}    
\end{theorem}

For convenience, we introduce two variables $\gamma_1=\frac{1}{p}, \gamma_2=\frac{1}{q}$.  Suppose $(\theta_1,\theta_2,\theta_3)$ is solution of the critical point equations (\ref{equation-Vs}), (\ref{equation-Vt}) and  (\ref{equation-Vu}), the the corresponding $(x,y,z,w)$ by formula (\ref{equation-corresponding}) is a solution to the geometric equation (\ref{equantion-geom1}).  Then both $\hat{V}$ and $Vol(M_{p,q})+\sqrt{-1}CS(M_{p,q})$ can be regarded as a function of $\gamma_1, \gamma_2$.  

So we define
\begin{align} \label{formula-Fgamma1gamma2}
&F(\gamma_1,\gamma_2)\\\nonumber
&=\pi^2\sqrt{-1}\left(-2\theta_1+\frac{1}{\gamma_2}\theta_1^2-2\theta_1^2-4\theta_2-2\left(\frac{2}{\gamma_1}+3\right)\theta_3+2\left(\frac{2}{\gamma_1}+1\right)u^2\right)\\\nonumber
&+\left(\frac{1}{\gamma_1}+7\right)\pi^2\sqrt{-1}+\frac{1}{\sqrt{-1}}\left(\frac{\pi^2}{6}+\text{Li}_2(e^{2\pi\sqrt{-1}(\theta_2+\theta_3)})+\text{Li}_2(e^{2\pi\sqrt{-1}(\theta_2-\theta_3)})\right.\\\nonumber
&\left.-\text{Li}_2(e^{2\pi\sqrt{-1}(\theta_2+\theta_1)})-\text{Li}_2(e^{2\pi\sqrt{-1}\theta_2})-\text{Li}_2(e^{2\pi\sqrt{-1}(\theta_2-\theta_1)})\right).    
\end{align}
and 
\begin{align}  \label{formula-Ggamma1gamma2}
G(\gamma_1,\gamma_2)&=-\frac{1}{\sqrt{-1}}\left(R(w)+R(x)+R\left(\frac{1}{1-y}\right)+R\left(\frac{1}{1-z}\right)\right)\\\nonumber
&-\frac{\pi}{2}\left(-2\pi\sqrt{-1}-(\log(w-1)+\log(x)+\log(y)-\log(y-1)+\pi\sqrt{-1})\right.\\\nonumber
&\left.+\gamma_1(2\pi\sqrt{-1}+\log(w-1)+\log(x)+\log(z)-\log(z-1)+\pi\sqrt{-1})\right)  
\end{align}

Therefore, in order to prove Thereom \ref{theorem-critical=volume}, we only need to prove the following identity. 

\begin{align} \label{identity-criticalvolume}
F(\gamma_1,\gamma_2)=G(\gamma_1,\gamma_2)    
\end{align}  

By some tedious calculations of the derivatives, we can prove that
    \begin{align}
 F_{\gamma_1}(\gamma_1,\gamma_2)=G_{\gamma_1}(\gamma_1,\gamma_2), \ F_{\gamma_2}(\gamma_1,\gamma_2)=G_{\gamma_2}(\gamma_1,\gamma_2)    
\end{align}
and $F(0,0)=G(0,0)$,  hence we finish the proof of Theorem \ref{theorem-critical=volume}.

We introduce
\begin{align}
\zeta(p,q)=\hat{V}(p,q;\theta_1^0,\theta_2^0,\theta_3^0),     
\end{align}
and 
\begin{align}
\zeta_{\mathbb{R}}(p,q)=\text{Re}\hat{V}(p,q;\theta_1^0,\theta_2^0,\theta_3^0).     
\end{align}

\begin{corollary} \label{corollary-zetaR}
For $(p,q)\in S_{p,q}$, we have
\begin{align}
2\pi \zeta_{\mathbb{R}}(p,q)>3.56337.
\end{align}    
\end{corollary}
\begin{proof}
One can compute that 
when $(p,q)\in S$, we have 
\begin{align}
    Vol\left(M_{p,q}\right)>3.56337. 
\end{align}
According to Theorem \ref{theorem-critical=volume}, we have
\begin{align}
  2\pi \zeta_{\mathbb{R}}(p,q)=Vol\left(M_{p,q}\right).   
\end{align}
\end{proof}

\subsection{Proof of Proposition \ref{prop-critical}} \label{Appendix-criticalexists}
The proof of this Proposition is similar to the proof of Proposition 5.3 in \cite{CZ23-1}. 
\begin{lemma} \label{lemma-t0s0}
    Suppose $(\theta_1^0,\theta_2^0,\theta_3^0)$ is a critical point of $\hat{V}(\theta_1,\theta_2,\theta_3)$ with $(\text{Re}(\theta_1^0),\text{Re}(\theta_2^0),\text{Re}(\theta_3^0))\in D_0$,  then we have  $(\text{Re}(\theta_1^0),\text{Re}(\theta_2^0),\text{Re}(\theta_3^0))\in U_0$. 
\end{lemma}

\begin{proof}
Suppose $\theta_i^0=\theta_{iR}^0+X_i^0\sqrt{-1}$, $(i=1,2,3)$ is a critical point of $\hat{V}$ with $(\theta_{1R}^0,\theta_{2R}^0,\theta_{3R}^0)\in D_0$, i.e.

\begin{equation} \label{formula-criticalpoint}
\left\{ \begin{aligned}
         &\frac{\partial \hat{V}}{\partial \theta_1}(\theta_{1}^0,\theta_{2}^0,\theta_{3}^0)=0,  \\
         & \frac{\partial \hat{V}}{\partial \theta_2}(\theta_{1}^0,\theta_{2}^0,\theta_{3}^0)=0. \\ 
         & \frac{\partial \hat{V}}{\partial \theta_3}(\theta_{1}^0,\theta_{2}^0,\theta_{3}^0)=0
                          \end{aligned}\right.
\end{equation}

We will prove that $(\theta_{1R}^0,\theta_{2R}^0,\theta_{3R}^0)\in U_0$. Note that 
\begin{align}
    \frac{\partial f}{\partial X_i}&=\frac{\partial \text{Re} \hat{V}(\theta_1+X_1\sqrt{-1},\theta_2+X_2\sqrt{-1},\theta_3+X_3\sqrt{-1})}{\partial X_2}\\\nonumber
    &=\text{Re} \left(\sqrt{-1}\frac{\partial \hat{V}(\theta_1+X_1\sqrt{-1},\theta_2+X_2\sqrt{-1},\theta_3+X_3\sqrt{-1})}{\partial \theta_i}\right)\\\nonumber
    &=-\text{Im} \left(\frac{\partial \hat{V}}{\partial \theta_i}\right),
\end{align}

According to the equation (\ref{formula-criticalpoint}), we obtain 
\begin{align} \label{equation:critical}
    \frac{\partial \hat{f}}{\partial X_i}(\theta_{1R}^0+X_1^0\sqrt{-1},\theta_{2R}^0+X_2^0\sqrt{-1},\theta_{3R}^0+X_3^0\sqrt{-1})&=0,  
\end{align}

By a straightforward computation, we obtain 
\begin{align}
    \frac{\partial \hat{f}}{\partial X_1}&=-\pi(q-2)\theta_1+\pi-\text{arg}(1-e^{2\pi\sqrt{-1}\left(\theta_2+\theta_1+(X_1+X_2)\sqrt{-1}\right)})\\\nonumber
    &+\text{arg}(1-e^{2\pi\sqrt{-1}\left(\theta_2-\theta_1+(X_2-X_1)\sqrt{-1}\right)}),
\end{align}

\begin{align}
    \frac{\partial \hat{f}}{\partial X_2}&=2\pi+\text{arg}(1-e^{2\pi\sqrt{-1}\left(\theta_2+\theta_3+(X_2+X_3)\sqrt{-1}\right)})\\\nonumber
    &+\text{arg}(1-e^{2\pi\sqrt{-1}\left(\theta_2-\theta_3+(X_2-X_3)\sqrt{-1}\right)})+\text{arg}(1-e^{2\pi\sqrt{-1}\left(\theta_2-\theta_1+(X_2-X_1)\sqrt{-1}\right)})\\\nonumber
    &-\text{arg}(1-e^{2\pi\sqrt{-1}\left(\theta_2+\theta_1+(X_2+X_1)\sqrt{-1}\right)})-\text{arg}(1-e^{2\pi\sqrt{-1}\left(\theta_2+X_2\sqrt{-1}\right)})\\\nonumber
    &-\text{arg}(1-e^{2\pi\sqrt{-1}\left(\theta_2-\theta_1+(X_2-X_1)\sqrt{-1}\right)}),
\end{align}

\begin{align}
\frac{\partial \hat{f}}{\partial X_3}&=-\pi(4p+2)\theta_3+(2p+3)\pi+\text{arg}(1-e^{2\pi\sqrt{-1}\left(\theta_2+\theta_3+(X_2+X_3)\sqrt{-1}\right)})\\\nonumber
&-\text{arg}(1-e^{2\pi\sqrt{-1}\left(\theta_2-\theta_3+(X_2-X_3)\sqrt{-1}\right)}).
\end{align}

Furthermore, for $0<\theta_2\pm \theta_3-1<\frac{1}{2}$, then 
\begin{align}
2\pi\left(\theta_2\pm \theta_3-\frac{3}{2}\right)<\text{arg}(1-e^{2\pi\sqrt{-1}\left(\theta_2\pm \theta_3+(X_2\pm X_3)\sqrt{-1}\right)})<0    
\end{align}
and for $\frac{1}{2}<\theta_2\pm \theta_3-1<1$, then 
\begin{align}
0<\text{arg}(1-e^{2\pi\sqrt{-1}\left(\theta_2\pm \theta_3+(X_2\pm X_3)\sqrt{-1}\right)})<2\pi\left(\theta_2\pm \theta_3-\frac{3}{2}\right)    
\end{align}

By using the equation (\ref{equation:critical}),  we get that if $(\theta_{1R}^0,\theta_{2R}^0,\theta_{3R}^0)\in D_{0}$, then $(\theta_{1R}^0,\theta_{2R}^0,\theta_{3R}^0)\in  D_{H}$. Then $(\theta_{1R}^0,\theta_{2R}^0,\theta_{3R}^0)$ must lie in $U_0$. Since if $(\theta_{1R}^0,\theta_{2R}^0,\theta_{3R}^0)\in D_{H}\setminus U_0$, by Proposition \ref{proposition-Hessianf}, the Hessian matrix  $\hat{f}(X_1,X_2,X_3)$ is positive definite on $(\theta_{1R}^0,\theta_{2R}^0,\theta_{3R}^0)$, and $(X_1^0,X_2^0,X_3^0)$ is a critical point of $\hat{f}(X_1,X_2,X_3)$. On the other hand, by Lemma , we have that $f(X_1,X_2,X_3)\rightarrow -\infty$ in some direction of $X_1^2+X_2^2+X_3^2\rightarrow \infty$. It is a contradiction. Hence  $(\theta_{1R}^0,\theta_{2R}^0,\theta_{3R}^0)\in U_0$. 
\end{proof}

Note that it is easy to generalize Lemmas 6.5-6.9 in \cite{CZ23-1}  to be used in our case.  Then we can finish the proof as in \cite{CZ23-1}.

\subsection{Estimation of the critical value}

In this section, we prove the  Lemma \ref{lemma-volumeestimate} which gives the estimation of the critical value $\zeta_{\mathbb{R}}(p,q)$. 
    Recall that $\zeta_{\mathbb{R}}(p,q)$ is given by
\begin{align}
        \zeta_{\mathbb{R}}(p,q)=\text{Re} \hat{V}(p,q,\theta_1^0,\theta_2^0,\theta_3^0), 
\end{align}
where $(\theta_1^0,\theta_2^0,\theta_3^0)$ is the unique solution of the equations
\begin{align} \label{formula-Vtheta1}
\hat{V}_{\theta_1}=\pi\sqrt{-1}(q\theta_1-2\theta_1-1)+\log(1-e^{2\pi\sqrt{-1}(\theta_2+\theta_1)})-\log(1-e^{2\pi\sqrt{-1}(\theta_2-\theta_1)})=0, 
\end{align}
\begin{align} \label{formula-Vtheta2}
\hat{V}_{\theta_2}&=-2\pi\sqrt{-1}-\log(1-e^{2\pi\sqrt{-1}(\theta_2+\theta_3)})-\log(1-e^{2\pi\sqrt{-1}(\theta_2-\theta_3)})\\\nonumber
&+\log(1-e^{2\pi\sqrt{-1}(\theta_2+\theta_1)})+\log(1-e^{2\pi\sqrt{-1}\theta_2})+\log(1-e^{2\pi\sqrt{-1}(\theta_2-\theta_1)})=0,    
\end{align}
\begin{align} \label{formula-Vtheta3}
&\hat{V}_{\theta_3}=\pi\sqrt{-1}((4p+2)\theta_3-(2p+3))-\log(1-e^{2\pi\sqrt{-1}(\theta_2+\theta_3)})\\\nonumber
&+\log(1-e^{2\pi\sqrt{-1}(\theta_2-\theta_3)})=0.
\end{align}

Putting $\gamma_1=\frac{1}{p}, \gamma_2=\frac{1}{q}$, we regard $(\theta_1^0,\theta_2^0,\theta_3^0)$ as a function of $\gamma_1,\gamma_2$, and the denote it by $(\theta_1(\gamma_1,\gamma_2),\theta_2(\gamma_1,\gamma_2),\theta_3(\gamma_1,\gamma_2))$,  then by expanding the above equations (\ref{formula-Vtheta1}), (\ref{formula-Vtheta2}) and (\ref{formula-Vtheta3}), we obtain 
\begin{align} \label{formula-thetaTaylor}
    \theta_1(\gamma_1,\gamma_2)&=\gamma_2-2(1+\sqrt{-1})\gamma_2^2+8\sqrt{-1}\gamma_2^3-\frac{\sqrt{-1}}{4}\pi^2\gamma_1^2\gamma_2^2-(1-\sqrt{-1})\pi^2\gamma_1^2\gamma_2^3+\cdots\\\nonumber
    \theta_2(\gamma_1,\gamma_2)&=\frac{\log(1-2\sqrt{-1})}{2\pi\sqrt{-1}}+1+\frac{1+2\sqrt{-1}}{40}\pi\gamma_1^2+\frac{-2+\sqrt{-1}}{5}\pi\gamma_2^2+\frac{-3+4\sqrt{-1}}{100}\pi^3\gamma_1^2\gamma_2^2+\cdots\\\nonumber
    \theta_3(\gamma_1,\gamma_2)&=\frac{1}{2}+\frac{1}{2}\gamma_1+\frac{1-\sqrt{-1}}{8}\gamma_1^2-\frac{\sqrt{-1}}{16}\gamma_1^3-\frac{\sqrt{-1}}{8}\pi^2\gamma_1^2\gamma_2^2+\cdots
\end{align}

Now, we put 
\begin{align}
    \hat{V}(\gamma_1,\gamma_2)&=\hat{V}(p,q,\theta_1(\gamma_1,\gamma_2),\theta_2(\gamma_1,\gamma_2),\theta_3(\gamma_1,\gamma_2))\\\nonumber
    &=\pi\sqrt{-1}\left(\frac{3}{2\gamma_2}+\frac{\theta_1^2}{2\gamma_2}-\theta_1^2-\theta_1+\left(\frac{2}{\gamma_1}+1\right)\theta_3^2-\left(\frac{2}{\gamma_1}+3\right)\theta_3-2\theta_2\right)\\\nonumber
&+\frac{1}{2\pi\sqrt{-1}}\left(\frac{\pi^2}{6}+\text{Li}_2(e^{2\pi\sqrt{-1}(\theta_2+\theta_3)})+\text{Li}_2(e^{2\pi\sqrt{-1}(\theta_2-\theta_3)})-\text{Li}_2(e^{2\pi\sqrt{-1}(\theta_2+\theta_1)})\right.\\\nonumber
&\left.-\text{Li}_2(e^{2\pi\sqrt{-1}\theta_2})-\text{Li}_2(e^{2\pi\sqrt{-1}(\theta_2-\theta_1)})\right),
\end{align}

then
\begin{align}
\frac{\partial \hat{V}(\gamma_1,\gamma_2)}{\partial\gamma_1}&=V_{\theta_1}\frac{d\theta_1}{d\gamma_1}+V_{\theta_2}\frac{d\theta_2}{d\gamma_1}+V_\frac{d\theta_3}{d\gamma_1}+\frac{2\pi\sqrt{-1}}{\gamma_1^2}(\theta_3-\theta_3^2)\\\nonumber
&=\frac{2\pi\sqrt{-1}}{\gamma_1^2}(\theta_3-\theta_3^2)\\\nonumber
&=\frac{\pi\sqrt{-1}}{2}\frac{1}{\gamma_1^2}-\frac{\pi\sqrt{-1}}{2}-\frac{1+\sqrt{-1}}{4}\pi \gamma_1-\frac{\pi^3}{4}\gamma_1\gamma_2^2-\frac{3}{16}\pi\gamma_1^2+\cdots    
\end{align}
\begin{align}
    \frac{\partial \hat{V}(\gamma_1,\gamma_2)}{\partial \gamma_2}&=V_{\theta_1}\frac{d\theta_1}{d\gamma_2}+V_{\theta_2}\frac{d\theta_2}{d\gamma_2}+V_{\theta_3}\frac{d\theta_3}{d\gamma_2}-\frac{\pi\sqrt{-1}}{2\gamma_2^2}(3-\theta_1^2)\\\nonumber
&=-\frac{\pi\sqrt{-1}}{2\gamma_2^2}(3-\theta_1^2)\\\nonumber
&=-\frac{3\pi\sqrt{-1}}{2}\frac{1}{\gamma_2^2}-\frac{\pi\sqrt{-1}}{2}+(-2+2\sqrt{-1})\pi\gamma_2+12\pi\gamma_2^2+
\cdots  
\end{align}

We obtain 
\begin{align}
2\pi \hat{V}(\gamma_1,\gamma_2)&=v_8-\sqrt{-1}\pi^2\frac{1}{\gamma_1}-3\sqrt{-1}\pi^2\frac{1}{\gamma_2}-\pi^2\sqrt{-1}\gamma_2\\\nonumber
&+(-2+2\sqrt{-1})\pi^2\gamma_2^2-\sqrt{-1}\pi^2\gamma_1-\frac{1+\sqrt{-1}}{4}\pi^2\gamma_1^2+\cdots,    
\end{align}
and 
\begin{align}
2\pi\text{Re}\hat{V}(\gamma_1,\gamma_2)=v_8-\frac{1}{4}\pi^2\gamma_1^2-2\pi^2\gamma_2^2-\frac{1}{8}\pi^2\gamma_1^3+8\pi^2\gamma_2^3+\cdots, 
\end{align}
where $v_8$ denotes the hyperbolic volume of the regular octahedron in $\mathbb{H}^3$. 

Note that, when $\gamma_2=0$, we obtain 
\begin{align}
2\pi \text{Re}\hat{V}(\gamma_1,0)=v_8-\frac{1}{4}\pi^2\gamma_1^2-\frac{1}{8}\pi^2\gamma_1^3+\cdots
\end{align}
which is just the formula (6.100) in (\cite{CZ23-1}). 

In conclusion, we have the following estimations for the critical value 
\begin{theorem}
 \begin{align}
        &2\pi\hat{V}(p,q,\theta_1^0,\theta_2^0,\theta_3^0)\\\nonumber
        &=v_8-\left(\sqrt{-1}\pi^2\frac{1}{\gamma_1}+3\sqrt{-1}\pi^2\frac{1}{\gamma_2}\right)-(\sqrt{-1}\pi^2\gamma_1+\sqrt{-1}\pi^2\gamma_2)\\\nonumber
        &-\left(\left(\frac{1}{4}+\frac{\sqrt{-1}}{4}\right)\pi^2\gamma_1^2+2(1-\sqrt{-1})\pi^2\gamma_2^2\right)+\left(-\frac{1}{8}\pi^2\gamma_1^3+8\pi^2\gamma_2^3\right)+\cdots
    \end{align} 
    where $\gamma_1=\frac{1}{p}, \gamma_2=\frac{1}{q}$. 
  \begin{align}
  &2\pi \text{Re} \hat{V}(p,q,\theta_1^0,\theta_2^0,\theta_3^0)\\\nonumber
  &=v_8-\left(\frac{1}{4}\pi^2\gamma_1^2+2\pi^2\gamma_2^2\right)+\left(-\frac{1}{8}\pi^2\gamma_1^3+8\pi^2\gamma_2^3\right)\\\nonumber
  &+\left(\left(-\frac{1}{32}\pi^2-\frac{1}{192}\pi^4\right)\gamma_1^4-\frac{1}{4}\pi^4\gamma_1^2\gamma_2^2+(-16\pi^2+\frac{1}{3}\pi^4)\gamma_2^4\right)+\cdots
  \end{align}  
In particular, by Theorem \ref{theorem-critical=volume}, we have
\begin{align}
   &Vol\left(S^3\setminus \mathcal{K}_p\right)\\\nonumber
   &=2\pi \text{Re}\hat{V}(p,\infty,\theta_1^0,\theta_2^0,\theta_3^0)\\\nonumber
   &=v_8-\frac{1}{4}\pi^2\gamma_1^2-\frac{1}{8}\pi^2\gamma_1^3+\left(\frac{1}{32}\pi^2+\frac{1}{192}\pi^4\right)\gamma_1^4+\cdots
\end{align}
and
  \begin{align}
      &2\pi\text{Re}\hat{V}\left(p,q,\theta_1\left(\frac{1}{2}\right),\theta_2\left(\frac{1}{2}\right),\frac{1}{2}\right)\\\nonumber
      &=2\pi \text{Re}\hat{V}\left(\infty,q,\theta_1,\theta_2,\theta_3\right)\\\nonumber
      &=v_8-2\pi^2\gamma_2^2+8\pi^2\gamma_2^3+(-16\pi^2+\frac{1}{3}\pi^4)\gamma_2^4+\cdots
  \end{align}
\end{theorem}

So we have 
\begin{align}
&2\pi \text{Re}\hat{V}(p,q,\theta_1^0,\theta_2^0,\theta_3^0)\\\nonumber &=Vol\left(S^3\setminus \mathcal{K}_p\right)-2\pi^2\gamma_2^2+8\pi^2\gamma_2^3-\frac{1}{4}\pi^4\gamma_1^2\gamma_2^2+\cdots  
\end{align}
or 
\begin{align}
&2\pi \text{Re}\hat{V}(p,q,\theta_1^0,\theta_2^0,\theta_3^0)\\\nonumber &=2\pi\text{Re}\hat{V}\left(p,q,\theta_1\left(\frac{1}{2}\right),\theta_2\left(\frac{1}{2}\right),\frac{1}{2}\right)-\frac{1}{4}\pi^2\gamma_1^2-\frac{1}{8}\pi^2\gamma_1^3+\left(-\frac{1}{32}\pi^2-\frac{1}{192}\pi^4\right)\gamma_1^4+\cdots  
\end{align}
By Theorem \ref{theorem-critical=volume}, we have
\begin{align}
&Vol\left(M_{p,q}\right)\\\nonumber
&=2\pi \text{Re}\hat{V}(p,q,\theta_1^0,\theta_2^0,\theta_3^0)\\\nonumber
  &=v_8-\left(\frac{1}{4}\pi^2\gamma_1^2+2\pi^2\gamma_2^2\right)+\left(-\frac{1}{8}\pi^2\gamma_1^3+8\pi^2\gamma_2^3\right)\\\nonumber
  &+\left(\left(-\frac{1}{32}\pi^2-\frac{1}{192}\pi^4\right)\gamma_1^4-\frac{1}{4}\pi^4\gamma_1^2\gamma_2^2+(-16\pi^2+\frac{1}{3}\pi^4)\gamma_2^4\right)+\cdots\\\nonumber
  &=2\pi\text{Re}\hat{V}\left(p,q,\theta_1\left(\frac{1}{2}\right),\theta_2\left(\frac{1}{2}\right),\frac{1}{2}\right)-\frac{1}{4}\pi^2\gamma_1^2-\frac{1}{8}\pi^2\gamma_1^3-\left(\frac{1}{32}\pi^2+\frac{1}{192}\pi^4\right)\gamma_1^4-\frac{1}{4}\pi^4\gamma_1^2\gamma_2^4+\cdots
\end{align}

From the above expansion formula, in principal, we can obtain that
\begin{theorem} \label{theorem-inprincipal1}
  \begin{align}
Vol\left(M_{p,q}\right)>2\pi\text{Re}\hat{V}\left(p,q,\theta_1\left(\frac{1}{2}\right),\theta_2\left(\frac{1}{2}\right),\frac{1}{2}\right)-\frac{49}{64}\pi^2\frac{1}{p^2}.    
\end{align}     
\end{theorem}

 Fixing $\theta_3=c_3$, we have
\begin{align}
&\hat{V}_{\theta_1}(p,q,\theta_1,\theta_2,c_3;0,0,m_3)\\\nonumber
&=\pi\sqrt{-1}(-1+(q-2)\theta_1)+\log(1-e^{2\pi\sqrt{-1}(\theta_2+\theta_1)})-\log(1-e^{2\pi\sqrt{-1}(\theta_2-\theta_1)})    
\end{align}
and
\begin{align}
&\hat{V}_{\theta_2}(p,q,\theta_1,\theta_2,c_3;0,0,m_3)\\\nonumber
&=-2\pi\sqrt{-1}+\log(1-e^{2\pi\sqrt{-1}(\theta_2+\theta_1)})+\log(1-e^{2\pi\sqrt{-1}(\theta_2-\theta_1)})\\\nonumber
&+\log(1-e^{2\pi\sqrt{-1}\theta_2})-\log(1-e^{2\pi\sqrt{-1}(\theta_2+c_2)})-\log(1-e^{2\pi\sqrt{-1}(\theta_2-c_2)}).    
\end{align}

Let $c_3=\frac{1}{2}+\tilde{c}_3$, we solve the critical point equations 
\begin{align}
\left\{ \begin{aligned}
         &\hat{V}_{\theta_1}(p,q,\theta_1,\theta_2,\frac{1}{2}+\tilde{c}_3;0,0,m_3)=0 \\
         &\hat{V}_{\theta_2}(p,q,\theta_1,\theta_2,\frac{1}{2}+\tilde{c}_3;0,0,m_3)=0
                          \end{aligned} \right.
\end{align}

 As a function of $\gamma_2=\frac{1}{q},\tilde{c}_3$, the solutions $\theta_1,\theta_2$ of the above critical point equation have the following expansions 
\begin{align}
\theta_1&=\gamma_2-2(1+\sqrt{-1})\gamma_2^2+8\sqrt{-1}\gamma_2^3+(16-16\sqrt{-1}+\frac{2\sqrt{-1}}{3}\pi^2)\gamma_2^4+\cdots \\\nonumber
&+\left(-\sqrt{-1}\pi^2\gamma_2^2+(-4+4\sqrt{-1})\pi^2\gamma_2^3+(24\pi^2-\frac{\sqrt{-1}}{3}\pi^4)\gamma_2^4\right.\\\nonumber
&\left.+(-(64+64\sqrt{-1})\pi^2+\frac{8\sqrt{-1}}{3}\pi^4)\gamma_2^5+\cdots\right) \tilde{c}_3^2+\cdots
\end{align}
and 
\begin{align}
\theta_2&=\frac{\log(1-2\sqrt{-1})}{2\pi\sqrt{-1}}+1+(-\frac{2}{5}+\frac{\sqrt{-1}}{5})\pi\gamma_2^2+(\frac{12}{5}+\frac{4\sqrt{-1}}{5})\pi\gamma_2^3+\cdots\\\nonumber
&+\left((\frac{1}{10}+\frac{\sqrt{-1}}{5})+(-\frac{3}{25}+\frac{4\sqrt{-1}}{25})\pi^3\gamma_2^2+(\frac{38}{25}+\frac{16\sqrt{-1}}{25})\pi^3\gamma_2^3+\cdots\right)\tilde{c}_3^2+\cdots
\end{align}

Then we obtain
\begin{align}
&2\pi \hat{V}(p,q,\theta_1,\theta_2,\frac{1}{2}+\tilde{c}_3;0,0,m_3)\\\nonumber
&=v_8+\frac{3\sqrt{-1}}{\gamma_2}\pi^2+2\pi^2\sqrt{-1}(-2m_3(\frac{1}{2}+\tilde{c}_2))-\sqrt{-1}\pi^2\gamma_2+(-2+2\sqrt{-1})\pi^2\gamma_2^2+8\pi^2\gamma_2^3+\cdots\\\nonumber
&+\left((-1-\sqrt{-1})\pi^2+4\sqrt{-1}p\pi^2-\pi^4\gamma_2^2+(4+4\sqrt{-1})\pi^4\gamma_2^3+\cdots \right)\tilde{c}_3^2+\cdots
\end{align}

In particular, when $c_3=\frac{1}{2}$,
\begin{align}
 &2\pi \hat{V}(p,q,\theta_1(\gamma_2,\frac{1}{2}),\theta_2(\gamma_2,\frac{1}{2}),c_3=\frac{1}{2};0,0,m_3)\\\nonumber
 &=v_8+\frac{3\sqrt{-1}}{\gamma_2}\pi^2+2\pi^2\sqrt{-1}(-2m_3(\frac{1}{2}+\tilde{c}_3))-\sqrt{-1}\pi^2\gamma_2+(-2+2\sqrt{-1})\pi^2\gamma_2^2+8\pi^2\gamma_2^3+\cdots
\end{align}
and 
\begin{align}
&2\pi\text{Re}\hat{V}\left(p,q,\theta_1\left(\frac{1}{2}\right),\theta_2\left(\frac{1}{2}\right),\frac{1}{2}\right)\\\nonumber
&=2\pi \text{Re}\hat{V}(p,q,\theta_1,\theta_2,c_2=\frac{1}{2};0,0,m_3)\\\nonumber
&=v_8-2\pi^2\gamma_2^3+8\pi^2\gamma_2^3-(16\pi^2-\frac{1}{3}\pi^4)\gamma_2^4+\cdots
\end{align}

Therefore, we get
\begin{align}
 &2\pi\text{Re}\hat{V}(p,q,\theta_1,\theta_2,c_3=\frac{1}{2}+\tilde{c}_3;0,0,m_3)\\\nonumber
 &=2\pi\text{Re}\hat{V}\left(p,q,\theta_1\left(\frac{1}{2}\right),\theta_2\left(\frac{1}{2}\right),\frac{1}{2}\right)+\left(-\pi^2-\pi^4\gamma_2^2+4\pi^4\gamma_2^3-\frac{1}{6}\pi^6\gamma_2^4+\frac{4}{3}\pi^6\gamma_2^5+\cdots \right)\tilde{c}_3^2+\cdots   
\end{align}

In principal, we have the following estimation
 
\begin{theorem} \label{theorem-inprincipal2}
\begin{align}
2\pi \text{Re}\hat{V}(p,q; \theta_1,\theta_2,c_3=\frac{1}{2}+\tilde{c}_3;0,0,m_3)<2\pi\text{Re}\hat{V}\left(p,q,\theta_1\left(\frac{1}{2}\right),\theta_2\left(\frac{1}{2}\right),\frac{1}{2}\right)-\pi^2 \tilde{c}_3^2.    
\end{align}    
\end{theorem}

From the formula 
\begin{align}
&2\pi \text{Re}\hat{V}(p,q,\theta_1^0,\theta_2^0,\theta_3^0)\\\nonumber &=Vol\left(S^3\setminus \mathcal{K}_p\right)-2\pi^2\gamma_2^2+8\pi^2\gamma_a^3-\frac{1}{4}\pi^4\gamma_1^2\gamma_2^2+\cdots,   
\end{align}
 in principal, we can prove that

 \begin{theorem} \label{theorem-inprincipal3}
     \begin{align}
Vol\left(M_{p,q}\right)>Vol\left(S^3\setminus \mathcal{K}_p\right)-6\pi^2\gamma_2^2.    
\end{align}
 \end{theorem}

Fixing $\theta_1=c_1$, we have
\begin{align}
&\hat{V}_{\theta_2}(p,q,c_1,\theta_2,\theta_3;m_1,0,0)\\\nonumber
&=-2\pi\sqrt{-1}+\log(1-e^{2\pi\sqrt{-1}(\theta_2-c_1)})+\log(1-e^{2\pi\sqrt{-1}\theta_2})\\\nonumber
&+\log(1-e^{2\pi\sqrt{-1}(\theta_2+c_1)})-\log(1-e^{2\pi\sqrt{-1}(\theta_2+\theta_3)})-\log(1-e^{2\pi\sqrt{-1}(\theta_2-\theta_3)}),   
\end{align}
and 
\begin{align}
&\hat{V}_{\theta_3}(p,q,c_1,\theta_2,\theta_3;m_1,0,0)\\\nonumber
&=\pi\sqrt{-1}(-(2p+3)+(4p+2)\theta_3)-\log(1-e^{2\pi\sqrt{-1}(\theta_2+\theta_3)})\\\nonumber
&+\log(1-e^{2\pi\sqrt{-1}(\theta_2-\theta_3}).   
\end{align}
Consider the critical point equations 
\begin{align}
 \left\{ \begin{aligned}
         & \hat{V}_{\theta_2}(p,q,c_1,\theta_2,\theta_3;m_1,0,0)&=0, \\
         &\hat{V}_{\theta_3}(p,q,c_1,\theta_2,\theta_3;m_1,0,0)&=0.
                          \end{aligned} \right.
\end{align}

 As a function of $\gamma_1=\frac{1}{p},\tilde{c}_1$, the solutions $\theta_2,\theta_3$ of the above critical point equations have the following expansions 
\begin{align}
\theta_2&=\frac{\log(1-2\sqrt{-1})}{2\pi\sqrt{-1}}+1+\frac{1+2\sqrt{-1}}{40}\pi \gamma_1^2+\frac{3+\sqrt{-1}}{80}\pi\gamma_1^3+\cdots\\\nonumber
&+\left(\frac{-2+\sqrt{-1}}{5}\pi+\frac{-3+4\sqrt{-1}}{100}\pi^3\gamma_1^2+\frac{12+9\sqrt{-1}}{400}\pi^3\gamma_1^3+\cdots \right)\tilde{c}_1^2+\cdots
\end{align}
\begin{align}
 \theta_3&=\frac{1}{2}+\frac{1}{2}\gamma_1+\frac{1-\sqrt{-1}}{8}\gamma_1^2-\frac{\sqrt{-1}}{16}\gamma_1^3-(\frac{1+\sqrt{-1}}{64}+\frac{\sqrt{-1}}{192}\pi^2)\gamma_1^4+\cdots \\\nonumber
 &+\left(-\frac{\sqrt{-1}}{8}\pi^2\gamma_1^2-\frac{1+\sqrt{-1}}{16}\pi^2\gamma_1^3-(\frac{3}{64}\pi^2+\frac{5\sqrt{-1}}{192}\pi^4)\gamma_1^4+\cdots \right)\tilde{c}_1^2+\cdots
\end{align}

Then we have
\begin{align}
  &2\pi \hat{V}(p,q,c_1=\tilde{c}_1,\theta_1,\theta_2;m_1,0,0)\\\nonumber
  &=v_8-\frac{\sqrt{-1}\pi^2}{\gamma_1}+3q\sqrt{-1}\pi^2-\sqrt{-1}\pi^2\gamma_1-\frac{1+\sqrt{-1}}{4}\pi^2\gamma_1^{2}-\frac{1}{8}\pi^2\gamma_1^3+\cdots \\\nonumber
  &-(1+2m_1)2\sqrt{-1}\pi^2\tilde{c}_1+\left((-2+(2+q)\sqrt{-1})\pi^2-\frac{1}{4}\pi^4\gamma_1^2+\frac{-1+\sqrt{-1}}{8}\pi^4\gamma_1^3+\cdots\right)\tilde{c}_1^2+\cdots
\end{align}

In particular, $c_1=\tilde{c}_1=0$, we have
\begin{align}
&2\pi \hat{V}(p,q,c_1=\tilde{c}_1,\theta_1,\theta_2;m_1,0,0)\\\nonumber
&=v_8-\frac{\sqrt{-1}\pi^2}{\gamma_1}-\sqrt{-1}\pi^2\gamma_1-\frac{1+\sqrt{-1}}{4}\pi^2\gamma_1^2-\frac{1}{8}\pi^2\gamma_1^3+\cdots. 
\end{align}
It follows that
\begin{align}
&2\pi\text{Re}\hat{V}(p,q,c_1=\tilde{c}_1,\theta_2,\theta_3;m_1,0,0)\\\nonumber
&=Vol(S^3\setminus \mathcal{K}_p)+(-2\pi^2-\frac{1}{4}\pi^4\gamma_1^2-\frac{1}{8}\pi^4\gamma_1^3-\frac{5}{192}\pi^6\gamma_1^4+\cdots)\tilde{c}_1^2+\\\nonumber
&+(\frac{1}{3}\pi^4-\frac{1}{24}\pi^6\gamma_1^2-\frac{1}{48}\pi^6\gamma_1^3+\cdots )\tilde{c}_1^4+\cdots. 
\end{align}

In principal, we have
\begin{theorem} \label{theorem-inprincipal4}
\begin{align}
2\pi \text{Re}\hat{V}(p,q; c_1,\theta_2,\theta_3;m_1,0,0)<Vol\left(S^3\setminus \mathcal{K}_p\right)-\frac{3}{2}\pi^2 c_1^2.    
\end{align}    
for 
$
c_1<\frac{1}{4}. 
$
\end{theorem}

\subsection{Two-dimensional saddle point method} \label{Appendix-twodimsaddle}

We consider the potential function
\begin{align}
&\hat{V}(p,q;\theta_1,\theta_2,\theta_3;m_1,m_2,m_3)\\\nonumber
&=\pi\sqrt{-1}\left(\frac{3q}{2}+\left(\frac{q}{2}-1\right)\theta_1^2-(2m_1+1)\theta_1-(2m_2+2)\theta_2+(2p+1)\theta_3^2-(2p+3+2m_3)\theta_3\right)\\\nonumber
&+\frac{1}{2\pi\sqrt{-1}}\left(\frac{\pi^2}{6}+\text{Li}_2(e^{2\pi\sqrt{-1}(\theta_2+\theta_3)})+\text{Li}_2(e^{2\pi\sqrt{-1}(\theta_2-
\theta_3)})-\text{Li}_2(e^{2\pi\sqrt{-1}(\theta_2+\theta_1)})\right.\\\nonumber
&\left.-\text{Li}_2(e^{2\pi\sqrt{-1}\theta_2})-\text{Li}_2(e^{2\pi\sqrt{-1}(\theta_2-\theta_1)})\right).
\end{align}

\subsubsection{The case $\theta_3=c_3$}
Fixing $\theta_3=c_3$, then
\begin{align}
\hat{V}_{\theta_1}=\pi\sqrt{-1}(-2m_1-1+q\theta_1-2\theta_1)+\log(1-e^{2\pi\sqrt{-1}(\theta_2+\theta_1)})-\log(1-e^{2\pi\sqrt{-1}(\theta_2-\theta_1)})    
\end{align}
\begin{align}
\hat{V}_{\theta_2}&=-(2m_2+2)\pi\sqrt{-1}+\log(1-e^{2\pi\sqrt{-1}(\theta_2+\theta_1)})+\log(1-e^{2\pi\sqrt{-1}(\theta_2-\theta_1)})\\\nonumber
&+\log(1-e^{2\pi\sqrt{-1}\theta_2})-\log(1-e^{2\pi\sqrt{-1}(\theta_2+c_3)})-\log(1-e^{2\pi\sqrt{-1}(\theta_2-c_3)})    
\end{align}

Let $z_1=e^{2\pi\sqrt{-1}\theta_1}$, $z_2=e^{2\pi\sqrt{-1}t}$, $C_3=e^{2\pi\sqrt{-1}c_3}$, then 
from the critical point equations 
\begin{align}
\hat{V}_{\theta_1}=0, \ \hat{V}_{\theta_2}=0,    
\end{align}
we obtain 
\begin{align} 
&z_2=\frac{z_1+z_1^{\frac{q}{2}}}{1+z_1^{\frac{q}{2}+1}}, \\
&z_2^2-(z_1+\frac{1}{z_1})z_2+z_1+\frac{1}{z_1}-C_3-\frac{1}{C_3}+1=0. 
\end{align}
It follows that
\begin{align} \label{equation-z1}
&z_1^{q+4}-(C_3+C_3^{-1})z_1^{q+3}+z_1^{q+2}-z_1^{\frac{q}{2}+4}+2z_1^{\frac{q}{2}+3}\\\nonumber
&-2(C_3+C_3^{-1}-1)z_1^{\frac{q}{2}+2}+2z_1^{\frac{q}{2}+1}-z_1^{\frac{q}{2}}+z_1^2-(C_3+C_3^{-1})z_1+1=0.    
\end{align}

Suppose $z_1(q,c_3)$ is a solution of the equation (\ref{equation-z1}),   set 
\begin{align}
\theta_1(q,c_3)=\frac{\log(z_1(q,C_3))}{2\pi\sqrt{-1}}+\mathbb{Z}, \\\nonumber
\theta_2(q,c_3)=\frac{\log(z_2(q,C_3))}{2\pi\sqrt{-1}}+\mathbb{Z}, 
\end{align}
then $(\theta_1(q,c_3),\theta_2(q,c_3))$ satisfies the critical point equations up to an argument.

Moreover, by using the method in \cite{CZ23-1}, we can show that the critical point equations 
\begin{align}
    \left\{ \begin{aligned}
         & \hat{V}_{\theta_1}(\theta_1,\theta_2,c_3)=0,  &  \  \\
         &   \hat{V}_{\theta_2}(\theta_1,\theta_2,c_3)=0.
                          \end{aligned}\right.
\end{align}
has a unique solution denoted by $(\theta_1(c_3),\theta_3(c_3))$.

\begin{proposition} \label{prop-twocritical}
For  $\frac{1}{2}+\frac{1}{8p}<c_3=\frac{1}{2}+\tilde{c}_3<\frac{3}{4}$,   we have 
    \begin{align}
2\pi \text{Re}\hat{V}(p,q;\theta_1(c_3),\theta_2(c_3),c_3)<2\pi\text{Re}\hat{V}(p,q;\theta_1^0,\theta_2^0,\theta_3^0)-\epsilon  
\end{align}
for some  $\epsilon>0$. 
\end{proposition}
\begin{proof}
For $\frac{1}{2}+\frac{1}{8p}<c_3=\frac{1}{2}+\tilde{c}_3<\frac{3}{4}$, by Theorems \ref{theorem-inprincipal1} and \ref{theorem-inprincipal2}, we have

\begin{align}
   &2\pi \text{Re}\hat{V}(p,q;\theta_1(q,c_3),\theta_2(q,c_3),c_3)\\\nonumber
   &\leq 2\pi\text{Re}\hat{V}\left(p,q,\theta_1\left(\frac{1}{2}\right),\theta_2\left(\frac{1}{2}\right),\frac{1}{2}\right)-\pi^2 \tilde{c}_3^2\\\nonumber
  &<2\pi\text{Re}\hat{V}\left(p,q,\theta_1\left(\frac{1}{2}\right),\theta_2\left(\frac{1}{2}\right),\frac{1}{2}\right)-\frac{49}{64}\pi^2\gamma_1^2\\\nonumber
  &<Vol\left(M_{p,q}\right)=2\pi \text{Re}\hat{V}(p,q,\theta_1^0,\theta_2^0,\theta_3^0).
\end{align}
\end{proof}

Let $\theta_3=c_3$ and $D_0(c_3)=\{\theta_3=c_3\}\cap D_0$,  we need to estimate the following integral used in Theorem \ref{theorem-m_1},

\begin{align}
    \int_{D_0(c_3)}\sin(\pi \theta_1)\sin(2\pi c_3)e^{(N+\frac{1}{2})\hat{V}_N(\theta_1,\theta_2,c_3;0,0,m_3)}d\theta_1d\theta_2.
\end{align}

\begin{theorem} \label{theorem-usedinm1}
For $\frac{1}{2}+\frac{1}{8p}<c_3=\frac{1}{2}+\tilde{c}_3<\frac{3}{4}$, we have
\begin{align}
    \int_{D_0(c_3)}\sin(\pi \theta_1)\sin(2\pi c_3)e^{(N+\frac{1}{2})\hat{V}_N(\theta_1,\theta_2,c_3;0,0,m_3)}d\theta_1d\theta_2=O(e^{(N+\frac{1}{2})(\zeta_{\mathbb{R}}(p,q)-\epsilon)})
\end{align}
for $\epsilon>0$.
\end{theorem}
\begin{proof}
On the region $D_0(c_3)$,  we can check the saddle point method condition for two-dimensional integral following the procedure in \cite{CZ23-1}, we omit the details here.  Then, together with Proposition \ref{prop-twocritical}, we prove the Theorem.  
\end{proof}

\subsubsection{The case $\theta_1=c_1$}
In this case, let $\theta_1=c_1$ and $D_0(c_1)=\{\theta_1=c_1\}\cap D_0$,  we need to estimate the following integral in Theorem \ref{theorem-m3},

\begin{align}
    \int_{D_0(c_1)}\sin(\pi c_1)\sin(2\pi \theta_3)e^{(N+\frac{1}{2})\hat{V}_N(c_1,\theta_2,\theta_3;0,0,m_3)}d\theta_2d\theta_3.
\end{align}

Similarly to previous subsection, we have 
\begin{theorem} \label{theorem-usedinm3}
For $\frac{2}{q}<c_1<\frac{1}{4}$, we have
\begin{align}
     \int_{D_0(c_1)}\sin(\pi c_1)\sin(2\pi \theta_3)e^{(N+\frac{1}{2})\hat{V}_N(c_1,\theta_2,\theta_3;0,0,m_3)}d\theta_2d\theta_3=O(e^{(N+\frac{1}{2})(\zeta_{\mathbb{R}}(p,q)-\epsilon)})
\end{align}
for $\epsilon>0$.
\end{theorem}

\begin{lemma} \label{Lemma-inequ}
  When $c_3=\frac{1}{2}+\tilde{c}_3\geq \frac{1}{2}+\frac{7}{8p}$, we have
  \begin{align}
   &2\pi\text{Re}\hat{V}(p,q;m_1,0,m_3,c_1,T_2(c_1,\frac{1}{2}+\tilde{c}_3),\frac{1}{2}+\tilde{c}_3)\\\nonumber
   &<2\pi \text{Re}\hat{V}(p,q,m_1,0,0,c_1,\theta_2(c_1),\theta_3(c_1)).    
  \end{align}
\end{lemma}
\begin{proof}
First, we have the estimates, 
\begin{align}
&2\pi \text{Re}\hat{V}(p,q,m_1,0,m_3,c_1,T_2(c_1,\frac{1}{2}+\tilde{c}_3),\frac{1}{2}+c_3)\\\nonumber
&=v_8-2\pi^2c_1^2-\pi^2\tilde{c}_3^2+\frac{1}{3}\pi^{4}c_{1}^{4}-\pi
^{4}c_{1}^{2}\tilde{c}_{3}^{2}-\frac{1}{12}\pi^{4}\tilde{c}_{3}^{4}+...\\\nonumber
&=v_{8}-2\pi^{2}c_{1}^{2}+\frac{1}{3}\pi^{4}c_{1}^{4}+...\\\nonumber
&-\pi^{2}\tilde{c}_{3}^{2}-\pi^{4}c_{1}^{2}\tilde{c}_{3}^{2}-\frac{1}{12}\pi^{4}\tilde{c}_{3}^{4}+...\\\nonumber
&=2\pi\text{Re}\hat{V}(p,q,m_1,0,m_3,c_{1},T_{2}(c_{1}
,\frac{1}{2}),\frac{1}{2})-\pi^{2}\tilde{c}_{3}^{2}-\pi^{4}c_{1}^{2}\tilde{c}_{3}
^{2}-\frac{1}{12}\pi^{4}\tilde{c}_{3}^{4}+...
\end{align}
On the other hand, 
\begin{align} \label{formula-inprincipal}
 &2\pi\text{Re}\hat{V}(p,q,m_{1},0,0,c_{1},\theta_2(\gamma_{1}
,c_{1}),\theta_3(\gamma_{1},c_{1})) \\\nonumber
&=v_{8}-\frac{1}{4}\pi^{2}\gamma_{1}^{2}-\frac{1}{8}\pi^{2}\gamma_{1}%
^{3}-(\frac{1}{32}\pi^{2}+\frac{1}{192}\pi^{4})\gamma_{1}^{4}+\\\nonumber &+(-2\pi^{2}-\frac{1}{4}\pi^{4}\gamma_{1}^{2}-\frac{1}{8}\pi^{4}\gamma_{1}%
^{3}-\frac{5}{192}\pi^{6}\gamma_{1}^{4}+...)c_{1}^{2}\\\nonumber&+(\frac{1}{3}\pi
^{4}-\frac{1}{24}\pi^{6}\gamma_{1}^{2}-\frac{1}{48}\pi^{6}\gamma_{1}%
^{3}+(\frac{3}{64}\pi^{6}-\frac{23}{1152}\pi^{8})\gamma_{1}^{4}+...)c_{1}
^{4}+...\\\nonumber
&=v_{8}-2\pi^{2}c_{1}^{2}+\frac{1}{3}\pi^{4}c_{1}^{4}+...\\\nonumber
&-\frac{1}{4}\pi^{2}\gamma_{1}^{2}-\frac{1}{8}\pi^{2}\gamma_{1}^{3}-(\frac
{1}{32}\pi^{2}+\frac{1}{192}\pi^{4})\gamma_{1}^{4}+(-\frac{1}{4}\pi^{4}\gamma_{1}^{2}-\frac{1}{8}\pi^{4}\gamma_{1}^{3}-\frac
{5}{192}\pi^{6}\gamma_{1}^{4}+...)c_{1}^{2}\\\nonumber
&+(-\frac{1}{24}\pi^{6}\gamma
_{1}^{2}-\frac{1}{48}\pi^{6}\gamma_{1}^{3}+(\frac{3}{64}\pi^{6}-\frac
{23}{1152}\pi^{8})\gamma_{1}^{4}+...)c_{1}^{4}+...\\\nonumber
&=2\pi\text{Re}\hat{V}(\infty,q,m_{1},0,0,c_{1},\theta_2(\gamma
_{1}=0,c_{1}),\theta_3(\gamma_{1}=0,c_{1}))-\frac{1}{4}\pi^{2}\gamma_{1}^{2}-\frac
{1}{8}\pi^{2}\gamma_{1}^{3}-(\frac{1}{32}\pi^{2}+\frac{1}{192}\pi^{4}%
)\gamma_{1}^{4}\\\nonumber &+(-\frac{1}{4}\pi^{4}\gamma_{1}^{2}-\frac{1}{8}\pi^{4}\gamma_{1}^{3}-\frac
{5}{192}\pi^{6}\gamma_{1}^{4}+...)c_{1}^{2}+(-\frac{1}{24}\pi^{6}\gamma
_{1}^{2}-\frac{1}{48}\pi^{6}\gamma_{1}^{3}+(\frac{3}{64}\pi^{6}-\frac
{23}{1152}\pi^{8})\gamma_{1}^{4}+...)c_{1}^{4}+...\\\nonumber
&=2\pi\operatorname{Re}\widehat{V}(p,q,m_{1},0,m_3,c_{1},T_{2}(c_{1}
,\frac{1}{2}),\frac{1}{2})-\frac{1}{4}\pi^{2}\gamma_{1}^{2}-\frac{1}{8}
\pi^{2}\gamma_{1}^{3}-(\frac{1}{32}\pi^{2}+\frac{1}{192}\pi^{4})\gamma_{1}%
^{4}\\\nonumber
&+(-\frac{1}{4}\pi^{4}\gamma_{1}^{2}-\frac{1}{8}\pi^{4}\gamma_{1}^{3}-\frac
{5}{192}\pi^{6}\gamma_{1}^{4}+...)c_{1}^{2}+(-\frac{1}{24}\pi^{6}\gamma
_{1}^{2}-\frac{1}{48}\pi^{6}\gamma_{1}^{3}+(\frac{3}{64}\pi^{6}-\frac
{23}{1152}\pi^{8})\gamma_{1}^{4}+...)c_{1}^{4}+...
\end{align}
Thus finally we have
\begin{align}
    &2\pi \text{Re}\hat{V}(p,q,m_1,0,m_3,c_1,T_2(c_1,\frac{1}{2}+\tilde{c}_3),\frac{1}{2}+c_3)\\\nonumber
    &=2\pi\text{Re}\hat{V}(p,q,m_{1},0,m_3,c_{1},T_{2}(c_{1}
,\frac{1}{2}),\frac{1}{2})-\pi^{2}\tilde{c}_{3}^{2}-\pi^{4}c_{1}^{2}\tilde{c}_{3}
^{2}-\frac{1}{12}\pi^{4}\tilde{c}_{3}^{4}+...\\\nonumber
&<2\pi\text{Re}\hat{V}(p,q,m_{1},0,m_3,c_{1},T_{2}(c_{1}
,\frac{1}{2}),\frac{1}{2})-\pi^{2}\tilde{c}_{3}^{2}\\\nonumber
&<2\pi\text{Re}
\hat{V}(p,q,m_{1},0,m_3,c_{1},T_{2}(c_{1},\frac{1}{2}),\frac{1}
{2})-\frac{49}{64}\pi^{2}\gamma_{1}^{2}.
\end{align}
In principal, by formula (\ref{formula-inprincipal}), we can prove the Lemma. 
\end{proof}

\subsection{One-dimensional saddle point method}  \label{appendix-onedim}
Consider the function, 
\begin{align}
 \hat{V}(\theta_1,\theta_2,\theta_3;m_1,m_2,m_3)=\hat{V}(\theta_1,\theta_2,\theta_3)-2\pi\sqrt{-1}m_1\theta_1-2\pi\sqrt{-1}m_2\theta_2-2\pi\sqrt{-1}m_3\theta_3.         
\end{align}

Let $\theta_1=c_1$, $\theta_3=c_3$ with $(c_1,c_3)\in [0,\frac{1}{4}]\times [\frac{1}{2},\frac{3}{4}]$ then 
\begin{align}
    &\hat{V}_{\theta_2}(c_1,\theta_2,c_3;m_1,m_2,m_3)\\\nonumber
    &=-2\pi\sqrt{-1}(m_2+1)-\log(1-e^{2\pi\sqrt{-1}(\theta_2-c_3)})-\log(1-e^{2\pi\sqrt{-1}(\theta_2+c_3)})\\\nonumber
&+\log(1-e^{2\pi\sqrt{-1}(\theta_2+c_1)})+\log(1-e^{2\pi\sqrt{-1}\theta_2})+\log(1-e^{2\pi\sqrt{-1}(\theta_2-c_1)}).    
\end{align}

From the critical point equation 
\begin{align} \label{equation-criticalequ-onedim0}
    V_{\theta_2}(c_1,\theta_2,c_3,m_1,m_2,m_3)=0,
\end{align} 
we obtain 
\begin{align}
(1-xC_3)(1-\frac{x}{C_3})=(1-x)(1-xC_1)(1-\frac{x}{C_1}),   
\end{align}
where we set $C_1=e^{2\pi\sqrt{-1}c_1}$ and $C_3=e^{2\pi\sqrt{-1}c_3}$. 
Hence, we obtain 
\begin{align} \label{equation-criticalequ-onedim}
 x^2-(C_1+\frac{1}{C_1})x+(C_1+\frac{1}{C_1})-(C_3+\frac{1}{C_3})+1=0.   
\end{align}
It follows that
\begin{align}
x=\cos(2\pi c_1)\pm 2\sqrt{\sin^4(\pi c_1)-\sin^2(\pi c_3)}.    
\end{align}

We set 
\begin{align}
T_2(c_1,c_3)=\frac{\log(\cos(2\pi c_1)-2\sqrt{\sin^4(\pi c_1)- \sin^2(\pi c_3)})}{2\pi\sqrt{-1}}+1,   
\end{align}
then
\begin{align} \label{equation-eT2}
    e^{2\pi\sqrt{-1}T_{2}(c_1,c_3)}=\cos(2\pi c_1)-  2\sqrt{\sin^4(\pi c_1)- \sin^2(\pi c_3)}
\end{align}
is a solution to the equation (\ref{equation-criticalequ-onedim}), moreover, $\theta_2(c_1,c_2)=T_2(c_1,c_3)$ is the unique solution  to critical point equation (\ref{equation-criticalequ-onedim0}) with $\text{Re}(\theta_2(c_1,c_2))\in (\frac{1}{2},1)$. 

From  the equation (\ref{equation-eT2}), we obtain
\begin{align}
    &e^{2\pi\sqrt{-1}T_2(c_1,c_3)}2\pi\sqrt{-1}\frac{dT_2(c_1,c_3)}{dc_1}\\\nonumber
    &=-2\pi\sin(2\pi c_1)-\frac{4\pi \sin^3(\pi c_3)\cos(\pi c_3)}{\sqrt{\sin^4(\pi c_1)-\sin^2(\pi c_3)}}\\\nonumber
    &=-2\pi\sin(2\pi c_1)\left(1+\frac{\sin^2(\pi c_1)}{\sqrt{\sin^4(\pi c_1)-\sin^2(\pi c_3)}}\right)
\end{align}
and 
\begin{align}  e^{2\pi\sqrt{-1}T_2(c_1,c_3)}2\pi\sqrt{-1}\frac{dT_2(c_1,c_3)}{dc_3}=\frac{\pi\sin(2\pi c_3)}{\sqrt{\sin^4(\pi c_1)-\sin^2(\pi c_3)}}.
\end{align}

Hence
\begin{align}
    \frac{dT_2(c_1,c_3)}{dc_1}=\frac{2\sqrt{-1}\sin(\pi c_1)\left(\cos(\pi c_1)+\frac{\sin^2(\pi c_1)}{\sqrt{\sin^4(\pi c_1)-\sin^2(\pi c_3)}}\right)}{\cos(2\pi c_1)-2\sqrt{\sin^4(\pi c_1)-\sin^2(\pi c_3)}},
\end{align}
and
\begin{align}
    \frac{dT_2(c_1,c_3)}{dc_3}=\frac{-\sqrt{-1}\sin(\pi c_3)\cos(\pi c_3)}{\cos(2\pi c_1)\sqrt{\sin^4(\pi c_1)-\sin^2(\pi c_2)}-2\sin^4(\pi c_3)+2\sin^2(\pi c_3)}.
\end{align}

It follows that
\begin{align}  \label{formula-dReVdc1}
\frac{d \text{Re}\hat{V}}{dc_1}&=\text{Re}\left(\frac{\partial \hat{V}}{\partial \theta_1}\frac{dc_1}{dc_1}+\frac{\partial \hat{V}}{\partial \theta_2}\frac{dT_2(c_1,c_3)}{dc_1}+\frac{\partial \hat{V}}{\partial \theta_3}\frac{dc_3}{dc_1} \right) \\\nonumber
&=\text{Re}\left(\frac{\partial \hat{V}}{\partial \theta_1}\right)\\\nonumber
&=\text{Re}(\log(1-e^{2\pi\sqrt{-1}(T_2(c_1,c_3)+c_1)}))-\text{Re}(\log(1-e^{2\pi\sqrt{-1}(T_2(c_1,c_3)-c_1)}))\\\nonumber
&=\log\left(\frac{2\sqrt{\sin^2(\pi c_3)-\sin^4(\pi c_1)}-\sin(2\pi c_1)}{2\sqrt{\sin^2(\pi c_3)-\sin^4(\pi c_1)}+\sin(2\pi c_1)}\right)
\end{align}
By the following Lemma \ref{lemma-c1c3}.

\begin{lemma} \label{lemma-c1c3}
For $(c_1,c_3)\in [0,\frac{1}{4}]\times [\frac{1}{2},\frac{3}{4}]$, we have
 \begin{align}
 \text{Re}\left(\log(1-e^{2\pi\sqrt{-1}(T_2(c_1,c_3)+c_1)})\right)=\log\left(2\sqrt{\sin^2(\pi c_3)-\sin^4(\pi c_1)}-\sin(2\pi c_1)\right)    
 \end{align}   
 and 
 \begin{align}
 \text{Re}\left(\log(1-e^{2\pi\sqrt{-1}(T_2(c_1,c_3)-c_1)})\right)=\log\left(2\sqrt{\sin^2(\pi c_3)-\sin^4(\pi c_1)}+\sin(2\pi c_1)\right)   
 \end{align}   
\end{lemma}
\begin{proof}
Since $\sin^4(\pi c_1)-\sin^2(\pi c_3)\leq 0$ by $(c_1,c_3)\in [0,\frac{1}{4}]\times [\frac{1}{2},\frac{3}{4}]$. 
By a straightforward computation, 
\begin{align}
&\text{Re}\left(\log(1-e^{2\pi\sqrt{-1}(T_2(c_1,c_3)+c_1)})\right) \\\nonumber
&=\text{Re}\left(\log\left(1-\left(\cos(2\pi c_1)-2\sqrt{\sin^4(\pi c_1)-\sin^2(\pi c_3)}\right)e^{2\pi\sqrt{-1}c_1}\right)\right)\\\nonumber
&=\text{Re}\left(\log\left(1-\left(\cos(2\pi c_1)-2\sqrt{-1}\sqrt{\sin^2(\pi c_2)-\sin^4(\pi c_1)}\right)\left(\cos(2\pi c_1)+\sqrt{-1}\sin(2\pi c_1)\right)\right)\right)\\\nonumber
&=\text{Re}\left(\log\left(1-\cos^2(2\pi c_1)-2\sin(2\pi c_1)\sqrt{\sin^2(\pi c_2)-\sin^4(\pi c_2)}\right.\right. \\\nonumber
&\left.\left. -\sqrt{-1}\left(\sin(2\pi c_1)\cos(2\pi c_1)-2\cos(2\pi c_1)\sqrt{\sin^2(\pi c_2)-\sin^4(\pi c_2)}\right)\right)\right)\\\nonumber
&=\frac{1}{2}\log\left(\left(1-\cos^2(2\pi c_1)-2\sin(2\pi c_1)\sqrt{\sin^2(\pi c_2)-\sin^4(\pi c_2)}\right)^2\right. \\\nonumber
&\left. +\left(\sin(2\pi c_1)\cos(2\pi c_1)-2\cos(2\pi c_1)\sqrt{\sin^2(\pi c_2)-\sin^4(\pi c_2)}\right)^2\right)\\\nonumber
&=\frac{1}{2}\log\left(\sin^2(2\pi c_1)-4\sin(2\pi c_1)\sqrt{\sin^2(\pi c_3)-\sin^4(\pi c_1)}+4(\sin^2(\pi c_3)-\sin^4(\pi c_1))\right)\\\nonumber
&=\log\left(2\sqrt{\sin^2(\pi c_3)-\sin^4(\pi c_1)}-\sin(2\pi c_1)\right).
\end{align}

The second formula in Lemma \ref{lemma-c1c3} can be computed similarly. 
\end{proof}

\begin{align} \label{formula-dReVdc3}
\frac{d\text{Re}\hat{V}}{dc_3}&=\text{Re}\left(\frac{\partial \hat{V}}{\partial \theta_3}\right)\\\nonumber
&=\text{Re}\left(-\log(1-e^{2\pi\sqrt{-1}(T_2(c_1,c_3)+c_3)})+\log(1-e^{2\pi\sqrt{-1}(T_2(c_1,c_3)-c_3)})\right)\\\nonumber
&=\frac{1}{2}\log\left(\frac{\cos^2(\pi c_1)+\cos(\pi c_2)\sqrt{1-\frac{\sin^4(\pi c_1)}{\sin^2(\pi c_2)}}}{\cos^2(\pi c_1)-\cos(\pi c_2)\sqrt{1-\frac{\sin^4(\pi c_1)}{\sin^2(\pi c_2)}}}\right)
\end{align}
where in the last $=$ we have used the Lemma \ref{lemma-c1c3-2}.

\begin{lemma} \label{lemma-c1c3-2}
For $(c_1,c_3)\in [0,\frac{1}{4}]\times [\frac{1}{2},\frac{3}{4}]$, we have
 \begin{align}
 &\text{Re}\left(\log(1-e^{2\pi\sqrt{-1}(T_2(c_1,c_3)+c_3)})\right)\\\nonumber
 &=\frac{1}{2}\log\left(8\sin^2(\pi c_3)\cos^2(\pi c_1)-8\sin^2(\pi c_3)\cos(\pi c_3)\sqrt{1-\frac{\sin^4(\pi c_1)}{\sin^2(\pi c_2)}}\right)    
 \end{align}   
 and 
 \begin{align}
 &\text{Re}\left(\log(1-e^{2\pi\sqrt{-1}(T_2(c_1,c_3)-c_3)})\right)\\\nonumber
 &=\frac{1}{2}\log\left(8\sin^2(\pi c_3)\cos^2(\pi c_1)+8\sin^2(\pi c_3)\cos(\pi c_3)\sqrt{1-\frac{\sin^4(\pi c_1)}{\sin^2(\pi c_2)}}\right)    
 \end{align}    
\end{lemma}
\begin{proof}
Since $\sin^4(\pi c_1)-\sin^2(\pi c_3)\leq 0$ by $(c_1,c_3)\in [0,\frac{1}{4}]\times [\frac{1}{2},\frac{3}{4}]$, 
By a straightforward computation, 
\begin{align}
&\text{Re}\left(\log(1-e^{2\pi\sqrt{-1}(T_2(c_1,c_3)+c_3)})\right) \\\nonumber
&=\text{Re}\left(\log\left(1-\left(\cos(2\pi c_1)-2\sqrt{\sin^4(\pi c_1)-\sin^2(\pi c_3)}\right)e^{2\pi\sqrt{-1}c_3}\right)\right)\\\nonumber
&=\text{Re}\left(\log\left(1-\left(\cos(2\pi c_1)-2\sqrt{-1}\sqrt{\sin^2(\pi c_2)-\sin^4(\pi c_1)}\right)\left(\cos(2\pi c_3)+\sqrt{-1}\sin(2\pi c_3)\right)\right)\right)\\\nonumber
&=\text{Re}\left(\log\left(1-\cos(2\pi c_1)\cos(2\pi c_3)-2\sin(2\pi c_3)\sqrt{\sin^2(\pi c_3)-\sin^4(\pi c_1)}\right.\right. \\\nonumber
&\left.\left. -\sqrt{-1}\left(\sin(2\pi c_3)\cos(2\pi c_1)-2\cos(2\pi c_3)\sqrt{\sin^2(\pi c_3)-\sin^4(\pi c_1)}\right)\right)\right)\\\nonumber
&=\frac{1}{2}\log\left(\left(1-\cos(2\pi c_1)\cos(2\pi c_3)-2\sin(2\pi c_3)\sqrt{\sin^2(\pi c_3)-\sin^4(\pi c_1)}\right)^2\right. \\\nonumber
&\left. +\left(\sin(2\pi c_3)\cos(2\pi c_1)-2\cos(2\pi c_3)\sqrt{\sin^2(\pi c_3)-\sin^4(\pi c_1)}\right)^2\right)\\\nonumber
&=\frac{1}{2}\log\left(4(\cos(2\pi c_1)+1)\sin^2(\pi c_1)-4\sin(2\pi c_3)\sqrt{\sin^2(\pi c_3)-\sin^4(\pi c_1)}\right)\\\nonumber
&=\frac{1}{2}\log\left(8\sin^2(\pi c_3)\cos^2(\pi c_1)-8\sin^2(\pi c_3)\cos(\pi c_3)\sqrt{1-\frac{\sin^4(\pi c_1)}{\sin^2(\pi c_2)}}\right).
\end{align}
The formula (\ref{lemma-c1c3-2}) can be computed similarly. 
\end{proof}

\begin{proposition} \label{prop-Rec1c3decrease}
\begin{align}
   \text{Re}\hat{V}(c_1,T_2(c_1,c_3),c_3)\leq   \text{Re}\hat{V}(c_1,T_2(c_1,c_{30}),c_{30})   
\end{align}
for $c_{30}\leq c_3$
\begin{align}
   \text{Re}\hat{V}(c_1,T_2(c_1,c_3),c_3)\leq   \text{Re}\hat{V}(c_{10},T_2(c_{10},c_{3}),c_{3})   
\end{align}
for $c_{10}\leq c_{1}$. 
\end{proposition}
\begin{proof}
By formulas (\ref{formula-dReVdc1}) and (\ref{formula-dReVdc3}), we have 
\begin{align}
\frac{d\text{Re}\hat{V}}{dc_1}<0,  \ \text{and} \ \frac{d\text{Re}\hat{V}}{dc_3}<0.
\end{align}
\end{proof}

\begin{proposition} \label{prop-hessian-onedim}
When $\theta_2$ satisfies the conditions, 
     $\frac{1}{2}<\theta_2<1$, $0<c_1<\frac{1}{4}$, $\frac{1}{2}<c_3<\frac{3}{4}$ and $\theta_2+c_1<1$. 
then 
\begin{align}
f_{X_2X_2}>0.     
\end{align} 
\end{proposition}
\begin{proof}
    Recall that
\begin{align}
f(\theta_1,X_1,\theta_2,X_2,\theta_3,X_3)=\text{Re} \hat{V},    
\end{align}
then
\begin{align}
f_{X_2X_2}=-\text{Im}\left(\sqrt{-1}\hat{V}_{\theta_2\theta_2}\right)=2\pi (a+b+c+d+e)    
\end{align}
where 
\begin{align}
a&=-\text{Im}\frac{1}{1-z_2/z_1}=-\frac{\sin(2\pi(\theta_2-\theta_1))}{e^{2\pi(X_2-X_1)}+e^{-2\pi(X_2-X_1)}-2\cos(2\pi(\theta_2-\theta_1))} \\\nonumber
b&=-\text{Im}\frac{1}{1-z_2}=-\frac{\sin(2\pi \theta_2)}{e^{2\pi X_2}+e^{-2\pi X_2}-2\cos(2\pi \theta_2)}  \\\nonumber
c&=-\text{Im}\frac{1}{1-z_1z_2}=-\frac{\sin(2\pi( \theta_2+\theta_1))}{e^{2\pi (X_2+X_1)}+e^{-2\pi (X_2+X_1)}-2\cos(2\pi (\theta_2+\theta_1))} \\\nonumber
d&=\text{Im}\frac{1}{1-z_2z_3}=\frac{\sin(2\pi (\theta_2+\theta_3))}{e^{2\pi (X_2+X_3)}+e^{-2\pi (X_2+X_3)}-2\cos(2\pi (\theta_2+\theta_3))}\\\nonumber
e&=\text{Im}\frac{1}{1-z_2/z_3}=\frac{\sin(2\pi (\theta_2-\theta_3))}{e^{2\pi (X_2-X_3)}+e^{-2\pi (X_2-X_3)}-2\cos(2\pi (\theta_2-\theta_3))} 
\end{align}

we have
\begin{align}
&(d+e)|_{\theta_3=c_3,X_3=0}\\\nonumber
&=\frac{\sin(2\pi(\theta_2+c_3))}{e^{2\pi X_2}-2\cos(2\pi(\theta_2+c_3))+e^{-2\pi X_2}}+\frac{\sin(2\pi(\theta_2-c_3))}{e^{2\pi X_2}-2\cos(2\pi(\theta_2-c_3))+e^{-2\pi X_2}}\\\nonumber
&=2\sin(2\pi \theta_2)\frac{\cos(2\pi c_3)(e^{2\pi X_2}+e^{-2\pi X_2})-2\cos(2\pi \theta_2)}{(e^{2\pi X}-2\cos(2\pi(\theta_2+c_3))+e^{-2\pi X_2})(e^{2\pi X}-2\cos(2\pi(\theta_2-c_3))+e^{-2\pi X_2})}
\end{align}

For $\frac{1}{2}<c_3<\theta_2<1$, we have
\begin{align}
  \cos(2\pi c_3)(e^{2\pi X_2}+e^{-2\pi X_2})-2\cos(2\pi\theta_2)\leq 2\cos(2\pi c_3)-2\cos(2\pi \theta_2)<0  
\end{align}
and for $\frac{1}{2}<\theta_2<1$, 
$
    \sin(2\pi\theta_2)<0. 
$
so we have $(d+e)|_{\theta_3=c_3,X_3=0}>0$. 

As to $a+c$,
\begin{align}
    &(a+c)|_{\theta_1=c_1,X_1=0}\\\nonumber
    &=-\frac{\sin(2\pi(\theta_2-c_1))}{e^{2\pi X_2}-2\cos(2\pi(\theta_2-c_1))+e^{-2\pi X_2}}-\frac{\sin(2\pi(\theta_2+c_1))}{e^{2\pi X_2}-2\cos(2\pi(\theta_2+c_1))+e^{-2\pi X_2}}\\\nonumber
    &=-2\sin(2\pi \theta_2)\frac{\cos(2\pi c_1)(e^{2\pi X_2}+e^{-2\pi X_2})-2\cos(2\pi\theta_2)}{(e^{2\pi X_2}-2\cos(2\pi(\theta_2+c_1))+e^{-2\pi X_2})(e^{2\pi X_2}-2\cos(2\pi(\theta_2-c_1))+e^{-2\pi X_2})}
\end{align}

For $0\leq c_1\leq \frac{1}{4}$ and $\theta_2+c_1<1$, we have
\begin{align}
 \cos(2\pi c_1)(e^{2\pi X_2}+e^{-2\pi X_2})-2\cos(2\pi \theta_2)\geq 2\cos(2\pi c_1)-2\cos(2\pi \theta_2)>0   
\end{align}
and $\sin(2\pi \theta_2)<0$ for $\frac{1}{2}<\theta_2<1$, so we have $(a+c)|_{\theta_1=c_1,X_1=0}>0$. Hence 
\begin{align}
 f_{X_2X_2}(c_1,0,\theta_2,X_2,c_3,0)>0   
\end{align}
when $\frac{1}{2}<\theta_2<1$, $0<c_1<\frac{1}{4}$, $\frac{1}{2}<c_3<\frac{3}{4}$ and $\theta_2+c_1<1$.
\end{proof}

\begin{proposition} \label{prop-finfty-onedim}
Given $(c_1,c_3)\in [0,\frac{1}{4}]\times [\frac{1}{2},\frac{3}{4}]$, for $\frac{1}{2}<\theta_2<1$, we have
\begin{align}
f(c_1,0,\theta_2,X_2,c_3,0) \rightarrow  +\infty \ \text{as} \ X_2^2\rightarrow \infty.     
\end{align}
\end{proposition}
\begin{proof}
Based on Lemma \ref{lemma-Li2}, given $(c_1,c_3)\in [0,\frac{1}{4}]\times [\frac{1}{2},\frac{3}{4}]$, we introduce the following function for the $\theta_2\in (\frac{1}{2},1)$,     
\begin{equation} \label{eq:2}
F(X_2)=\left\{ \begin{aligned}
         &0  &  \ (\text{if} \ X_2\geq 0) \\
         &\left(\theta_2+c_3-\frac{3}{2}\right)X_2 & \ (\text{if} \ X_2<0)
                          \end{aligned} \right.
                          \end{equation}
\begin{equation*}
 +\left\{ \begin{aligned}
         &0  &  \ (\text{if} \ X_2\geq 0) \\
         &\left(\theta_2-c_3-\frac{1}{2}\right)X_2 & \ (\text{if} \ X_2<0)
                          \end{aligned} \right.
                          \end{equation*}
\begin{equation*}
-\left\{ \begin{aligned}
         &0  &  \ (\text{if} \ X_2\geq 0) \\
         &\left(\theta_2-\frac{1}{2}\right)X_2 & \ (\text{if} \ X_2<0)
                          \end{aligned} \right.  
                          \end{equation*}
\begin{equation*}
-\left\{ \begin{aligned}
         &0  &  \ (\text{if} \ X_2\geq 0) \\
         &\left(\theta_2+c_1-\frac{1}{2}\right)X_2 & \ (\text{if} \ X_2<0)
                          \end{aligned} \right.  
                          \end{equation*}  
\begin{equation*}
-\left\{ \begin{aligned}
         &0  &  \ (\text{if} \ X_2\geq 0) \\
         &\left(\theta_2-c_1-\frac{1}{2}\right)X_2 & \ (\text{if} \ X_2<0)
                          \end{aligned} \right.  
                          \end{equation*}                           

\begin{equation*}
   +X_2
\end{equation*}
i.e. 
\begin{equation} \label{eq:2}
F(X_2)=\left\{ \begin{aligned}
         &X_2  &  \ (\text{if} \ X_2\geq 0) \\
         &\left(\frac{1}{2}-\theta_2\right)X_2 & \ (\text{if} \ X_2<0)
                          \end{aligned} \right.
                          \end{equation}

It follows that, when $X_2^2\rightarrow \infty$, $F(X_2)\rightarrow +\infty$. 
\end{proof}

\end{document}